\documentclass[11pt]{amsart}
\usepackage[latin1]{inputenc}   
\usepackage[T1]{fontenc}
\usepackage{amsmath}
\usepackage{amsthm}
\usepackage{thmtools}
\usepackage{amssymb}
\usepackage{amscd}
\usepackage{enumitem}
\usepackage{bm}
\usepackage{caption}
\usepackage[dvipsnames]{xcolor}
\usepackage[pdftex,colorlinks, linkcolor = MidnightBlue, citecolor = Mahogany, urlcolor = BrickRed, pagebackref]{hyperref}

\usepackage{todonotes}
\usepackage[letterpaper,left=3cm,right=2.75cm,top = 2.75cm, bottom = 2.75cm]{geometry}
\usepackage{geometry}

\usepackage{setspace}
\usepackage{tikz}
\def\Im{\operatorname{Im}}
\def\Ker{\operatorname{Ker}}

\def\dim{\operatorname{dim}}

\def\Fuk{\operatorname{Fuk}}
\def\DFuk{\operatorname{DFuk}}
\def\Gimm{\operatorname{\Omega_{\text{cob}}^{\text{imm}}} \left (S_g \right)}
\def\Gimmunob{\operatorname{\Omega_{\text{cob}}^{\text{imm,unob}}} \left (S_g \right)}
\def\Diff{\operatorname{Diff}}

\def\Int{\operatorname{Int}}
\def\Id{\operatorname{Id}}

\def\Ind{\operatorname{Ind}}

\def\Hom{\operatorname{Hom}}
\def\deg{\operatorname{deg}}
\def\Mod{\operatorname{Mod}}
\def\Sign{\operatorname{Sign}}
\def\Cone{\operatorname{Cone}}
\def\Hol{\operatorname{Hol}}
\def\Spann{\operatorname{Span}}

\newcommand\norm[1]{\left \vert #1 \right \vert}
\newcommand\ens[2]{\left \{ #1 \middle \vert #2 \right \}}

\newcommand \Z{\mathbb{Z}}
\newcommand \R {\mathbb{R}}
\newcommand \N {\mathbb{N}}
\newcommand \C {\mathbb{C}}
\newcommand \D {\mathbb{D}}
\newcommand \HP {\mathbb{H}}

\theoremstyle{plain}
\newtheorem{Theo}{Theorem}[section]
\newtheorem{prop}[Theo]{Proposition}
\newtheorem{lemma}[Theo]{Lemma}
\newtheorem{coro}[Theo]{Corollary}
\newtheorem{defi}[Theo]{Definition}

\theoremstyle{definition}
\newtheorem{rk}[Theo]{Remark}

\let\originalleft\left
\let\originalright\right
\def\left#1{\mathopen{}\originalleft#1}
\def\right#1{\originalright#1\mathclose{}}

\author{Alexandre Perrier}
\email{perrier@dms.umontreal.ca}
\title{Lagrangian cobordism groups of higher genus surfaces}

\begin{document}

\begin{abstract}
We study Lagrangian cobordism groups of oriented surfaces of genus greater than two. We compute the immersed oriented Lagrangian cobordism group of these surfaces. We show that a variant of this group, with relations given by \emph{unobstructed} immersed Lagrangian cobordisms computes the Grothendieck group of the derived Fukaya category. The proof relies on an argument of Abouzaid \cite{ab07}.  	
\end{abstract}

\maketitle

\tableofcontents

\section{Introduction}
\label{section:Introduction}
\subsection{Immersed Lagrangians and cobordisms}

In this paper, we consider a (Riemann) surface $S_g$ of genus $g \geqslant 1$ equipped with an area form $\omega$.

We recall the definition of a Lagrangian cobordism following Biran-Cornea (\cite{BC13}).

\begin{defi}
\label{defi:Lagrangiancobordism}
Let $\gamma_0, \ldots, \gamma_N : S^1 \looparrowright S_g$ and $\tilde{\gamma}_0, \ldots \tilde{\gamma}_M : S^1 \looparrowright S_g$ be immersed curves. Let $F : V \looparrowright \C \times S_g$ be a Lagrangian immersion.

We say that $F$ is an \emph{immersed Lagrangian cobordism} from $\gamma_1, \ldots, \gamma_N$ to $\tilde{\gamma}_1, \ldots \tilde{\gamma}_M$ if

\begin{enumerate}[label=(\roman*)]
	\item there is $\varepsilon > 0$ such that outside $[- \varepsilon, \varepsilon ] \times \R$, $F$ is an embedding with image 
	\[ \coprod_{i= 1 \ldots N} (-\infty, - \varepsilon] \times \gamma_i \cup \coprod_{j = 1 \ldots M} [\varepsilon, \infty) \times \tilde{\gamma}_j, \]
	\item the set $F^{-1} \left ([- \varepsilon, \varepsilon] \times \R \right )$ is compact.	
\end{enumerate}
Such a cobordism will be denoted by $V : \left (\gamma_1, \ldots, \gamma_N \right ) \leadsto \left ( \tilde{\gamma}_1, \ldots, \tilde{\gamma}_N \right) $.
\end{defi}

\begin{rk}
\begin{enumerate}
	\item If $V$ is oriented and its orientation agrees with the natural orientations of $(-\infty, - \varepsilon] \times \gamma_i$ and $[\varepsilon, \infty) \times \tilde{\gamma}_j$, we say that $V$ is an \emph{oriented} immersed Lagrangian cobordism.
	\item When the curves $\gamma_i$, $\tilde{\gamma}_i$ and the surface $F$ are embedded, we will say that $V$ is a \emph{Lagrangian cobordism}.
	\item The definition goes back to Arnold (\cite{Ar80}). The reader should be aware that the definition in \cite{Ar80} is slightly different from ours although equivalent (see Lemma \ref{lemma:ArnoldcobordantequalscobordantBC}). 
\end{enumerate}
\end{rk}

Now, consider the set $\mathcal{L}_{Imm}$ of Lagrangian immersions from an arbitrary number of copies of $S^1$ to $S_g$. Define an equivalence relation $\sim$ on $\mathcal{L}_{Imm}$ by 
\[ \gamma_1 \sim \gamma_2 \]
if and only if there is an immersed Lagrangian cobordism from $\gamma_1$ to $\gamma_2$.	

Here $\gamma_1$ and $\gamma_2$ are two elements of $\mathcal{L}_{Imm}$.

\begin{defi}
\label{defi:immersedcobgroup}
The \emph{immersed Lagrangian cobordism group} of $S_g$ is the quotient
\[ \mathcal{L}_{Imm} / \sim . \]
We will denote it by $\Gimm$
\end{defi}

The set $\Gimm$ an abelian group whose sum is given by disjoint union. The neutral element is the void set. The inverse of a generator $\gamma : S^1 \to S_g$ is the curve $\gamma^{-1}$ obtained by reversing the orientation of $\gamma$. 

The following Lemma shows that this group effectively detects the cobordism class of a curve.

\begin{lemma}
Let $\gamma_1, \ldots, \gamma_n : S^1 \looparrowright S_g$ be immersed curve in $S_g$. Their classes in $\Gimm$ satisfy
\[ [\gamma_1] + \ldots + [\gamma_n] = 0 \]
if and only if there is an oriented immersed Lagrangian cobordism
\[ V: (\gamma_1, \ldots, \gamma_n ) \leadsto \emptyset. \]
\end{lemma}

Due to Gromov's h-principle for Lagrangian immersions, topological invariants determine the Lagrangian cobordism group of the surface $S_g$.

\begin{Theo}
\label{Theo:immersedcobordismgroup}
We denote by $\chi(S_g)$ the Euler characteristic of $S_g$. There is an isomorphism
\[ \Gimm \to H_1 \left (S_g, \Z \right ) \oplus \Z / \chi(S_g) \Z. \]	
\end{Theo}

Here, the map $\Gimm \to H_1 \left (S_g, \Z \right )$ is the homology class. Meanwhile, the map $\Gimm \to \Z / \chi(S_g)$ is a variant of a topological index defined by Chillingworth (see \cite{Chi72}, see also \cite[Appendix A]{ab07}). Along the way, we give an alternate definition of this index in line with the usual definition of the Maslov index in symplectic topology.

\begin{rk}
We can find many computations of Lagrangian cobordism groups in the literature. In \cite{Ar80}, Arnold computed the Lagrangian cobordism groups of $\R^2$ and of the cotangent bundle $T^{*} S^1$. 

Eliashberg showed in \cite{E84} that some of these groups are isomorphic to fundamental groups of some Thom spaces. Audin used these results to compute the generators of some other cobordism groups (\cite{A85} and \cite{A87}).
\end{rk}

\subsection{Floer theory and cobordism groups}
Let $L$ be a Lagrangian submanifold of a symplectic manifold $(M,\omega)$. The Maslov index induces a morphism 
\[ \mu_L : \pi_2(M,L) \to \Z. \] 
On the other hand, symplectic area induces a morphism 
\[\omega : \pi_2(M,L) \to \R. \] 
The lagrangian submanifold $L$ is \emph{monotone} if there is $\lambda > 0$ such that 
\[ \omega_L = \lambda \mu_L. \]

In this case, there is a well-defined Fukaya category $\Fuk(M,\omega)$ whose objects are monotone Lagrangian submanifolds satisfying a topological condition (see \cite{BC14} or \cite{She16}). The $A_\infty$-category $Fuk(M,\omega)$ has a \emph{derived category} $DFuk(M,\omega)$ (defined in \cite{Sei08}). Note that this category is \emph{not} the split-completion of $\DFuk(M,\omega)$. The category $\DFuk(M,\omega)$ is triangulated, so one can speak of its \emph{Grothendieck group} 
\[K_0(\DFuk(M,\omega)). \] 
Recall that this is the abelian group generated by the objects of $\DFuk(M,\omega)$ with relations given by 
\[ Y = Z + X \] 
whenever there is an exact triangle
\[ X \to Y \to Z \to X[1]. \]
Now, Biran and Cornea proved that there is a natural surjective group morphism (\cite[Corollary 1.2.1]{BC14})
\begin{equation}
 \Theta_{BC} : \Omega_{\text{cob}}^{\text{emb}} (M,\omega) \twoheadrightarrow K_0(\DFuk(M,\omega)).
 \label{eqn:Birancorneamap}	
\end{equation}

There are several results on this map. In \cite{Hau15}, Haug shows that the map \ref{eqn:Birancorneamap} is an isomorphism when $(M,\omega)$ is a torus of dimension $2$ and the Lagrangians are equipped with local systems. 

In \cite{He17}, Hensel gives algebraic conditions under which the map \ref{eqn:Birancorneamap} is an isomorphism. These are, in particular, verified for the torus.

More recently, in \cite{SS18}, Sheridan and Smith use Mirror symmetry to prove the existence of certain Maslov $0$ Lagrangian tori in $K3$ surfaces. In \cite{SS18-2}, they study Lagrangian cobordism group in Lagrangian torus fibrations over tropical affine manifolds.
 
The main Theorem of our paper is a generalization of Haug's result to surfaces of genus $g \geqslant 2$. For this, we use a version of the cobordism group which takes a broader class of cobordisms into account.

\begin{Theo}
\label{theo:theBiranCorneamapisanisomorphism}
Assume that the genus $g$ of $S_g$ is greater or equal than $2$. There is a natural isomorphism
\[ \Gimmunob \to K_0(\DFuk(S_g)). \]
\end{Theo}

In \cite{ab07}, Abouzaid showed that $K_0(\DFuk(S_g))$ is isomorphic to $ \R \oplus H_1(S_g,\Z) \oplus \Z / \chi(S_g) \Z$. Therefore we have the following

\begin{coro}
\label{coro:compututionofGimmunob}
There is an isomorphism
	\[ \Gimmunob \to \R \oplus H_1(S_g,\Z) \oplus \Z / \chi(S_g) \Z .\]	
\end{coro}

 \emph{Unobstructed Lagrangian cobordisms} give the relations of $\Gimmunob$. These are immersed Lagrangian cobordism which satisfy a technical condition. We postpone the actual definition to section \ref{section:immersedcobordismandcones}.

We shall consider a variant of the Fukaya category whose objects are defined below.

\begin{defi}
\label{defi:unobimmersion}
An immersion $\gamma : S^1 \looparrowright S_g$ is \emph{unobstructed} if it satisfies the following assumptions. 
\begin{enumerate}[label=(\roman*)]
	\item It has no triple points and all its double points are transverse.
	\item Let $\tilde{S}_g$ be the universal cover of $S_g$, the curve $\gamma$ lifts to a curve $\tilde{\gamma} : \R \to \tilde{S}_g$. We assume that $\tilde{\gamma}$ is properly embedded.   	
\end{enumerate}
\end{defi}

\begin{rk}
\label{rk:definitiongeneric}
When $(i)$ holds, we say that $\gamma$ is \emph{generic}.	
\end{rk}

\begin{figure}
\captionsetup{justification=centering,margin=2cm}

\begin{tikzpicture}[y=0.80pt, x=0.80pt, yscale=-1.000000, xscale=1.000000, inner sep=0pt, outer sep=0pt]
  \path[draw=black,line join=miter,line cap=butt,even odd rule,line width=1.053pt]
    (360.8860,23.9490) .. controls (322.2881,23.7390) and (318.6022,60.6645) ..
    (318.6022,79.6718) .. controls (318.6022,98.6792) and (387.8807,206.8456) ..
    (415.9533,206.8456) .. controls (444.0259,206.8456) and (456.2705,189.1844) ..
    (456.2705,169.0979) .. controls (456.2705,149.0113) and (403.7437,98.6982) ..
    (361.8694,99.4444) .. controls (319.9950,100.1906) and (277.3017,156.0886) ..
    (277.3017,170.8954) .. controls (277.3017,170.8954) and (276.3960,205.8672) ..
    (308.7687,206.8456) .. controls (341.1414,207.8240) and (407.1032,99.8850) ..
    (407.1032,79.6718) .. controls (407.1032,59.4586) and (399.4840,24.1591) ..
    (360.8860,23.9490) -- cycle;
  \path[draw=black,fill=black,line join=round,line cap=round,miter limit=4.00,fill
    opacity=0.156,line width=0.286pt] (327.8405,101.7729) .. controls
    (326.2362,98.7283) and (323.4125,93.1576) .. (322.0725,89.3943) --
    (319.6679,82.6406) -- (318.5376,73.0550) .. controls (318.0206,68.6702) and
    (323.4340,27.8694) .. (351.0536,24.6448) .. controls (361.2515,23.4542) and
    (374.6370,24.6856) .. (383.0354,29.1558) .. controls (388.4700,32.0485) and
    (396.2877,40.1791) .. (399.3989,46.1745) .. controls (404.2174,55.4596) and
    (409.0625,73.2758) .. (407.8906,81.1817) .. controls (407.2042,85.8120) and
    (399.2537,104.8628) .. (397.1309,107.2121) .. controls (397.0014,107.3555) and
    (395.0770,107.4749) .. (392.1998,106.2771) .. controls (381.7767,101.9379) and
    (376.0060,99.5364) .. (364.5893,99.5838) .. controls (355.2233,99.6226) and
    (353.8297,99.9972) .. (348.3393,101.8105) .. controls (345.0000,102.9134) and
    (340.6074,104.8316) .. (338.3596,105.9640) .. controls (336.1119,107.0964) and
    (332.1640,108.0658) .. (332.0687,108.3801) .. controls (331.4724,110.3443) and
    (329.4447,104.8176) .. (327.8405,101.7729) -- cycle;
  \path[draw=black,fill=black,line join=round,line cap=round,miter limit=4.00,fill
    opacity=0.156,line width=0.286pt] (359.2953,153.7375) .. controls
    (354.2880,146.9454) and (342.0655,128.7811) .. (337.5547,121.1479) .. controls
    (335.4652,117.6122) and (333.3334,113.4136) .. (332.5713,112.1617) --
    (330.8512,109.3362) -- (334.9819,106.5323) .. controls (339.2474,103.6370) and
    (344.9770,101.3648) .. (351.2857,100.8447) .. controls (355.0073,100.5379) and
    (357.9797,99.1080) .. (365.2143,99.3830) .. controls (375.0386,99.7565) and
    (383.3684,101.0387) .. (391.2638,106.1972) -- (396.0439,108.1739) --
    (392.6420,114.7055) .. controls (386.2961,126.8898) and (365.6770,158.5164) ..
    (363.8454,159.6484) .. controls (363.5434,159.8350) and (361.2503,156.3894) ..
    (359.2953,153.7375) -- cycle;

\end{tikzpicture}

\caption{An obstructed immersed curve and a teardrop (shaded)}
\label{figure:teardrop}

\end{figure}
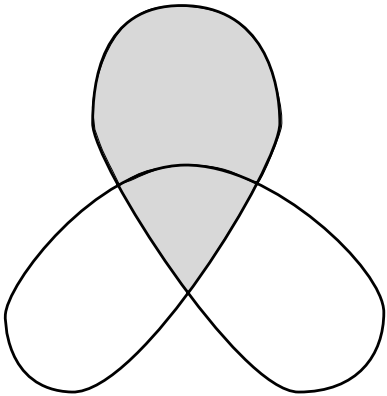

At this point, there is only one thing the reader needs to keep in mind. Unobstructed objects do not bound \emph{teardrops} which are polygons with a unique corner (see figure \ref{figure:teardrop}). It is indeed well-known that these give an obstruction to the definition of Floer theory of immersed objects. See the work of Akaho and Joyce (\cite{AJ10}), Abouzaid (\cite{ab07}), Alston and Bao (\cite{AB14}).

\subsection{Relation with \cite{Hau15}}
In \cite{Hau15}, Haug actually showed that there is a (split) exact sequence
\begin{equation} 
0 \to \R / \Z \xrightarrow[]i{} \Omega_{\text{cob}}^{\text{emb}} (M,\omega) \to H_1(T^2,\Z) \to 0. 
\label{eqn:exactsequenceforthetorus}
\end{equation}
The proof requires \emph{Mirror symmetry}. More precisely, using geometric arguments, Haug proves that the kernel of $\Omega_{\text{cob}}^{\text{emb}} (M,\omega) \to H_1(T^2,\Z)$ is the image of $i$. On the other hand, Mirror Symmetry for the torus yields an equivalence $\DFuk(T^2) \simeq D^b(X)$ between the derived Fukaya category of the torus, whose objects are curves equipped with local systems, and the bounded derived category of Coherent sheaves of an elliptic curve $X$ over the Novikov field $\Lambda$. Haug uses this to show that the application $i$ is injective.

In our paper, we show that there is an analog of the exact sequence \ref{eqn:exactsequenceforthetorus} for the group $\Gimmunob$ (see Theorem \ref{Theo:Computationofunobgroup} for the precise statement). However, the main difference is as follows. We \emph{do not use} Mirror symmetry for the proof. Moreover, we do not take local systems into account. Therefore, our main result is purely geometric.

This is in contrast with all the results above which use ideas coming from mirror symmetry to study Lagrangian cobordism groups.

\subsection{Outline of the paper}
 The proofs of both Theorems \ref{Theo:immersedcobordismgroup} and \ref{Theo:Computationofunobgroup} use the action of the Mapping Class Group of $S_g$ to find generators of $\Gimm$ and $\Gimmunob$. This idea is due to Abouzaid \cite{ab07}.
 
 In the first section, we study the immersed Lagrangian cobordism group $\Gimm$ and give the proof of Theorem \ref{Theo:immersedcobordismgroup}. Most of the results and definitions are not new (some even date back to Arnold). However, we tried to give details which we did not find in the literature.
 
In the second section, we define the Fukaya category of unobstructed curves following Seidel's book \cite{Sei08} and Alston and Bao's paper \cite{AB14}. We also give a combinatorial description of this category.

In the third section, we give the definition of an unobstructed Lagrangian cobordism. We explain why Biran-Cornea's map \ref{eqn:Birancorneamap} extends to this setting.
 
In the fourth section, we prove Theorem \ref{theo:theBiranCorneamapisanisomorphism} and \ref{coro:compututionofGimmunob}. To do this, we describe the action of the Mapping Class Group on $\Gimmunob$ using unobstructed Lagrangian cobordisms.
 
 \subsection{Acknowledgements} This work is part of my doctoral thesis at Université de Montréal under the direction of Octav Cornea. I wish to thank him for proposing me this project and his very thoughtful guidance over these years!
 
 I also wish to thank Jordan Payette for explaining the proof of Lemma \ref{lemma:holonomyequalshamiltonianinvariant} to me as well as Jean-Philippe Chassé for reading the text and pointing out some mistakes. 
 
\section{Computation of the immersed cobordism group}

In this section, we give the proof of Theorem \ref{Theo:immersedcobordismgroup}. In the first subsection, we show that embedded curves generate $\Gimm$. In the second subsection, we define a map 
\[ \Gimm \to H_1(S_g,\Z) \oplus \Z / \chi(S_g)\Z . \]
We check that it is well-defined and surjective. At last, we modify a geometric argument of Abouzaid (\cite{ab07}) to show that this map is injective.

\subsection{Properties of the immersed cobordism group}

\subsubsection{Lagrangian cobordism and isotopy}
\label{subsubsection:isotopyandcobordism}
We will use the following Lemma.
\begin{lemma}
\label{lemma:isotopiccurvesareimmlagcob}
Assume that $\gamma_-: S^1 \looparrowright S_g$ and $\gamma_+ : S^1 \looparrowright S_g$ are two isotopic immersed curves, then there is an immersed Lagrangian cobordism from $\gamma_-$ to $\gamma_+$.
\end{lemma}

\begin{proof}
Choose an isotopy $(\gamma_s(t))_{s \in \R}$ such that 
\[\left\{ \begin{array}{c}
	\gamma_s(t) = \gamma_-(t) \text{ for s <  0 } \\
	\gamma_s(t) = \gamma_+(t) \text{ for s > 1 }
\end{array}
\right.
\]

Now, consider the following immersion
\[ f : \begin{array}{ccc}
 \R \times S^1 & \mapsto & \C \times S_g \\
 (s,t) & \mapsto & \left (s, \gamma_s(t) \right)	.
 \end{array}
\]
This map is covered by an isotropic bundle map 
\[ F : T\R \times TS^1 \to T \C \times T S_g \]
defined by
\[ F \left(\partial_s \right) = (1,0), \ F \left(\partial_t \right)= \left(0, \frac{d \gamma_s(t)}{dt}  \right). \] 
Moreover, since $H^2(\R \times S^1,\R) = 0$, we have $f^*(dx \wedge dy +  \omega) = 0$ in $H^2(\R \times S^1,\R)$.

We can now apply \cite[16.3.2]{EM02}	 to find an immersion $\tilde{f}: S^1 \times \R \to \C \times S_g$ whose ends coincide with $f$. The map $\tilde{f}$ is the relevant Lagrangian cobordism.
\end{proof}

\subsubsection{Resolution of double points}
\label{subsubsection:surgeryofdoublepoints}
We will use a variant of the Weinstein neighborhood Theorem for Lagrangian immersions constantly throughout this section. 

Recall that there is a canonical identification between the fiber bundle $\pi : T^* S^1 \to S^1$ and the product bundle $S^1 \times \R \to S^1$. We denote by $T^*_\varepsilon S^1$ the set $\ens{(q,p) \in T^* S^1}{\norm{p} < \varepsilon} $. 

\begin{lemma}
\label{lemma:Weinsteinnbhdthmforimmersions}
Let $\gamma : S^1 \to S_g$ be an immersed curve. There are $\varepsilon > 0$ and a local embedding
\[ \psi : T^*_\varepsilon S^1 \to S_g \]	
such that
\begin{itemize}
	\item $\psi$ restricted to the zero section is equal to $\gamma$,
	\item $\psi^* \omega$ is the standard symplectic form $\omega_{\text{std}}$ on $T^*S^1$. 	
\end{itemize}

\end{lemma}
The proof is a simple exercise.

\begin{lemma}
\label{lemma:immersedcurvescobordanttogenericcurves}
Let $\gamma : S^1 \looparrowright S_g$ be an immersed curve. Then $\gamma$ is Lagrangian cobordant to a generic\footnote{See \ref{rk:definitiongeneric} for the definition} immersed curve $\tilde{\gamma} : S^1 \looparrowright S_g$.	
\end{lemma}

\begin{rk}
\label{rk:Lagrangiansuspension}
The proof uses a variant of the \emph{Lagrangian suspension} (\cite[2.1.2]{ALP92}). If $\left (\phi^t_H \right)_{t \in [0,1]}$ is a Hamiltonian isotopy of $S_g$ and $\gamma$ an immersed curve, then the immersion
\[ (t,x) \in [0,1] \times S^1 \mapsto \left (t,H_t \circ\phi^t_H \circ \gamma(x), \phi^t_H \circ \gamma (x) \right) \]
is a Lagrangian cobordism between $\gamma$ et $\phi^1_H (\gamma)$. Notice that this cobordism is embedded if $\gamma$ is.	
\end{rk}

\begin{proof}
We call $\psi$ the local embedding given by Lemma \ref{lemma:Weinsteinnbhdthmforimmersions}. We let $x \in S^1$. Choose a disk neighborhood $U_x \subset S^1$ containing $x$ such that $\psi_{\vert \pi^{-1}(U_x)}$ is an embedding which we denote by $\Phi$. Moreover we let $U = \gamma^{-1} \left (\psi(U_x) \right ) \backslash U_x$

Let $\eta > 0$. We claim that there is a function $f_x : S^1 \to \R$ such that the following holds.
\begin{itemize}[label= $\bullet$]
	\item The derivatives of $f_x$ satisfy $\norm{f_x'} < \eta$ and $\norm{f_x''} < \eta$.
	\item Denote by $\gamma_x $ the immersion $ t \mapsto \psi (t,-f_x'(t))$. If $\gamma_x(t_1) = \gamma_x(t_2)$ with $t_1 \in U_x$, then $\gamma_x'(t_1)$ and $\gamma_x'(t_2)$ are transverse. 
\end{itemize}
 
The proof is an application of Sard's Theorem. Consider the map $F = \pi_\R \circ \Phi^{-1} \circ \gamma_{\vert U}$. For every $\alpha >0$, there is a regular value $z$ of $F$ satisfying $\norm{z} < \alpha$. Now we let $f_x$ be a smooth function such that $-f'$ is constant equal to a regular value over $U_x$ and $\norm{f_x'}, \norm{f_x''} < \eta$. The reader may easily check that this is the desired function.

The curves $\gamma$ and $\gamma_x$ are cobordant. Indeed, choose a smooth cutoff function $\beta : \R \to \R$ such that $\beta(t) = 0$ for $t \leqslant 0$ and $\beta(t) = 1$ for $t \geqslant 1$. The relevant cobordism is the image of the map 
\[ \begin{array}{ccc} 
 	\R \times S^1 & \to & \C \times S_g \\
 	(t,x) & \mapsto & \left (t, \beta(t) f(x), \psi(x,- \beta(t)f'(x)) \right )	
 \end{array}.
\]

Now choose $x_1, \ldots, x_N$ such that $U_{x_1}, \ldots, U_{x_N}$ is a covering of $S^1$. We use the construction above iteratively to get an immersed curve which is cobordant to $\gamma$ and with transverse double points.
\end{proof}

We can solve any double point of a generic immersion through a cobordism (\cite[page 9]{Ar80}, see also \cite[1.4]{ALP92}). For the convenience of the reader, we will summarize the proof of this fact and fill in some details.

First, we recall the (standard) procedure for solving the double point of a generic immersion. This is a particular instance of the Lagrangian surgery (see \cite{LS91} and \cite{Pol91}). Let $\gamma : S^1 \looparrowright S_g$ be a generic immersed curve and $x = \gamma (p) = \gamma(q)$ a double point with $p \neq q \in S^1$. There are 
\begin{itemize}
\item an open neighborhood $U$ of $x$, 
\item a real number $r >0$,
\item a symplectomorphism $\phi : U \to B(0,r) \subset \C$ 
\end{itemize}
such that $\phi \circ \gamma$ parameterizes the real axis near $p$ and parameterizes the imaginary axis near $q$.

Pick a smooth path $c : \R \to \C$ such that
\begin{enumerate}
	\item for $t \leqslant - 1$, $c(t) = t$,
	\item for $t \geqslant 1$, $c(t) = it$,
	\item the derivatives $x'$ and $y'$ satisfy $x'> 0$ and $y' > 0$,
	\item for all $t \in \R$, $(x(-t),y(-t)) =(- y(t), -x(t) )$.
\end{enumerate}

The \emph{surgery} of $\gamma$ at the point $(p,q)$ with parameter $\varepsilon > 0$ is the curve obtained by replacing the image of $\gamma$ by the images of the curves $\varepsilon c$ and $- \varepsilon c$ in the chart $\phi$. We denote it by $\gamma_{(p,q),\phi,\varepsilon}$. Note that this curve depends on the \emph{ordered} pair $(p,q)$.

Notice in particular that all surgeries at a given double point are isotopic to one another, hence Lagrangian cobordant by Lemma \ref{lemma:immersedcurvescobordanttogenericcurves}.

We shall prove the following Proposition.

\begin{prop}
\label{prop:resolutionofdoublepointthroughimmersedcob}
Let $\gamma : S^1 \looparrowright S_g$ be a generic immersed curve and $x = \gamma (p) = \gamma(q)$ a double point with $p \neq q \in S^1$.

There are a chart $\phi$ and a real number $\varepsilon > 0$ such that $\gamma$ is cobordant to $\gamma_{(p,q),\phi,\varepsilon}$.
\end{prop}

\begin{proof}
First, we need the following.

\begin{lemma}
\label{lemma:ArnoldcobordantequalscobordantBC}
Let $\Sigma$ a compact surface with boundary and let 
\[ F : (\Sigma, \partial \Sigma) \looparrowright ([-1,1] \times \R \times S_g, \partial [-1,1] \times \R \times S_g) \]
be a Lagrangian immersion transverse to $\partial [-1,1] \times \R \times \C$ along $\partial \Sigma$.

Then, the projection of $F_{\vert \partial \Sigma}$ to $S_g$ is the union of two immersions $\gamma_-$ and $\gamma_+$ lying over $\{-1\} \times \R \times S_g$ and $\{1\} \times \R \times S_g$ respectively.

There is an immersed Lagrangian cobordism from $\gamma_-$ to $\gamma_+$.
\end{lemma}

\begin{rk}
\label{rk:DefArnoldcobordant}
Such immersions are what Arnold called Lagrangian cobordisms in his original paper(\cite{Ar80}). Therefore, we will call these objects \emph{Lagrangian cobordisms in Arnold's sense}.
\end{rk}

\begin{proof}[Proof of Lemma \ref{lemma:ArnoldcobordantequalscobordantBC}]
Denote by $\partial^+ \Sigma$ the union of connected components of $\partial \Sigma$ which projects to $\{1\} \times \R$ in the $\C$ factor. We identify $\partial^+ \Sigma$ with a disjoint union of copies of $S^1$ : $\partial^+ \Sigma = \sqcup_{i=1 \ldots N} S^1$. On $\partial^+ \Sigma$, $F$ is of the form $t \mapsto \left(1,f(t),\gamma_+ (t) \right)$ with $f : \partial^+ \Sigma \to \R$ a smooth function.

By Lemma \ref{lemma:Weinsteinnbhdthmforimmersions}, we can extend $\gamma_+$ to a local symplectomorphism 
\[ \psi : \coprod_{i=1 \ldots N} S^1 \times (-\varepsilon,\varepsilon) \to S_g .\]
We choose a smooth function $\tilde{f}$ with
\[ \tilde{f}: \begin{array}{ccc}
 \coprod_{i=1 \ldots N} S^1 \times (-\varepsilon,\varepsilon) & \to & S_g \\
 (s,t) & \mapsto & \left\{
 \begin{array}{c}
 0 \text{ if } \norm{t} > \frac{2 \varepsilon}{3}, \\
 f(s,t) if \norm{t} < \frac{\varepsilon}{3}
 \end{array}  
 \right.	
 \end{array}.
\]
Notice that its hamiltonian flow $ \left (\phi_{\tilde{f}}^t \right )_{t \in [0,1]}$ satisfies $\phi_{\tilde{f}}^t (S^1 \times (- \varepsilon, \varepsilon)) = S^1 \times (- \varepsilon, \varepsilon) $. 

We define a Lagrangian cobordism in Arnold sense as follows
\[ G : \begin{array}{ccc}
 		[0,1] \times S^1 & \to & \C \times S_g \\
 		(t,z) & \to & \left (t,\psi \circ \phi_{\tilde{tf}}^t (\gamma_+(z)) \right )	
 \end{array}. \]
Now, we consider the union of the maps $F$ and $G+(1,0)$. A Lagrangian smoothing of the resulting cobordism is the desired Lagrangian cobordism.
\end{proof}

We build a local model for the resolution of the double point. The quartic
	\[ \Sigma = \ens{(t,x,y) \in [-1,1] \times \R^2}{y^2-x^2+t = 0}. \]
is the set of critical points of the generating families
\[ f_{t,x} : \begin{array}{ccc}
 \R & \to & \R \\
 y & \mapsto & \frac{y^3}{3} - yx^2 + ty
 \end{array}.
\]
We rotate the Lagrangian immersion associated to this by an angle of $\frac{\pi}{4}$ to obtain
\[ F : \begin{array}{ccc} 
\Sigma & \to & \C \times \C \\
(t,x,y) & \mapsto & \left (t,y,\frac{x + 2xy}{ \sqrt{2} }, \frac{x-2xy}{\sqrt{2}} \right)
\end{array}. \]
Notice also that the map $F$ a Lagrangian cobordism in Arnold sense between a double point and its resolution.

We now modify $F$ so that it is equal to $\R \times \R \sqcup \R \times i\R$ outside a neighborhood of $\R \times \{ 0 \}$.

The set $\Im(F) \backslash \left ( [-1,1] \times B(0,1) \right )$ is an embedded manifold with four connected components. Two of them are contained in the quadrant 
\[ [-1,1] \times \ens{e^{i\theta}}{\theta \in [- \frac{\pi}{4}, \frac{\pi}{4}] \mod \pi }. \] 
We denote them by $L_1$. Two of them are contained in the quadrant 
\[[-1,1] \times \ens{e^{i\theta}}{\theta \in [- \frac{\pi}{4}, \frac{\pi}{4}] \mod \pi }.\] 
We denote them by $L_2$.

Notice that the linear projection $\pi_\R : [-1,1] \times \C \to \R \times \R$ which maps $(t,s,u,v)$ to $(t,u)$ restricts to a diffeomorphism from $L_1$ to the band $[-1,1] \times (\R \backslash [-1,1])$. From this, we deduce that $L_1$ is of the form
	\[ \ens{\left ( t,f(x,t),x,g(x,t) \right )}{(t,x) \in [-1,1] \times (\R \backslash [-1,1])}. \]
Since $L_1$ is Lagrangian, the form $f dt + g dx$ is closed. Furthermore, the set $[-1,1] \times (\R \backslash [-1,1])$ is homotopy equivalent to two points. Hence, there is a smooth function $h$ such that $f = \partial_t h$ and $g = \partial_t h$.

Now, choose a bump function $\beta : [-1,1] \times (\R \backslash [-1,1]) \to \R$ such that $\beta(t,x) = 1$ on $[-1,1] \times [- \frac{5}{4}, \frac{5}{4}]$ and $\beta(t,x) = 0$ outside $[-1,1] \times [-2,2]$. Define $\tilde{L}_1$ to be the following embedded Lagrangian 
\[ \tilde{L}_1 := \ens{\left ( t,\partial_t(\beta h),x, \partial_x(\beta h) \right )}{(t,x) \in [-1,1] \times (\R \backslash [-1,1])}. \]

We define $\tilde{L}_2$ in the same manner. The projection $\pi_{i\R} : [-1,1] \times \C \to \R \times \R$ which maps $(t,s,u,v)$ to $(t,v)$ restricts to a diffeomorphism $L_2 \to [-1,1] \times (\R \backslash [-1,1] )$. We deduce that $L_2$ is of the form 
\[ \ens{\left ( t,\partial_t h, \partial_y h, y \ \right)}{(t,y) \in [-1,1] \times (\R \backslash [-1,1])}. \]
and we put 
\[ \tilde{L}_2 := \ens{\left ( t,\partial_t(\beta h), \partial_y(\beta h),y \right )}{(t,y) \in [-1,1] \times (\R \backslash [-1,1])}. \]

Now, the map $F$ restricted to the set $\Sigma \cap [-1,1]^3$ and the two embedded submanifolds $\tilde{L_1}, \tilde{L_2}$ yield an immersion 
\[ H : \Sigma \to \C \times \C \]
equal to $[-1,1] \times \R \cup [-1,1] \times i\R$ outside $\Sigma \cap [-1,1]^3$.

Consider the immersion $i := [-1,1] \times \gamma$. Recall that we chose a chart $\phi : U \to B(0,r)$ around $x$. In the chart $\Id \times \phi$, the immersion $i$ is equal to $[-1,1] \times \R \cup [-1,1] \times i\R$. We replace this by $\varepsilon H$ for $\varepsilon$ small enough and smooth the resulting immersion. The result is a Lagrangian cobordism in Arnold sense (see Remark \ref{rk:DefArnoldcobordant}) between $\gamma$ and its surgery at $x$. By Lemma \ref{lemma:ArnoldcobordantequalscobordantBC}, we obtain that $\gamma$ and its surgery are cobordant.
\end{proof}

All of this allows us to deduce the following result due to Arnold (\cite{Ar80}).

\begin{prop}
The classes of Lagrangian embeddings generate the immersed Lagrangian cobordism group $\Gimm$.
\end{prop}

\begin{proof}
Let $\gamma$ be a Lagrangian immersion. By Lemma \ref{lemma:immersedcurvescobordanttogenericcurves}, $\gamma$ is Lagrangian cobordant to a generic curve $\tilde{\gamma}$. Repeated applications of Lemma \ref{prop:resolutionofdoublepointthroughimmersedcob} show that $\tilde{\gamma}$ is Lagrangian cobordant to an embedding.
\end{proof}

\subsubsection{Resolution of intersection points}
\label{subsubsection:surgeryofintersectionpoints}
Let $\gamma_1 : S^1 \to S_g$ and $\gamma_2 : S^1 \to S_g$ be two transverse generic immersed curves. Let $x = \gamma_1 (p) = \gamma_2 (q)$ be an intersection point of $\gamma_1$ and $\gamma_2$. We can perform the Lagrangian surgery of $\gamma_1$ and $\gamma_2$ at $x$ (as defined in \cite{Pol91} and \cite{LS91}) to obtain a curve $\gamma_1 \#_x \gamma_2$.

It is an observation of Biran and Cornea that in the embedded case, the curves $\gamma_1$ and $\gamma_2$ are cobordant to their surgery $\gamma_1 \#_{x,\varepsilon} \gamma_2$ (\cite[Lemma 6.1.1]{BC13} ). We explain how to adapt their argument to the case of oriented immersed curves.

We say that $x$ is of degree $1 \in \Z / 2 \Z$ if the oriented basis $(\gamma_1'(p), \gamma_2'(q))$ is positive with respect to the orientation of $T_x S_g$ and that the degree is $0$ otherwise. 

\begin{prop}
\label{prop:Lagrangiansurgeryoftwocurves}
In the above setting, assume that the intersection point $x$ is of degree $1$. Then, there is an immersed oriented Lagrangian cobordism 
\[ V : (\gamma_1, \gamma_2) \leadsto \gamma_1 \#_{x,\varepsilon} \gamma_2 . \]	
\end{prop}

\begin{proof}
First, we introduce some notations. We choose a Darboux chart $\phi : U \ni x \to B(0,r)$ (with $0 < r < \frac{1}{2}$) such that $\phi \circ \gamma_1$ parameterizes $\R \cap B(0,r)$ from left to right and such that $\phi \circ \gamma_2$ parameterizes $i\R \cap B(0,r)$ from bottom to top. We let $\psi : B(0,r) \times S_g \to B(0,r) \times B(0,r)$ be the Darboux chart given by $\Id \times \phi $.

Moreover, let $\alpha : \R \to \C$ be the path given by $\alpha(t) = t$ and $\beta = (x,y) : \R \to \C$ be a smooth path satisfying the following conditions
\begin{itemize}
	\item $\beta(t) = t + i$ for $t < -1$,
	\item $\beta(t) = t - i$ for $t > 1$,
	\item $\beta(t) = -it$ for $t \in (-r, r)$,
	\item $x'(t) \geqslant 0$ and $y'(t) \leqslant 0$ for all $t$.
\end{itemize}
We also define $P$ to be a smooth oriented pair of pants with boundary components labeled by $C_1$, $C_2$ and $C_3$ (see Figure \ref{figure:immersionofthesurgery}).

\input{Figures/FigureA.tex}

We give the handle that is used to resolve the intersection point $x$. Let $c: \R \to \C$ be a path such as in subsection \ref{subsubsection:surgeryofdoublepoints}. For $\varepsilon > 0$, put 
\[ H_\varepsilon^+ = \ens{\varepsilon c(t) z}{t \in \R, z = (x,y) \in S^1, \Re(c(t)x) \leqslant 0}. \]
Define an new immersion as follows. Consider the immersion given by $(\alpha_{\vert \R_-} \times \gamma_1 ) \coprod ( \beta_{\vert \R_-} \times \gamma_2)$. Remove its intersection with $B(0,r) \times U$ and replace it with $\psi^{-1}(H_\varepsilon^+)$. This yields an oriented Lagrangian immersion $i^- : P \to \C \times S_g$ such that $i^-$ coincides with $\alpha \times \gamma_1$ on a neighborhood of $C_1$, $i^{-1}$ coincides with $\beta \times \gamma_2$ on a neighborhood of $C_2$ and $i^-$ is the immersion $\{ 0 \} \times \gamma_1 \#_{x,\varepsilon} \gamma_2$ over $C_3$. (The immersion $i-$ is oriented because of the assumption on the degree of $x$). Moreover, the outward pointing direction to $C_3$ maps through $di^-$ to a vector pointing into fourth quadrant.

Notice that the double points of the immersion $i$ are of three types,
\begin{itemize}
	\item those given by the cartesian product of $\alpha_{\vert \R_-}$ and the double points of $\gamma_1$,
	\item those given by the cartesian product of $\beta_{\vert \R_-}$ and the double points of $\gamma_2$,
	\item the intersection points between $\gamma_1$ and $\gamma_2$ different from $x$ at the point $0$.	
\end{itemize}

We now extend this immersion so that it becomes an actual Lagrangian cobordism. We explain this following the procedure of \cite{BC13}.

We consider the genus $0$ surface with four boundary components $S$ obtained by the gluing of two copies of $P$ along the boundary $C_3$ and call its two new boundary components $\tilde{C}_1, \tilde{C}_2$ (see Figure \ref{figure:immersionofthesurgery}). Moreover, we let $i^+$ be the immersion $P \to \C \times S_g$ given by the composition of $i^-$ with the reflexion $(z,x) \in \C \times S_g \mapsto (-z,x)$. Their union yields a Lagrangian immersion $i : S \to \C \times S_g$ which is a Lagrangian cobordism $(\gamma_1,\gamma_2) \leadsto (\gamma_2, \gamma_1)$. We extend this to a local embedding $\iota : T^*_\varepsilon S \to \C \times S_g$ such that its restriction to the zero section coincides with $i$ and the pullback of the symplectic form coincides with standard one. 

For $\alpha >0$ small enough, the immersion $j : (-\alpha,\alpha) \times S^1 \to \C \times S_g$ given by $j(s,t) = (s(1-i), \gamma_1 \#_{x,\varepsilon} \gamma_2(t))$ lifts to a Lagrangian embedding $\tilde{j} : (-\alpha,\alpha) \times S^1 \to T^*_\varepsilon S $ through $\iota$. This is a consequence of the homotopy lift Theorem for covers. Indeed, $j$ coincides with $\{0\} \times \gamma_1 \# \gamma_2$ on $\{0 \} \times S^1$.

Reducing $\alpha$ if necessary, we can assume that $\tilde{j}$ is the graph of a closed one-form $\lambda$ (because the tangent space of $\tilde{j}$ at a point $(0,x)$ is transversal to the fiber). Since $(-\alpha,\alpha) \times S^1$ is homotopy equivalent to $\{0\} \times S^1$ and $\lambda_{\vert \{0\} \times S^1}$ is zero, there is a smooth function $F : (-\alpha,\alpha) \times S^1 \to \R$ such that $\lambda = dF$ and $F_{\vert \{0 \} \times S^1} = 0$. Let $\beta: (-\alpha,\alpha)$ be a smooth function such that $\beta(t) = 0$ for $t < 0$, $\beta(t) = 1$ for $t > \frac{\alpha}{2}$ and $\beta'(t) >0$. We replace the graph of $\lambda$ by the graph of $d (\beta F)$. Composing with $\iota$, we get an immersion of the pair of pants with coincides with $i$ and with $j$ at its ends.
\end{proof} 

We give a precise description of the double points of the cobordism $(\gamma_1, \gamma_2) \leadsto \gamma_1 \#_x \gamma_2$. In order to do this, we first describe some relevant charts near the double points of $i^+ \sqcup i^-$.

In what follows, we identify the restriction of $i^+ \sqcup i^-$ to $C$ (see Figure \ref{figure:immersionofthesurgery}) with the immersion $\{0\} \times (\gamma_1 \#_x \gamma_2) $. We let $\varepsilon$ be a positive real smaller than $\frac{2}{3}$ that we may reduce if necessary.

\paragraph{\bf{Chart near a self-intersection point of $\gamma_1$}} Let $y$ be a self-intersection point of $\gamma_1$. Call $s \neq t  \in C$ its pre-images by $\{0\} \times (\gamma_1 \#_x \gamma_2)$.

We choose a Darboux chart $\phi_{1,y}: U_{1,y} \to B(0,r_{1,y}) \subset \C $ near $y$ such that $\phi_{1,y}(\gamma_1)$ parameterizes the real line $\R$ (resp. $i \R$) near $s$ (resp. near $t$).

We can consider the following maps,
\begin{align} 
\psi_{s} : (x,y,a,b) \in (-\varepsilon,\varepsilon)^4 & \mapsto (x,y,a,b) \in \C \times U_{i,y}, \\
\psi_t : (x,y,a,b) \in (-\varepsilon,\varepsilon)^4 & \mapsto (x,y,-b,a) \in \C \times U_{i,y}, 
\end{align}
which is expressed in the chart $\Id \times \phi_{1,y}$. These are Darboux embedding (the domain is equipped with the symplectic form $dx \wedge dy + da \wedge db$).

Since $i^+ \sqcup i^-$ coincides with $\R \times \R$ near $s$, $\psi_s$ restricted to $(-\varepsilon,\varepsilon) \times \{0\} \times (-\varepsilon,\varepsilon) \times \{0\}$ yields coordinates of $S$ near $s$. So the map $\psi_s$ is actually an embedding of a neighborhood of $s$ in $T^*_\varepsilon S$.

Similarly, the map $\psi_t$ is an embedding of a neighborhood of $t$ in $T^*_\varepsilon S$.

\paragraph{\bf{Chart near a self-intersection point of $\gamma_2$}}
We let $y$ be a self-intersection point of $\gamma_2$ and call $s \neq t$ its pre-images by $\{0\} \times \gamma_1 \#_x \gamma_2$.
We choose a Darboux chart $\phi_{2,y} : U_{2,y} \to B(0,r_{2,y}) \subset \C$ such that $\phi_{2,y} (\gamma_2)$ parameterizes the line $\R$ (resp. $i\R$) near $s$ (resp. near $t$).
We consider the following maps,
\begin{align} 
\psi_{s} : (x,y,a,b) \in (-\varepsilon,\varepsilon)^4 & \mapsto (-y,x,a,b) \in \C \times U_{i,y}, \\
\psi_t : (x,y,a,b) \in (-\varepsilon,\varepsilon)^4 & \mapsto (-y,x,-b,a) \in \C \times U_{i,y}, 
\end{align}
which is read in the chart $\Id \times \phi_{1,y}$. These are Darboux embedding (the domain is equipped with the symplectic form $dx \wedge dy + da \wedge db$).

Since $i^+ \sqcup i^-$ coincides with $\R \times \R$ near $s$, $\psi_s$ restricted to $(-\varepsilon,\varepsilon) \times \{0\} \times (-\varepsilon,\varepsilon) \times \{0\}$ yields coordinates of $S$ near $s$. So the map $\psi_s$ is actually an embedding of a neighborhood of $s$ in $T^*_\varepsilon S$.

Similarly, the map $\psi_t$ is an embedding of a neighborhood of $t$ in $T^*_\varepsilon S$.
 
\paragraph{\bf{Chart near an intersection point of $\gamma_1$ and $\gamma_2$}}
We let $y \neq x$ be an intersection point of $\gamma_1$ and $\gamma_2$ different from the surgered point above. We choose a Darboux chart $\phi_y : U_y \to B(0,r_y) \subset \C$ such that $\phi_y(\gamma_1) \subset \R$ and $\phi_y(\gamma_2) \subset i \R$.

We consider the following maps
\begin{align}
	\psi_s : (x,y,a,b) \in (-\varepsilon,\varepsilon)^4 & \mapsto (x,y,a,b) \in \C \times U_y \\
	\psi_t : (x,y,a,b) \in (-\varepsilon,\varepsilon)^4 & \mapsto (-y,x,-b,a) \in \C \times U_y ,
\end{align}
which is read in the chart $\phi_y$. These are Darboux embedding when the domain is equipped with the symplectic form $dx \wedge dy + da \wedge db$.

Call $s \in C$ (resp. $t \in C$) the preimage of $(0,y)$ by $i^+ \sqcup i^-$ such that a small neighborhood of $s$ (resp. $t$) is mapped to $\R \times \R$ (resp. $i\R \times i\R$). The map $\psi_s$ (resp. $\psi_t$) yields local coordinates of $S$ near (resp. $t$). Hence, the map $\psi_s$ (resp. $\psi_t$) is a local embedding of a neighborhood of $s$ (resp. $t$) in $T^*_\varepsilon S$.

An easy Moser argument shows that the maps above extends to a local Weinstein embedding $ \Psi : V \subset T^*_\varepsilon \to \C \times S_g$. Here $V$ is a neighborhood of $C$ in $T^*_\varepsilon$.

Recall that near $\{0\} \times \gamma_1 \#_x \gamma_2$ the immersion $\{y = -x \} \times \gamma_1 \#_x \gamma_2$ is the image of $Gr(dF)$ by $\Psi$. 

\begin{itemize}
	\item Choose a self-intersection point $y = \gamma_1(s) = \gamma_1(t)$ of $\gamma_1$. We see that, in the coordinates $(x,a)$ near $s$, the function $F$ is given by $(x,a) \mapsto - \frac{x^2}{2}$. In the coordinates $(x,a)$ near $t$, $F$ is given by $(x,a) \mapsto - \frac{x^2}{2} $
	\item Similarly, consider a self-intersection point $y = \gamma_2(s) = \gamma_2(t)$ of $\gamma_2$. In the coordinates $(x,a)$ near $s$, the function $F$ is given by $(x,a) \mapsto \frac{x^2}{2}$. Near $t$, $F$ coincides with $(x,a) \mapsto \frac{x^2}{2}$.
	\item Lastly, choose an intersection point $y = \gamma_1(s) = \gamma_2(s)$ between $\gamma_1$ and $\gamma_2$. In the coordinates $(x,a)$ near $s$, the function $F$ is given by $(x,a) \mapsto - \frac{x^2}{2}$. In the coordinates $(x,a)$ near $t$, the function $F$ is given by $(x,a) \mapsto \frac{x^2}{2}$.
\end{itemize}

We choose the function cutoff function $\beta$ so that it depends only on $x$ in each of the coordinate patches above. Moreover $\beta$ satisfies the following hypotheses with $\frac{1}{2} > \alpha > 0$ and $0 < \eta < \alpha$.

\begin{itemize}
	\item $\beta = 0$ for $t \leqslant \frac{\alpha}{2}$,
	\item $\beta = 1$ for $t \geqslant \frac{\alpha}{2}$,
	\item $\beta' \geqslant 0$ and $\beta' \leqslant \frac{1}{\varepsilon - \alpha + \eta}$. 	
\end{itemize}
In particular this implies, for $x \in (-\varepsilon,\varepsilon)$, $x \beta + \frac{x^2}{2}\beta' \leqslant \frac{3 \varepsilon}{2} < 1$.

\begin{figure}

\captionsetup{justification=centering,margin=2cm}

\begin{tikzpicture}[y=0.80pt, x=0.80pt, yscale=-1.000000, xscale=1.000000, inner sep=0pt, outer sep=0pt]
  \path[draw=black,line join=round,line cap=round,miter limit=4.00,line
    width=0.100pt,rounded corners=0.0000cm] (131.1738,70.2304) rectangle
    (257.2430,196.2996);
  \path[draw=black,line join=miter,line cap=butt,miter limit=4.00,even odd
    rule,line width=0.300pt] (131.1147,133.3895) .. controls (131.1147,133.3895)
    and (175.8223,132.8984) .. (194.1680,133.2536) .. controls (212.5138,133.6089)
    and (217.7572,139.5867) .. (223.1128,147.5222) .. controls (228.4684,155.4576)
    and (229.6819,168.6910) .. (239.1559,178.3262) .. controls (248.6299,187.9614)
    and (253.7147,192.6990) .. (257.2214,196.3069);
  \path[draw=black,line join=round,line cap=round,miter limit=4.00,line
    width=0.100pt,rounded corners=0.0000cm] (289.4335,70.2304) rectangle
    (415.5027,196.2996);
  \path[draw=black,line join=miter,line cap=butt,miter limit=4.00,even odd
    rule,line width=0.300pt] (352.4310,70.1358) .. controls (352.4310,70.1358) and
    (351.9398,114.8434) .. (352.2951,133.1891) .. controls (352.6503,151.5348) and
    (358.6282,156.7783) .. (366.5636,162.1338) .. controls (374.4990,167.4895) and
    (387.7325,168.7030) .. (397.3677,178.1770) .. controls (407.0028,187.6510) and
    (411.7404,192.7358) .. (415.3484,196.2425);
  \path[draw=black,line join=round,line cap=round,miter limit=4.00,line
    width=0.100pt,rounded corners=0.0000cm] (442.0594,70.2304) rectangle
    (568.1286,196.2996);
  \path[draw=black,line join=miter,line cap=butt,miter limit=4.00,even odd
    rule,line width=0.300pt] (505.0569,70.1358) .. controls (505.0569,70.1358) and
    (504.5657,114.8434) .. (504.9210,133.1891) .. controls (505.2762,151.5348) and
    (511.2541,156.7783) .. (519.1895,162.1338) .. controls (527.1249,167.4895) and
    (540.3584,168.7030) .. (549.9936,178.1770) .. controls (559.6287,187.6510) and
    (564.3663,192.7358) .. (567.9743,196.2425);
  \path[draw=black,line join=miter,line cap=butt,miter limit=4.00,even odd
    rule,line width=0.300pt] (442.0003,133.3895) .. controls (442.0003,133.3895)
    and (486.7079,132.8984) .. (505.0536,133.2536) .. controls (523.3994,133.6089)
    and (528.6428,139.5867) .. (533.9984,147.5222) .. controls (539.3540,155.4576)
    and (540.5675,168.6910) .. (550.0415,178.3262) .. controls (559.5155,187.9614)
    and (564.6003,192.6990) .. (568.1070,196.3069);

\end{tikzpicture}
\caption{The projections of the surgery cobordism near the double points.}
\label{figure:projectionsurgerycobordism}
\end{figure}
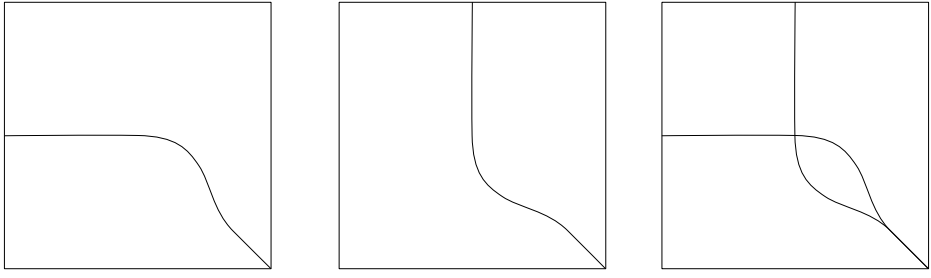

We deduce the following. 
\begin{itemize}
	\item Near $y = \gamma_1(s) = \gamma_1(t)$, the immersion is given by the two embeddings
	\begin{align*}
		(x,a) & \mapsto \left(x, -x\beta(x) - \frac{x^2}{2} \beta'(x), a ,0 \right) \\
		(x,a) & \mapsto \left(x,-x\beta(x)-\frac{x^2}{2} \beta'(x), 0,a \right).	
	\end{align*}
	There is a segment of double point which projects to the line $x \mapsto (x,-x\beta-\frac{x^2}{2} \beta')$(see Figure \ref{figure:projectionsurgerycobordism}).
	
	\item Similarly, near $y = \gamma_2(s) = \gamma_2(t)$, the immersion is given by the two embeddings
	\begin{align*}
		(x,a) & \mapsto \left( -x\beta(x) - \frac{x^2}{2} \beta'(x),x, a,0 \right) \\
		(x,a) & \mapsto \left(-x\beta(x) - \frac{x^2}{2} \beta'(x),x,0,a \right).	
	\end{align*}
 	There is a segment of double points which projects to the line $x \mapsto (- x\beta(x) - \frac{x^2}{2} \beta'(x),x)$ (see Figure \ref{figure:projectionsurgerycobordism}).
 	\item Near $y = \gamma_1(s) = \gamma_2(t)$, the immersion is given by the two embeddings
 	\begin{align*}
 		(x,a) & \mapsto \left(x, -x\beta(x) - \frac{x^2}{2} \beta'(x), a, 0 \right) \\	
 		(x,a) & \mapsto \left(-x\beta(x) - \frac{x^2}{2} \beta'(x),x,a,0 \right).
 	\end{align*}
	There is a double point at $(0,0)$ and a segment of double points which projects to the line $y = -x$ (see Figure \ref{figure:projectionsurgerycobordism}).
\end{itemize}

In what follows, we will call $DP_1$ (resp. $DP_2$) the set of double points coming from the double points of $\gamma_1$ (resp.$\gamma_2$). We will call $DP$ the set of double points along the line $y = -x$ and $DP_0$ the set of double points which project to $(0,0) \in \C$. 

\begin{lemma}
\label{lemma:degenerationofilambda}
There is a family $(i_\lambda)_{\lambda \in [0,1]}$ of immersion $S \looparrowright \C \times S_g$ such that
\begin{enumerate}[label =(\roman*)]
	\item we have $i_1 = i$, $i_0$ is the piecewise smooth immersion $S \to S_g$ given by $i^+ \sqcup j$,
	\item for any compact set $K \subset S \backslash C_3$, the map $i_\lambda$ is constant for $\lambda$ small enough,
	\item $i_\lambda $ converges uniformly to $i_0$ as $\lambda$ goes to $0$.
	\item The maps $i_\lambda$ are constant in a neighborhood of $DP_1$, $DP_2$ and $DP$.
\end{enumerate}
\end{lemma}

\begin{proof}
We consider the family of cutoff functions $\beta_\lambda(x) = \beta(\frac{x}{\lambda})$. The family $(i_\lambda)_{\lambda \in [0,1]}$ is obtained by replacing $\Phi ( Gr(d (\beta F)))$ by $\Phi ( Gr(d (\beta_\lambda F)) )$ for $\lambda \in [0,1]$.
\end{proof}

\subsection{Computation of the cobordism group}

\subsubsection{The applications $\pi$ and $\mu$}
\label{subsubsection:definitionofpiandmu}
Let $\gamma_1, \ldots, \gamma_N$ be immersed curves. We assume that there is an oriented immersed Lagrangian cobordism $V : \left (\gamma_1, \ldots, \gamma_n \right) \leadsto \emptyset$. In this case, we say that $\gamma_1, \ldots, \gamma_N$ are \emph{immersed Lagrangian cobordant}.

It is easy to see that the classes of theses curves in $H_1 \left (S_g,\Z \right)$ must satisfy
\[ \sum_{i = 1}^N \left [ \gamma_i \right ] = 0 .\]
Therefore, the map which associates to an immersed curve $\gamma$ its homology class $[\gamma]$ induces a well-defined group morphism
\[ \pi : \Gimm \to H_1 \left(S_g,\Z \right). \]

As stated in the introduction, there is a morphism
\[ \mu : \Gimm \to \Z / \chi \left(S_g \right) \Z, \]
which is a variant of the Maslov index. We define it following Seidel's paper (\cite[2.b.]{Sei00}). An alternate definition as a winding number appears in a paper of Chillingworth (\cite{Chi72}). 

We fix a complex structure $j$ on $S_g$, so that $TS_g$ is a complex line bundle. Choose another line bundle $Z \to S_g$ of degree $1$ over $S_g$ and a complex isomorphism 
\begin{equation} \Phi : TS_g \tilde{\longrightarrow} Z^{\otimes \chi(S_g)} .
\label{eqn:definitionofPhi} 
\end{equation}
Denote by $TS_g \backslash \{ 0 \}$ the tangent bundle of $S_g$ minus the zero section. We let $(\gamma, \tilde{\gamma} ) : S^1 \to T S_g \backslash \{0 \} $ be a nowhere vanishing curve in $T S_g$. We also let $v : S^1 \to Z$ be a nowhere vanishing section of the fiber bundle $\gamma^* Z$. There is a function $\lambda_v : S^1 \to \C^*$ such that for all $t \in S^1$
\[ \tilde{\gamma}(t) = \lambda_v(t) \Phi(v(t))^{\otimes \chi(S_g)}. \]
If $w$ is another nowhere vanishing section of $\gamma^*Z$, denote by $\lambda_w : S^1 \to \C^*$ the function such that
\[ \forall t \in S^1, \ \tilde{\gamma}(t) = \lambda_w(t) \Phi(w(t))^{\otimes \chi(S_g)}. \]
There is a function $\mu : S^1 \to \C^*$ such that
\[ \forall t \in S^1, \ v(t) = \mu(t) w(t) . \]
So
\[ \forall t \in S^1, \ \tilde{\gamma}(t) = \lambda_v(t) \mu(t)^{\chi(S_g)} \phi\left (w(t) \right)^{\otimes \chi(S_g)} .\]
Therefore, we have $\deg \left (\lambda_v \right) = \deg \left(\lambda_w \right)$ modulo $\chi(S_g)$. So it makes sense to define the Maslov index $\mu_\Phi \left(\tilde{\gamma} \right) \in \Z / \chi(S_g)$ by
\[ \mu_\Phi \left(\tilde{\gamma} \right) = \deg(\lambda_v) \mod \chi(S_g). \]
Let $(\gamma_1, \tilde{\gamma_1}) : S^1 \to TS_g \backslash \{ 0 \}$ be another nowhere vanishing curve. We assume that $\tilde{\gamma}$ and $\tilde{\gamma_1}$ are homotopic. Then, it is easy to check that
\[ \mu_\Phi \left(\tilde{\gamma} \right) = \mu_\Phi \left(\tilde{\gamma_1} \right). \]
We conclude that there is a well defined morphism 
\[ \mu_\Phi \in \Hom \left (\pi_1 (TS_g \backslash \{ 0\}), \Z / \chi(S_g) \Z \right ) = H^1 \left (TS_g \backslash \{ 0\}, \Z / \chi(S_g) \Z \right ). \]
An immersed curve $\gamma : S^1 \to S_g$ has a canonical lift $\tilde{\gamma}$ to $T S_g \backslash \{ 0 \}$ given by 
\[ \tilde{\gamma} :
\begin{array}{ccc}
S^1 & \to & TS_g \backslash \{ 0 \} \\
t & \mapsto & (\gamma(t), \gamma'(t))
\end{array}.
\]
 We put,
\[ \mu_\Phi(\gamma) := \mu_\Phi ([\tilde{\gamma}]) \in \Z / \chi(S_g) \Z, \]
where $\left [\tilde{\gamma} \right]$ is the class of $\tilde{\gamma}$ in the homology group $H_1 \left(TS_g \backslash \{0 \}, \Z / \chi(S_g) \Z \right)$.

\begin{prop}
\label{prop:invarianceofmuundercobordism}
Let 
\[ \gamma_1, \ldots, \gamma_N : S^1 \looparrowright S_g\] 
be immersed Lagrangian cobordant curves. In $\Z / \chi(S_g) \Z$, we have the relation
\[ \sum_{i=1}^N \mu_\Phi(\gamma_i) = 0 . \]
\end{prop}

\begin{proof}
First, we generalize the above construction of $\mu$. Denote by $\pi_{\C} : \C \times S_g \to \C$ and $\pi_{S_g} : \C \times S_g \to S_g$ the projection on the first and second factor respectively. There is a canonical isomorphism 
\[ \pi_{S_g}^* \Lambda^1 TS_g \tilde{\to} \Lambda^2 T(\C \times S_g). \] 
We compose this with the map \ref{eqn:definitionofPhi} to obtain an isomorphism
\[ \Psi : (\pi_{S_g}^*Z)^{\otimes \chi(S_g)} \tilde{\longrightarrow} \Lambda^2 T(\C \times S_g). \]
Let $\gamma : S^1 \to \C \times S_g$ be a smooth loop. We let 
\[ \Lambda(t) \subset T_{\gamma(t)} (\C \times S_g), \ t \in S^1 \] 
be a smooth loop of oriented lagrangian subspaces over $\gamma$. For each $t \in S^1$, we let $(e_1(t),e_2(t))$ be a (real) basis of the vector space $\Lambda(t)$. We assume that the family $(e_1,e_2)$ is smooth. We let $v$ be a trivialization of the complex line bundle $\left ( \pi_{S_g} \circ \gamma \right )^* Z$. 

The function 
\[ \begin{array}{ccc}
S^1 & \mapsto & \Lambda^2 T\left(\C \times S_g \right) \\
t & \mapsto & e_1(t) \wedge e_2(t)
\end{array},\]
is nowhere vanishing. So there is a smooth function $\lambda : S^1 \to \C^*$ such that 
\[ e_1 \wedge e_2 = \lambda(t) \Psi(v(t)) . \]
Now, we put
\[ \mu_\Phi(\Lambda) = \deg(\lambda) .\]
As before, one can easily check that this does not depend on the homotopy class of $\Lambda$ and does not depend on the choice of the section $v$. Thus, this induces a well-defined class 
\[ \mu_\Phi \in H^1(\mathcal{GL}^{or}(T(\C \times S_g)), \Z / \chi(S_g) \Z), \]
in the first cohomology group of the oriented Lagrangian Grassmannian.

Let $\gamma : S^1 \to S_g$ be an immersed curve and $x \in \R$. For $t \in S^1$, we let 
\[ \Lambda(t) = \Spann \left( (1,0), (0,\gamma'(t)) \right) \subset T_{(x,\gamma(t))} (\C \times S_g). \] 
Then, for any $t \in S^1$, the space $\Lambda(t)$ is Lagrangian. It is an easy exercise to check
\[ \mu_\Phi( \Lambda ) = \mu(\gamma) . \]

Let $i : W \looparrowright \C \times S_g$ be an oriented immersed Lagrangian cobordism between the immersed curves $\gamma_1, \ldots, \gamma_N$. Then by the discussion above
\[ \mu_\Phi (\gamma_1) + \ldots + \mu_\Phi(\gamma_N) = \langle i_V^*\mu, [\partial W] \rangle . \]
The class $\partial W$ is a boundary in $H_1 \left (W, \Z / \chi(S_g) \Z \right )$, so the left term is $0$.	
\end{proof}
 
We conclude that there is a well-defined morphism
\[ \mu_\Phi : \Gimm \to \Z / \chi(S_g) \Z . \]

\begin{rk}
\label{rk:dependanceofmu}
The map $\mu_\Phi$ \emph{depends} on the choice of the isomorphism $\Phi$. In fact, two such maps differ by the morphism induced by a cohomology class $p^*\alpha$ with $\alpha \in H^1(S_g, \Z / \chi(S_g) \Z)$ and $p$ the projection $TS_g \backslash \{ 0 \} \to S_g$. 

From now on, we fix once and for all one such $\Phi$. We will, therefore, denote $\mu_\Phi$ by $\mu$.
\end{rk}

\subsubsection{Action of the mapping class group on $\Gimm$}
\label{subsubsection:actionoftheMCGonGimm}
As usual, the \emph{Mapping Class Group} of the surface $S_g$ is the quotient of the group of orientation preserving diffeomorphisms by its identity component,
\[ \Mod \left( S_g \right) = \Diff^+ \left( S_g \right) / \Diff_0\left( S_g \right). \]
The Mapping Class Group has a natural left action on $\Gimm$. Given two classes $[\phi] \in \Mod\left(S_g \right)$ and $[\gamma] \in \Gimm$, the action is given by $[\phi] \cdot [\gamma] = [\phi \circ \gamma]$.

We recollect a few well-known facts on the Mapping Class Group. The reader may find proofs and statements in the book by Farb and Margalit \cite{FM012}.

A particular class of elements of the Mapping Class groups are given by \emph{Dehn twists}, which we now define. Let $\alpha : S^1 = \R / \Z \to S_g$ be an embedded curve. Choose a Weinstein embedding $\psi : [0,1] \times S^1 \to S_g$ such that $\psi_{\vert \{ \frac{1}{2} \} \times S^1 } = \alpha$. We let $f : [0,1] \to \R$ be an increasing smooth function equal to $1$ in a neighborhood of $1$ and equal to $0$ in a neighborhood of $0$. The map 
\[ \begin{array}{ccc} [0,1] \times S^1 & \to & [0,1] \times S^1 \\ (t,\theta) & \mapsto & \left(t, \theta + f(t) \right) \end{array} \]
 extends by the identity to a symplectomorphism 
 \[ T_\alpha : S_g \to S_g \] 
 which is called the \emph{Dehn twist about $\alpha$}. Notice that its class in $\Mod(S_g)$ does not depend on the choice of $\phi$ and $f$.
 
It is a well-known fact that these transformations generate the Mapping Class Group. More precisely, we let $\alpha_1, \ldots, \alpha_g$, $\beta_1, \ldots, \beta_g$ and $\gamma_1, \ldots, \gamma_{g-1}$ be the embedded curves represented in Figure \ref{figure:Lickorishgenerators}.
\begin{figure}
\captionsetup{justification=centering,margin=2cm}

\definecolor{cff000e}{RGB}{255,0,14}
\definecolor{c005100}{RGB}{0,81,0}
\definecolor{c003cb2}{RGB}{0,60,178}

\begin{tikzpicture}[y=0.80pt, x=0.80pt, yscale=-1.000000, xscale=1.000000, inner sep=0pt, outer sep=0pt]
  \path[draw=cff000e,line join=round,line cap=round,miter limit=4.00,fill
    opacity=0.357,line width=0.400pt] (220.2988,198.7895) ellipse (0.5510cm and
    0.4289cm);
  \path[draw=c005100,line join=round,line cap=round,miter limit=4.00,fill
    opacity=0.357,line width=0.400pt]
    (220.6536,256.3137)arc(90.000:150.000:5.853188 and
    25.604)arc(150.000:210.000:5.853188 and 25.604)arc(210.000:270.000:5.853188
    and 25.604);
  \path[draw=c005100,dash pattern=on 1.60pt off 1.60pt,line join=round,line
    cap=round,miter limit=4.00,fill opacity=0.357,line width=0.400pt]
    (220.2101,205.1055)arc(-90.000:0.000:5.853188 and
    25.604)arc(0.000:90.000:5.853188 and 25.604);
  \path[draw=cff000e,line join=round,line cap=round,miter limit=4.00,fill
    opacity=0.357,line width=0.400pt] (295.1487,198.7611) ellipse (0.5510cm and
    0.4289cm);
  \path[draw=c005100,line join=round,line cap=round,miter limit=4.00,fill
    opacity=0.357,line width=0.400pt]
    (295.8582,256.7669)arc(90.000:150.000:5.853188 and
    25.604)arc(150.000:210.000:5.853188 and 25.604)arc(210.000:270.000:5.853188
    and 25.604);
  \path[draw=c005100,dash pattern=on 1.60pt off 1.60pt,line join=round,line
    cap=round,miter limit=4.00,fill opacity=0.357,line width=0.400pt]
    (295.4147,205.5587)arc(-90.000:0.000:5.853188 and
    25.604)arc(0.000:90.000:5.853188 and 25.604);
  \path[draw=c003cb2,line join=round,line cap=round,miter limit=4.00,fill
    opacity=0.357,line width=0.400pt] (258.3167,201.1958) ellipse (0.8327cm and
    0.2591cm);
  \path[draw=c003cb2,line join=round,line cap=round,miter limit=4.00,fill
    opacity=0.357,line width=0.400pt]
    (228.8126,201.1958)arc(180.000:270.000:29.504171 and
    9.180)arc(270.000:360.000:29.504171 and 9.180);
  \path[draw=black,line join=miter,line cap=butt,miter limit=4.00,even odd
    rule,line width=0.400pt] (287.7786,200.7310) .. controls (287.7786,200.7310)
    and (292.2385,197.3034) .. (295.4682,197.3034) .. controls (299.2447,197.3034)
    and (303.1578,200.7310) .. (303.1578,200.7310)(284.4821,195.9961) .. controls
    (284.4821,195.9961) and (289.1564,205.1555) .. (296.0111,204.8329) .. controls
    (301.8814,204.5567) and (307.5401,195.9961) ..
    (307.5401,195.9961)(213.0288,200.8912) .. controls (213.0288,200.8912) and
    (217.4888,197.4637) .. (220.7184,197.4637) .. controls (224.4949,197.4637) and
    (228.4080,200.8912) .. (228.4080,200.8912)(209.7323,196.1564) .. controls
    (209.7323,196.1564) and (214.4067,205.3157) .. (221.2613,204.9932) .. controls
    (227.1317,204.7170) and (232.7903,196.1564) ..
    (232.7903,196.1564)(329.4401,159.7151) .. controls (329.4401,159.7151) and
    (308.2011,144.5339) .. (296.3312,144.5339) .. controls (284.6642,144.5339) and
    (270.1522,159.1958) .. (258.4924,159.7151) .. controls (243.8159,160.3687) and
    (235.6076,144.5339) .. (220.6536,144.5339) .. controls (199.2795,144.5339) and
    (176.9958,171.7807) .. (176.6660,197.7908) .. controls (176.3254,224.6533) and
    (200.6174,256.4668) .. (220.6536,256.4668) .. controls (235.5336,256.4668) and
    (245.3409,242.1920) .. (258.4924,242.1920) .. controls (271.6438,242.1920) and
    (284.7744,256.4668) .. (296.3312,256.4668) .. controls (303.4960,256.4668) and
    (329.4401,242.1920) .. (329.4401,242.1920);
  \path[draw=cff000e,line join=round,line cap=round,miter limit=4.00,fill
    opacity=0.357,line width=0.400pt] (458.0980,197.6821) ellipse (0.5510cm and
    0.4289cm);
  \path[draw=c005100,line join=round,line cap=round,miter limit=4.00,fill
    opacity=0.357,line width=0.400pt]
    (458.4527,255.2063)arc(90.000:150.000:5.853188 and
    25.604)arc(150.000:210.000:5.853188 and 25.604)arc(210.000:270.000:5.853188
    and 25.604);
  \path[draw=c005100,dash pattern=on 1.60pt off 1.60pt,line join=round,line
    cap=round,miter limit=4.00,fill opacity=0.357,line width=0.400pt]
    (458.0093,203.9981)arc(270.000:360.000:5.853188 and
    25.604)arc(-0.000:90.000:5.853188 and 25.604);
  \path[draw=cff000e,line join=round,line cap=round,miter limit=4.00,fill
    opacity=0.357,line width=0.400pt] (532.9479,197.6538) ellipse (0.5510cm and
    0.4289cm);
  \path[draw=c005100,line join=round,line cap=round,miter limit=4.00,fill
    opacity=0.357,line width=0.400pt]
    (533.6573,255.6595)arc(90.000:150.000:5.853188 and
    25.604)arc(150.000:210.000:5.853188 and 25.604)arc(210.000:270.000:5.853188
    and 25.604);
  \path[draw=c005100,dash pattern=on 1.60pt off 1.60pt,line join=round,line
    cap=round,miter limit=4.00,fill opacity=0.357,line width=0.400pt]
    (533.2139,204.4513)arc(270.000:360.000:5.853188 and
    25.604)arc(-0.000:90.000:5.853188 and 25.604);
  \path[draw=black,line join=miter,line cap=butt,miter limit=4.00,even odd
    rule,line width=0.400pt] (465.4706,199.7696) .. controls (465.4706,199.7696)
    and (461.0107,196.3421) .. (457.7810,196.3421) .. controls (454.0045,196.3421)
    and (450.0914,199.7696) .. (450.0914,199.7696)(468.7671,195.0348) .. controls
    (468.7671,195.0348) and (464.0928,204.1942) .. (457.2381,203.8716) .. controls
    (451.3678,203.5954) and (445.7091,195.0348) ..
    (445.7091,195.0348)(540.2204,199.9299) .. controls (540.2204,199.9299) and
    (535.7604,196.5023) .. (532.5308,196.5023) .. controls (528.7543,196.5023) and
    (524.8412,199.9299) .. (524.8412,199.9299)(543.5169,195.1950) .. controls
    (543.5169,195.1950) and (538.8426,204.3544) .. (531.9879,204.0319) .. controls
    (526.1175,203.7556) and (520.4589,195.1950) ..
    (520.4589,195.1950)(423.8091,158.7537) .. controls (423.8091,158.7537) and
    (445.0481,143.5725) .. (456.9181,143.5725) .. controls (468.5850,143.5725) and
    (483.0970,158.2345) .. (494.7568,158.7537) .. controls (509.4333,159.4073) and
    (517.6416,143.5725) .. (532.5956,143.5725) .. controls (553.9697,143.5725) and
    (576.2534,170.8194) .. (576.5832,196.8294) .. controls (576.9239,223.6920) and
    (552.6318,255.5055) .. (532.5956,255.5055) .. controls (517.7156,255.5055) and
    (507.9083,241.2307) .. (494.7568,241.2307) .. controls (481.6054,241.2307) and
    (468.4748,255.5055) .. (456.9181,255.5055) .. controls (449.7532,255.5055) and
    (423.8091,241.2307) .. (423.8091,241.2307);
  \path[draw=c003cb2,line join=round,line cap=round,miter limit=4.00,fill
    opacity=0.357,line width=0.400pt] (494.9871,199.7680) ellipse (0.8327cm and
    0.2591cm);
  \path[draw=c003cb2,line join=round,line cap=round,miter limit=4.00,fill
    opacity=0.357,line width=0.400pt]
    (465.4830,199.6945)arc(180.000:270.000:29.504171 and
    9.180)arc(-90.000:0.000:29.504171 and 9.180);
  \path[draw=c005100,line join=miter,line cap=butt,miter limit=4.00,even odd
    rule,line width=0.400pt] (213.6475,232.9586) -- (214.8447,230.6362) --
    (216.0420,232.9586);
  \path[draw=c003cb2,line join=miter,line cap=butt,miter limit=4.00,even odd
    rule,line width=0.400pt] (256.7942,190.5028) -- (258.6123,192.0322) --
    (256.7942,193.5617);
  \path[draw=cff000e,line join=miter,line cap=butt,miter limit=4.00,even odd
    rule,line width=0.400pt] (199.6005,200.1205) -- (200.7978,197.7980) --
    (201.9950,200.1205);
  \path[draw=c005100,line join=miter,line cap=butt,miter limit=4.00,even odd
    rule,line width=0.400pt] (526.6310,230.1141) -- (527.8283,227.7916) --
    (529.0255,230.1141);
  \path[draw=c005100,line join=miter,line cap=butt,miter limit=4.00,even odd
    rule,line width=0.400pt] (451.3725,228.6957) -- (452.5698,226.3733) --
    (453.7670,228.6957);
  \path[draw=c005100,line join=miter,line cap=butt,miter limit=4.00,even odd
    rule,line width=0.400pt] (288.8521,231.1460) -- (290.0493,228.8235) --
    (291.2466,231.1460);
  \path[draw=cff000e,line join=miter,line cap=butt,miter limit=4.00,even odd
    rule,line width=0.400pt] (512.2379,199.5063) -- (513.4352,197.1838) --
    (514.6324,199.5063);
  \path[draw=cff000e,line join=miter,line cap=butt,miter limit=4.00,even odd
    rule,line width=0.400pt] (437.4078,199.5185) -- (438.6050,197.1960) --
    (439.8023,199.5185);
  \path[draw=cff000e,line join=miter,line cap=butt,miter limit=4.00,even odd
    rule,line width=0.400pt] (274.4130,200.9216) -- (275.6103,198.5991) --
    (276.8075,200.9216);
  \path[draw=c003cb2,line join=miter,line cap=butt,miter limit=4.00,even odd
    rule,line width=0.400pt] (256.7942,190.5028) -- (258.6123,192.0322) --
    (256.7942,193.5617);
  \path[draw=c003cb2,line join=miter,line cap=butt,miter limit=4.00,even odd
    rule,line width=0.400pt] (495.0124,188.9556) -- (496.8304,190.4850) --
    (495.0124,192.0145);
  \path[draw=black,dash pattern=on 0.80pt off 2.40pt,line join=miter,line
    cap=butt,miter limit=4.00,even odd rule,line width=0.400pt]
    (334.3473,159.4885) -- (418.0657,159.4885);
  \path[draw=black,dash pattern=on 0.80pt off 2.40pt,line join=miter,line
    cap=butt,miter limit=4.00,even odd rule,line width=0.400pt]
    (334.3473,241.0591) -- (418.0657,241.0591);
    
    \draw (190,200) node {\tiny $\beta_1$};
    \draw (205,230) node {\tiny $\alpha_1$};
    \draw (257,185) node {\tiny $\gamma_1$};
    \draw (280,230) node {\tiny $\alpha_2$};
    \draw (323,200) node {\tiny $\beta_2$};
    \draw (423,200) node {\tiny $\beta_{g-1}$};
    \draw (560,200) node {\tiny $\beta_g$};
    \draw (440,230) node {\tiny $\alpha_{g-1}$};
    \draw (520,230) node {\tiny $\alpha_g$};
	\draw (495,185) node {\tiny $\gamma_g$};

\end{tikzpicture}

\caption{The Lickorish generators of the Mapping Class Group}
\label{figure:Lickorishgenerators}
\end{figure}
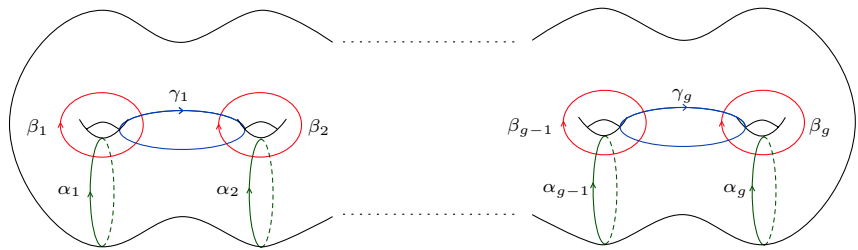

\begin{Theo}[Lickorish,1964, \cite{L64}]
\label{Theo:Lickorishgeneratorsofthe MCG}
The Dehn twists about the curves $\alpha_1, \ldots, \alpha_g$, $\beta_1, \ldots, \beta_g$ and $\gamma_1, \ldots, \gamma_{g-1}$ generate the Mapping Class Group.

In particular, any orientation-preserving diffeomorphism $\phi$ is the product of a symplectomorphism $\psi$ and a diffeomorphism $\chi$ isotopic to the identity. 	
\end{Theo}

In particular, any positive homeomorphism is isotopic to a symplectomorphism As a corollary of this and Lemma \ref{lemma:isotopiccurvesareimmlagcob}, we obtain the following.

\begin{lemma}
\label{lemma:theactionoftheMCGiswelldefined}
The map
\[ \begin{array}{ccc} \Mod(S_g) \times \Gimm & \to & \Gimm \\
\left( [\phi], [\gamma_1] + \ldots + [\gamma_N] \right) & \mapsto & [\phi \circ \gamma_1] + \ldots + [\phi \circ \gamma_N] \end{array} \]
is well-defined and is a group action on $\Gimm$.	
\end{lemma}

\begin{proof}
First we check that if $[\gamma_1] + \ldots + [\gamma_N] = 0$ in $\Gimm$, then $[\phi \circ \gamma_1] + \ldots + [\phi \circ \gamma_N]$. For this, write $\phi = \psi \circ \chi$ with $\psi$ symplectic and $\chi$ isotopic to the identity. There is an immersed oriented Lagrangian cobordism $ V : (\gamma_1, \ldots, \gamma_N) \leadsto \emptyset$. Then $\psi(V)$ is a Lagrangian cobordism between the curves $\psi(\gamma_1), \ldots, \psi(\gamma_N)$ which are isotopic (hence Lagrangian cobordant by \ref{lemma:isotopiccurvesareimmlagcob}) to $\phi(\gamma_1), \ldots, \phi(\gamma_N)$.

Similarly, Lemma \ref{lemma:isotopiccurvesareimmlagcob} implies that if $\phi$ is isotopic to $\psi$, then $[\phi \circ \gamma_1] + \ldots + [\phi \circ \gamma_N] = [\psi \circ \gamma_1] + \ldots + [\psi \circ \gamma_N]$ in $\Gimm$.
\end{proof}

We also have the following proposition:

\begin{prop}
\label{prop:ActionofaDehntwistonGimm}
Let $\beta$ be an embedded curve in $S_g$. Then in $\Gimm$
\[ \left[ T_\alpha (\beta) \right] = \left( \beta \cdot \alpha \right) [\alpha] + [\beta].\]
Here, $ \beta \cdot \alpha $ is the homological intersection number of $\beta$ and $\alpha$.
\end{prop}

\begin{proof}
Up to isotopy, we assume that $\alpha$ and $\beta$ are in \emph{minimal position} (i.e.\ the number of intersection points is minimal in their respective isotopy class).

There is a geometric procedure which produces a curve isotopic to $T_\alpha \beta$ after a sequence of surgeries such as in \ref{subsubsection:surgeryofdoublepoints}. The whole process is represented in Figure \ref{figure:surgeryforDehnTwist}.

\input{Figures/FigureC.tex}

Call $x_1, \ldots, x_N$ the intersection points of $\alpha$ and $\beta$ ordered according to the orientation of $\alpha$ (here $N$ is the number of intersection points between $\alpha$ and $\beta$). For $k \in \{ 1, \ldots, N \}$, we fix a Darboux chart $\phi_k : B(0,r) \to S_g$ around $x_k$ such that
\begin{itemize}
	\item in this chart, $\beta$	 is the oriented line $\R$,
	\item in this chart, $\alpha$ has image $i\R$.
\end{itemize}

The first step consists of the surgery between $\alpha$ and $\gamma$ if $x_1$ is of degree $1$ and of the surgery between $\alpha^{-1}$ and $\gamma$ if $x_1$ is of degree $0$. This yields a curve $c_1$

In the second step, we perturb $\alpha$ to a curve $\tilde{\alpha}_2$ as in the second row of Figure \ref{figure:surgeryforDehnTwist}. The main features of $\tilde{\alpha}$ are as follows
\begin{itemize}
	\item there are $x_2^2, \ldots x_N^2$ intersection points lying close to $x_2, \ldots, x_N$.
	\item there is one other intersection point $y_1$ \emph{above} $\beta$ in the Darboux chart $\phi_2$.
\end{itemize}
Now, we perform the surgery between $c_1$ and $\tilde{\alpha_2}$ at $x_2^2$ if it is of degree $1$ and between $c_1$ and $\tilde{\alpha_2}^{-1}$ otherwise. 

Assume that we performed the surgery of $\beta$ with $k$ curves $\alpha, \tilde{\alpha}_2, \ldots, \tilde{\alpha}_k$ isotopic to $\alpha$ to obtain a curve $c_k$. We perturb $\alpha$ to a curve $\tilde{\alpha}_{k+1}$ such as in Figure \ref{figure:surgeryforDehnTwist}. It satisfies the following assumptions.
\begin{itemize}
	\item There are $x_k^k, \ldots, x_k^N$ intersection points between $c_k$ and $\tilde{\alpha}_{k+1}$	 close to $x_k, \ldots, x_N$.
	\item There are intersections points $y_1, \ldots, y_{k}$ which lie above $\beta$ in the chart $\phi_k$.
\end{itemize}
Now, we perform the surgery between $c_k$ and $\tilde{\alpha}_{k+1}$ at $x_k$ according to the orientation of $\tilde{\alpha}_{k+1}$. The handle is big enough to delete the intersection points $y_1, \ldots, y_k$.

Notice that each surgery produces an oriented immersed Lagrangian cobordism by Propositions \ref{prop:Lagrangiansurgeryoftwocurves}. Composing these cobordisms and using Lemma \ref{lemma:isotopiccurvesareimmlagcob} about isotopic curves, we obtain an immersed Lagrangian cobordism
\[ \left( \alpha, \beta, \ldots, \beta, \beta^{-1}, \ldots, \beta^{-1} \right) \leadsto \gamma , \]
with as many copies of $\alpha$ as there are intersection points of degree $0$ and as many copies of $\alpha^{-1}$ as there are intersection points of degree 1. Hence in the Lagrangian cobordism group $\Gimm$
\[ [\gamma] = [\beta] + \left( \alpha \cdot \beta \right) [\alpha]. \]
This concludes the proof since $\gamma$ is Lagrangian cobordant to $T_\alpha (\beta)$ (Lemma \ref{lemma:isotopiccurvesareimmlagcob}).
\end{proof}

\subsubsection{Tori and pairs of pants}
We describe a set of generators for $\Gimm$ using the action of the mapping class group described above.

First, let $\gamma_1$ and $\gamma_2$ be two embedded curves in $S_g$. We suppose that each of these is the oriented boundary of an embedded torus. By the change of coordinates principle (\cite[1.3]{FM012}), there is a product of Dehn twists $\phi$ which maps $\gamma_1$ to a curve isotopic (hence immersed Lagrangian cobordant) to $\gamma_2$. By Proposition \ref{prop:ActionofaDehntwistonGimm}, we have $[\gamma_1] = [\gamma_2]$ in $\Gimm$. 
We conclude that there is a well-defined element 
	\begin{equation}
T \in \Gimm
\label{eqn:defTunobgroup}
	\end{equation}
which represent any oriented boundary of a torus in $S_g$.

First, we compute the Maslov index of the class $T$.

\begin{lemma}
\label{lemma:Maslovindexofatorus}
For any choice of isomorphism 
\[ \Phi : Z^{\otimes \chi(S_g)} \tilde{\rightarrow} TS_g, \] we have
\[ \mu_\Phi (T) = -1 \in \Z / \chi(S_g)\Z, \]
($T$ is the class defined in \ref{eqn:defTunobgroup}). 	
\end{lemma}

\begin{proof}
First, the index of $T$ does not depend on $\Phi$. To see this, let $\gamma$ be a representative of $T$ and $v$ a trivialization of $Z$ along $\gamma$. Let $\Psi : TS_g \tilde{\longrightarrow} Z^{\otimes \chi(S_g)} $	be an another complex isomorphism. Then $\Psi \circ \Phi^{-1}$ has the form $(z,v) \mapsto (z, \mu(z)v)$ where $\mu : S_g \to \C^*$ is a nowhere vanishing function. 
If 
\[ \gamma'(t) = \lambda(t) \Phi^{-1}(v \otimes \ldots \otimes v), \]
then 
\begin{eqnarray*} \gamma'(t) & = & \Psi^{-1} \circ \Psi \circ \Phi^{-1} (v \otimes \ldots \otimes v)\\
& = &  \lambda(t) \mu(\gamma(t)) \Psi^{-1} (v \otimes \ldots \otimes v). 
\end{eqnarray*}
But since $\mu$ extends to $S_g$ and $\gamma$ is homologically trivial, we have $\deg(\mu \circ \gamma) = 0$.

Let us turn to the computation of $\mu(T)$. It is a quick application of the Poincaré-Hopf theorem. Let $\tilde{T}$ be the torus bounded by $\gamma$ and $D$ be a disk embedded in $\tilde{T}$.

We choose trivializations of $T S_g$ over $S_g \backslash D$ and over $D$ so that $T S_g$ is identified with the fiber bundle obtained by gluing
$ \left( S_g \backslash D \right) \times \C$ on $D \times \C$ along the map
\[ f: \begin{array}{ccc}
\left( S_g \backslash D \right ) \times \C  & \to &  D \times \C \\
\left( \phi(e^{i\theta}), z \right) & \mapsto & \left (e^{i\theta}, e^{i \chi(S_g)\theta} z \right).
\end{array} \]
Here $\phi : \partial D \to \partial(S_g \backslash D)$ is an orientation reversing diffeomorphism.

Similarly, we define the line bundle $Z$ as the gluing of $ \left( S_g \backslash D \times \C \right) $ on $ \left( D \times \C \right)$ along the map
\[ g: \begin{array}{ccc}
\left( S_g \backslash D \right ) \times \C  & \to &  D \times \C \\ 	
\left( \phi(e^{i\theta}), z \right) & \mapsto & \left (e^{i\theta}, e^{i \theta} z \right).
\end{array} \]
The isomorphism $\Phi : Z^{\otimes \chi(S_g)} \to TS_g$ is given by $(a,\lambda_1 \otimes \ldots \otimes \lambda_n) \mapsto (a, \lambda_1 \ldots \lambda_n)$. Moreover, a non-zero section of $Z$ over $\gamma$ is given by $z \in \Im(\gamma) \mapsto (z,1)$. So the Maslov index of $\gamma$ is just the index of $\gamma'$ read in the trivialization above.

Choose a vector field $X$ on $\tilde{T}$ which coincides with $\gamma'$ over $\gamma$ and has a unique zero in $D$. This zero has degree $-1$ since the Euler characteristic of $\tilde{T}$ is $-1$. The degree of $\gamma'$ in the trivialization above is equal to the degree of $X$ over the boundary of $S_g \ D$ since this is a homological invariant. Given the expression of $-1$, this degree is also the degree of $X$ in the trivialization over $D$ plus $\chi(S_g)$. Hence it is $-1 + \chi(S_g)$ since $X$ has a zero of degree $-1$ on $D$. 
\end{proof}

\begin{rk}
The same proof also shows that if $\gamma$ is the oriented boundary of an embedded surface of genus $\tilde{g}$, then its index satisfies
\[ \mu_\Phi (\gamma) = \chi(S_1) \mod \chi(S_g), \]
for any trivialization $\Phi$.  
\end{rk}

We now express the class of any separating curve with $T$. The proof uses the surgeries of \cite[Lemma 7.6]{ab07}.

\begin{lemma}
\label{lemma:classofanonseparatingcurve}
Let $\gamma$ be the oriented boundary of an embedded surface $S_1$. Then in $\Gimm$
\[ 	[\gamma] = \chi(S_1) \cdot T. \]
\end{lemma}

\begin{proof}
The proof follows from induction over the genus of the surface bounded by $\gamma$.
If $\gamma$ bounds a torus, there is nothing to prove. 

We assume that the formula is true for any curve which bounds a surface of genus less than $g_1 - 1$. We assume that $\gamma$ is the oriented boundary a surface $S_1$ of genus $g_1 \geqslant 2$.

Choose three curves $\gamma_1$, $\gamma_2$ and $\gamma_3$ such that the following hold (see Figure \ref{figure:FigureD}).

\begin{figure}
\captionsetup{justification=centering,margin=2cm}

\begin{tikzpicture}[y=0.80pt, x=0.80pt, yscale=-1.000000, xscale=1.000000, inner sep=0pt, outer sep=0pt]
  \path[draw=black,line join=miter,line cap=butt,miter limit=4.00,even odd
    rule,line width=0.326pt] (362.5113,122.2744) .. controls (362.5113,122.2744)
    and (368.5213,118.3082) .. (372.8734,118.3082) .. controls (377.9624,118.3082)
    and (383.2356,122.2744) .. (383.2356,122.2744)(358.0691,116.7955) .. controls
    (358.0691,116.7955) and (364.3680,127.3942) .. (373.6050,127.0210) .. controls
    (381.5156,126.7014) and (389.1409,116.7955) ..
    (389.1409,116.7955)(261.7823,122.4598) .. controls (261.7823,122.4598) and
    (267.7923,118.4936) .. (272.1444,118.4936) .. controls (277.2333,118.4936) and
    (282.5065,122.4598) .. (282.5065,122.4598)(257.3400,116.9809) .. controls
    (257.3400,116.9809) and (263.6389,127.5796) .. (272.8759,127.2064) .. controls
    (280.7866,126.8868) and (288.4118,116.9809) ..
    (288.4118,116.9809)(418.6523,74.8131) .. controls (418.6523,74.8131) and
    (390.0317,57.2463) .. (374.0363,57.2463) .. controls (358.3145,57.2463) and
    (338.7589,74.2122) .. (323.0467,74.8131) .. controls (303.2694,75.5694) and
    (292.2083,57.2463) .. (272.0570,57.2463)(272.0570,186.7688) .. controls
    (292.1085,186.7688) and (305.3245,170.2507) .. (323.0467,170.2507) .. controls
    (340.7689,170.2507) and (358.4630,186.7688) .. (374.0363,186.7688) .. controls
    (383.6913,186.7688) and (418.6523,170.2507) .. (418.6523,170.2507);
  \path[draw=black,line join=miter,line cap=butt,miter limit=4.00,even odd
    rule,line width=0.326pt] (181.7160,122.2420) .. controls (181.7160,122.2420)
    and (175.7060,118.2758) .. (171.3539,118.2758) .. controls (166.2649,118.2758)
    and (160.9918,122.2420) .. (160.9918,122.2420)(186.1583,116.7631) .. controls
    (186.1583,116.7631) and (179.8594,127.3618) .. (170.6224,126.9886) .. controls
    (162.7117,126.6689) and (155.0864,116.7631) ..
    (155.0864,116.7631)(125.5751,74.7807) .. controls (125.5751,74.7807) and
    (154.1956,57.2139) .. (170.1910,57.2139) .. controls (185.9128,57.2139) and
    (205.4684,74.1798) .. (221.1807,74.7807) .. controls (240.9580,75.5370) and
    (252.0191,57.2139) .. (272.1704,57.2139)(272.1704,186.7363) .. controls
    (252.1188,186.7363) and (238.9029,170.2183) .. (221.1807,170.2183) .. controls
    (203.4585,170.2183) and (185.7644,186.7363) .. (170.1910,186.7363) .. controls
    (160.5361,186.7363) and (125.5751,170.2183) .. (125.5751,170.2183);
  \path[draw=black,dash pattern=on 3.91pt off 3.91pt,line join=round,line
    cap=round,miter limit=4.00,fill opacity=0.067,line width=0.326pt]
    (221.3155,122.4724) ellipse (0.2241cm and 1.3427cm);
  \path[draw=black,line join=round,line cap=round,miter limit=4.00,fill
    opacity=0.067,line width=0.326pt]
    (323.3443,170.0500)arc(90.000:150.000:7.941321 and
    47.578)arc(150.000:210.000:7.941321 and 47.578)arc(210.000:270.000:7.941321
    and 47.578);
  \path[draw=black,line join=round,line cap=round,miter limit=4.00,fill
    opacity=0.067,line width=0.326pt]
    (271.7283,186.6372)arc(90.000:150.000:6.924429 and
    29.760)arc(150.000:210.000:6.924429 and 29.760)arc(210.000:270.000:6.924429
    and 29.760);
  \path[draw=black,dash pattern=on 3.91pt off 3.91pt,line join=round,line
    cap=round,miter limit=4.00,fill opacity=0.067,line width=0.326pt]
    (271.7283,87.8343) ellipse (0.2007cm and 0.8625cm);
  \path[draw=black,line join=round,line cap=round,miter limit=4.00,fill
    opacity=0.067,line width=0.326pt]
    (221.3155,170.0500)arc(90.000:150.000:7.941321 and
    47.578)arc(150.000:210.000:7.941321 and 47.578)arc(210.000:270.000:7.941321
    and 47.578);
  \path[draw=black,dash pattern=on 3.91pt off 3.91pt,line join=round,line
    cap=round,miter limit=4.00,fill opacity=0.067,line width=0.326pt]
    (323.3443,74.8949)arc(270.000:360.000:7.941321 and
    47.578)arc(0.000:90.000:7.941321 and 47.578);
  \path[draw=black,line join=round,line cap=round,miter limit=4.00,fill
    opacity=0.067,line width=0.326pt]
    (271.7283,118.3940)arc(90.000:150.000:7.110409 and
    30.560)arc(150.000:210.000:7.110409 and 30.560)arc(210.000:270.000:7.110409
    and 30.560);
  \path[draw=black,dash pattern=on 3.91pt off 3.91pt,line join=round,line
    cap=round,miter limit=4.00,fill opacity=0.067,line width=0.326pt]
    (271.7283,127.1164)arc(-90.000:0.000:6.924429 and
    29.760)arc(-0.000:90.000:6.924429 and 29.760);
  \path[draw=black,fill=black,line join=round,line cap=round,miter limit=4.00,line
    width=0.268pt] (215.3597,122.6569) .. controls (214.2449,121.9434) and
    (213.8045,120.2621) .. (213.3585,118.4168) .. controls (212.8980,120.3219) and
    (212.4338,121.9680) .. (211.3573,122.6569) -- cycle;
  \path[draw=black,fill=black,line join=round,line cap=round,miter limit=4.00,line
    width=0.268pt] (266.6249,83.1877) .. controls (265.5101,83.9013) and
    (265.0698,85.5826) .. (264.6237,87.4279) .. controls (264.1633,85.5228) and
    (263.6990,83.8767) .. (262.6225,83.1877) -- cycle;
  \path[draw=black,fill=black,line join=round,line cap=round,miter limit=4.00,line
    width=0.268pt] (266.8775,152.2569) .. controls (265.7627,152.9704) and
    (265.3223,154.6518) .. (264.8763,156.4971) .. controls (264.4158,154.5920) and
    (263.9516,152.9459) .. (262.8751,152.2569) -- cycle;
  \path[draw=black,fill=black,line join=round,line cap=round,miter limit=4.00,line
    width=0.268pt] (317.3851,126.5713) .. controls (316.2703,125.8578) and
    (315.8299,124.1765) .. (315.3839,122.3311) .. controls (314.9235,124.2362) and
    (314.4592,125.8823) .. (313.3827,126.5713) -- cycle;
\draw (205,123) node {$\gamma_3$};
\draw (255,157) node {$\gamma_1$};
\draw (255,87) node {$\gamma_2$};
\draw (310,123) node {$\gamma$};
\end{tikzpicture}

\caption{Two pair of pants}
\label{figure:FigureD}
\end{figure}
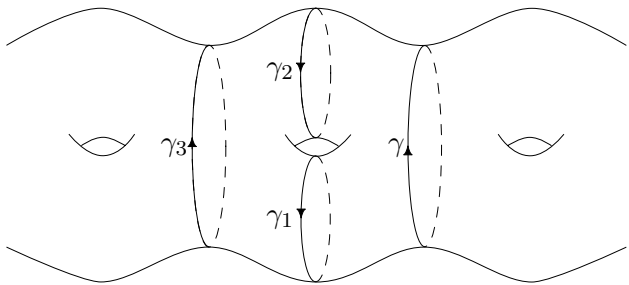

\begin{itemize}
	\item The curves $\gamma_1$ and $\gamma_2$ are non-separating and $\gamma$, $\gamma_1$ and $\gamma_2$ form the oriented boundary of a pair of pants.
	\item The curve $\gamma_3$ is separating and bounds a surface $S_3$ of genus $g_1 - 1$.
	\item The curves $\gamma_1$, $\gamma_2$ and $\gamma_3$ bound another pair of pants.
\end{itemize}

\input{Figures/FigureE.tex}

Now, choose two embedded curves $\alpha$ and $\beta$ as in Figure \ref{figure:FigureE}. We first perform surgeries of $\alpha$ with $\gamma_2$ and $\beta$ as indicated in the left-hand side of Figure \ref{figure:FigureE}. This yields an immersed curve $c$. If we perform surgeries of $\alpha$ with $\gamma_2$ and $\gamma$ as in the right-hand side of Figure \ref{figure:FigureE}, we obtain an immersed curve isotopic to $c$. Therefore,
\[ - \gamma_2 + \alpha + T = \gamma_1 + \alpha + \gamma, \]
so 
\[ \gamma + \gamma_1 + \gamma_2 = T. \]
The same argument yields
\[ T = -\gamma_3 - \gamma_1 - \gamma_2. \]
Hence,
\[ \gamma = \gamma_3 + 2T. \]
But $\gamma_3$ is the oriented boundary of a surface of genus $g_1 -1$, so
\[ [\gamma_3] = \chi(S_3) \cdot T. \]
Hence,
\[ \gamma = (\chi(S_3) + 2) \cdot T = \chi(S_1) T .\]
\end{proof}

We now have the following Lemma.

\begin{lemma}
\label{lemma:subgroupgeneratedbyseparatingcurves}
The restriction of $\mu$ to the subgroup $H$ of $\Gimm$ generated by separating curves is an isomorphism.	
\end{lemma}

\begin{proof}
Lemma \ref{lemma:classofanonseparatingcurve} implies that $T$ generates the group $H$. Moreover, $\mu(T) = -1$ implies that the order of $T$ is either infinite or a multiple of $\chi(S_g)$.

On the other hand, let $\gamma$ be the oriented boundary of a torus. Then, $\gamma^{-1}$ is the oriented boundary of a surface of genus $g-1$. Hence by Lemma \ref{lemma:classofanonseparatingcurve}, we have
\begin{eqnarray*}
 T & = & - \left(1-2(g-1) \right) T \\ 
& = & (-3 + 2g) T. 
\end{eqnarray*}
So $\chi(S_g)T = 0$. This concludes the proof.
\end{proof}

Finally, we compute the class of the Lickorish generator $\gamma_i$ (see \ref{Theo:Lickorishgeneratorsofthe MCG}) in $\Gimm$.

\begin{lemma}
\label{lemma:classofgammai}
Recall that we denoted by $\alpha_1, \ldots, \alpha_g$, $\beta_1, \ldots, \beta_g$ and $\gamma_1, \ldots, \gamma_{g-1}$.

Let $i \in \{ 1, \ldots , g-1 \}$. Then we have
\[	[\gamma_i]  = [ \alpha_{i+1} ] - [\alpha_i] - T . \]
\end{lemma}

\begin{proof}
\input{Figures/FigureN.tex}

By the change of coordinates principle, we can assume that $\alpha_i$, $\gamma_{i}$ and $\alpha_{i+1}^{-1}$ are as in Figure \ref{figure:FigureN}. We also fix two curves $\alpha$ and $\beta$ as in Figure \ref{figure:FigureN}.

The proof proceeds as in the proof of Lemma \ref{lemma:classofanonseparatingcurve}. We perform the surgeries indicated on the left to obtain a curve $c$ and the surgeries on the right to obtain a curve isotopic to $c$. So
\[ \alpha + \gamma_i - \alpha_{i+1} = \alpha + \beta - \alpha_i .\]
Moreover, we have $\beta = -T$, so
\[ \alpha_{i+1} - \alpha_i - \gamma_i = T . \] 
\end{proof}

With this preparation, we can pass to the
 
\subsubsection{Proof of Theorem \ref{Theo:immersedcobordismgroup}}
\begin{proof}
We use the action of the Mapping Class group of $S_g$ to conclude. Let $\gamma$ be a non-separating curve. By the change of coordinates principle (\cite[1.3]{FM012}), there is a product of Dehn twists about the $\alpha_i, \gamma_i$ and $\beta_i$ which maps $\gamma$ to $\alpha_1$. Hence, $\gamma$ lies in the group generated by the $\alpha_i, \beta_i$ and $\gamma_i$ by Proposition \ref{prop:ActionofaDehntwistonGimm}. By Lemma \ref{lemma:classofgammai}, $\gamma$ lies in the subgroup generated by the $\alpha_i, \beta_i$ and $T$.

Since any separating embedded curve is a multiple of $T$, we conclude that the group $\Gimm$ is generated by the $\alpha_i, \beta_i$ and $T$.

The map $\pi : \Gimm \to H_1(S_g,\Z)$ is surjective. Furthermore, we have $\pi \oplus \mu (T) = (0,-1)$. So the map $\pi \oplus \mu$ is surjective.

On the other hand, $\pi \oplus \mu$ is injective. To see this, let 
\[ x = \sum_{i=1}^g n_i \alpha_i + \sum_{i=1}^g m_i \beta_i + k T.\] such that $\pi(x) = 0$ and $\mu(x) = 0$. Take the image under $\pi$ to obtain 
\[ \sum_{i=1}^g n_i \alpha_i + \sum_{i=1}^g m_i \beta_i \]
in $H_1(S_g,\Z)$. So the $n_i$ and the $m_i$ are zero. Furthermore, $0 = \mu(x) = - k$, so $\chi(S_g) \vert k$. Hence, $x = k T = 0$.
\end{proof}

\section{Fukaya categories of surfaces}
\label{section:Fukaycategoriessurface}

There is a well-defined Fukaya category whose objects are the unobstructed immersed curves. Its construction follows Seidel's perturbation scheme \cite{Sei08}. Such a category has already been studied for exact manifolds with convex boundary by Alston and Bao \cite{AB14}. Abouzaid (\cite{ab07}) also constructed a pre-category with immersed objects using combinatorial methods.

In this section, we recall the main steps of the construction of $\Fuk(S_g)$, highlighting the parts that need special care due to the immersed setting. We do this in Subsections \ref{subsection:definition} and \ref{subsection:Preliminaries}.

In Subsection \ref{subsection:Propertiesfukcat}, we shall briefly explain why $\Fuk(S_g)$ recovers the pre-$A_\infty$ category defined in Abouzaid's paper.

\subsection{Preliminaries}
\label{subsection:Preliminaries}

\subsubsection*{Floer datum}
\label{subsubsection:Floerdatum}
Recall from definition \ref{defi:unobimmersion} that an unobstructed curve is an immersed curve with no triple points, transverse double points and which lifts to an embedding in the universal cover. For each ordered pair of unobstructed immersions $(\gamma_1, \gamma_2)$, we fix a \emph{Floer datum} $H_{\gamma_1,\gamma_2}$. This is a smooth, time-independent, hamiltonian $H_{\gamma_1,\gamma_2} : S_g \to \R$ such that $\gamma_1$ and $\phi^{-1}_{H_{\gamma_1,\gamma_2}}\left( \gamma_2 \right)$ are in general position. If $\gamma_1$ and $\gamma_2$ already are in general position, we make the choice $H_{\gamma_1,\gamma_2} = 0$.

We also fix a smooth complex structure $j$ on $S_g$ which is compatible with $\omega$.

\subsubsection*{Coherent perturbations}
\label{subsubsection:Coherentpertubations}
Recall from \cite[(9f)]{Sei08} that for $d \geqslant 2$, there is a compactified universal family of pointed disks
\[ \mathcal{S}^{d+1} \to \mathcal{R}^{d+1}. \]
We fix a coherent universal choice of strip-like ends (see \cite[9g]{Sei08}) for these families. We denote these ends by $\left( \varepsilon^i_r \right)_{0 \leqslant i \leqslant d, r \in \mathcal{S}^{d+1} } $.

Let $(\gamma_0, \ldots, \gamma_d)$ be a $d+1$ tuple of unobstructed curves. We fix a \emph{perturbation datum} for the family $\mathcal{S}^{d+1} \to \mathcal{R}^{d+1}$ labeled by the tuple $(\gamma_0, \ldots, \gamma_d )$. This is a family of one form 
\[ K_{\gamma_0, \ldots, \gamma_d} \in \Omega^1(\mathcal{S}, \mathcal{H}) \] 
with values in the space $\mathcal{H}$ of hamiltonians on $S_g$. Additionally, we assume that for any $r \in \mathcal{S}^{d+1}$, 
\[ K_{\vert T \partial \pi^{-1}(r)} = 0. \] 
We require that this family satisfies the following hypothesis.

\begin{center}
 (H) : If the curves $\gamma_0, \ldots, \gamma_d$ are in general position, then 
 	\[ K_{\gamma_0, \ldots, \gamma_d} = 0.\]
\end{center}

We have the following proposition.
\begin{prop}
\label{prop:existenceofacoherentchoiceofperturbation}
There is a coherent choice of perturbation datum\footnote{See \cite[section (9i)]{Sei08} for the definition} which satisfies the hypothesis $(H)$.
\end{prop}

\begin{proof}
We start with $K_{\gamma_0,\gamma_1,\gamma_2} = 0$ for every $3$-uple of two by two transverse lagrangians. Then we use the induction process described in \cite[(9i)]{Sei08} to obtain a coherent pertubation perturbation datum.

For every $d+1$-tuple $(\gamma_0, \ldots, \gamma_d)$ of two by two transverse unobstructed curves, any $m$-tuple of the form $(\gamma_{i_1}, \ldots, \gamma_{i_m})$ (with $m < d$) consists of two by two transverse unobstructed curves. Therefore the gluing induction process provides a perturbation with $K_{\gamma_0, \ldots, \gamma_d} = 0$. 
\end{proof}

\subsubsection*{Definition and regularity of the relevant moduli spaces}
\label{subsubsection:definition/regularitymodulispaces}

We introduce the moduli spaces of (perturbed) holomorphic curves that we consider throughout this section.

First, let $Z = \R \times [0,1]$ be the standard strip with coordinates $s$ and $t$ and equipped with the standard complex structure. 

Let $\gamma_0$ and $\gamma_1$ be two unobstructed curves. Then the set
 $\mathcal{P}_{H_{\gamma_0,\gamma_1}}$ of hamiltonian chords from $\gamma_0$ to $\gamma_1$ is finite. We choose two such chords $c_-$ and $c_+$.

We consider continuous strips $u : Z \to S_g$ from $c_-$ to $c_+$ satisfying the following conditions.
\begin{enumerate}[label=(\roman*)]
	\item There is a continuous lift $u^- : \R \times \{0\} \to S^1$ such that $u(s,0) = \gamma_0 \circ u_0 $	.
	\item There is a continuous lift $u^+ : \R \times \{0\} \to S^1$ such that $u(s,1) = \gamma_1 \circ u_1 $	.
	\item The map $u$ converges uniformly to $c_-$ and $c_+$ when $s$ goes to infinity:
		\[ \lim_{s \to -\infty} u(s,\cdot) = c_- , \ \lim_{s \to + \infty} u(s,\cdot) = c_+. \]
\end{enumerate}

Let $u_0$ and $u_1$ be two continuous strips with lifts given by $u_0^+, u_0^-$ and $u_1^+, u_1^-$ respectively. We say that $u_0$ and $u_1$ are \emph{homotopic} if there are
\begin{itemize}
	\item a continuous family $(v_t)_{t \in [0,1]}$ of maps $Z \to S_g$,
	\item continuous families $\left(v_t^{\pm} \right)_{t \in [0,1]}$ of maps $\R \to S^1$, 	
\end{itemize}
such that
\begin{itemize}
	\item for each $t \in [0,1]$, $v_t$ is a continuous strip from $c_-$ to $c_+$ with continuous lifts given by $v_t^{\pm}$,
	\item we have $\left(u_0,u_0^-,u_0^+ \right) = \left(v_0,v_0^-,v_0^+ \right)$ and $\left(u_1,u_1^-,u_1^+ \right) = \left( v_1,v_1^-,v_1^+ \right)$.
\end{itemize}
We fix such a homotopy class $A$.

\begin{defi}
We let $\widetilde{\mathcal{M}} \left( c_-,c_+,A \right)$ be the set of \emph{Floer strips} from $c_-$ to $c_+$ in the homotopy class $A$. 

A map $u : Z \to S_g$ is an element of $\widetilde{\mathcal{M}}\left(c_-,c_+,A\right)$ if it satisfies the conditions  (i), (ii), (iii) above and the \emph{Floer equation}
\[ \frac{\partial u}{\partial s} + j \left( \frac{\partial u}{\partial t} - X_{H_{\gamma_0,\gamma_1}}(u) \right ) = 0. \]
This set admits a natural $\R$-action and we let
\[ \mathcal{M} \left (c_-,c_+,A \right) = \widetilde{\mathcal{M}}\left(c_-,c_+,A\right) / \R. \]
\end{defi}

Let $\gamma_0, \ldots, \gamma_d$ be a $d+1$-uple of unobstructed curves, $c_0 \in \mathcal{P}_{H_{\gamma_0,\gamma_d}}$ and $c_i \in \mathcal{P}_{H_{\gamma_{i-1},\gamma_i}}$ for $1 \leqslant i \leqslant d$. We also consider \emph{continuous polygons} with boundary conditions at $\gamma_0, \ldots, \gamma_d$.

To define these, fix a disk with $d+1$-marked points $s = \pi^{-1}(r)$ with $r \in \mathcal{R}^{d+1}$. A \emph{continuous polygon} is a continuous map $u : s \to S_g$ such that
\begin{enumerate}[label=(\roman*)]
	\item For each arc $C_i \subset \partial s$ between the $i$-th and $i+1$-th punctures, there is a continuous map $u_i : C_i \to S^1$ such that $u_{\vert C_i} = \gamma_i \circ u_i$.
	\item The map $u$ converges uniformly to $c_i$ on the $i$-th strip-like end,
	\[	\lim_{s \to - \infty} u \circ \varepsilon^0_r (s,\cdot) = c_0, \ \lim_{s \to + \infty} u \circ \varepsilon^i_r (s,\cdot) = c_i .\] 
\end{enumerate}
Let $u^0$ and $u^1$ be two continuous polygons with lifts given by $u_i^0$ and $u_i^1$ for $i = 0 \ldots d$ respectively. We say that $u^0$ and $u^1$ are \emph{homotopic} if there are
\begin{itemize}
	\item continuous families $(v^t)_{t \in [0,1]}$ of maps $s \to S_g$ for a fixed $s \in \pi^{-1}(r)$ with $r \in \mathcal{R}^{d+1}$,
	\item continuous families $(v^t_i)_{t \in [0,1]}$ of maps $C_i \to S^1$ for $i \in \{0, \ldots, d \}$,	
\end{itemize}
such that
\begin{itemize}
	\item for each $t \in [0,1]$, $v_t$ is continuous polygon with boundary lifts given by the $v_t^i$ for $i = 0 \ldots d$,
	\item we have $v_0 = u_0$ (resp. $v_1 = u_1$) and $v_0^i = u_0^i$ (resp. $v_0^1 = u_0^1$) for $i = 0 \ldots d$.
\end{itemize}
We fix such a homotopy class $B$.

\begin{defi}
We let $\mathcal{M} \left( c_0,\ldots,c_d,B \right)$ be the set of \emph{holomorphic polygons} in the homotopy class $A$. 

A map $u : s \to S_g$, with $s \in \pi^{-1}(r)$ for some $r \in \mathcal{R}^{d+1}$, is an element of the moduli space 
\[ \mathcal{M} \left( c_0,\ldots,c_d,B \right) \] 
if and only if it satisfies the conditions (i),(ii),(iii) above and the following equation
\[ \left(du - X_{K_{\gamma_0, \ldots, \gamma_d}} \right)^{(0,1)} = 0. \]
\end{defi}

Standard regularity arguments imply that for a generic choice of Floer and perturbation datum, these moduli spaces are regular. See for instance \cite[Section 5]{AB14} for a write up in the case of immersions.

However, we would like to keep perturbation data which verify the hypothesis $(H)$. The following lemma, due to Seidel, makes this possible. The proof contained in \cite{Sei08} goes through as stated in the book.

Let $u: (s, \partial s) \to (M,i(L))$ be a Floer strip or a holomorphic polygon. One can linearize the Cauchy-Riemann equation at $u$ to obtain an \emph{extended Cauchy-Riemann operator}
\[ D_{s,u} : T_s \mathcal{R} \times W^{1,p} \left(u^*TM,u^*TL \right) \to L^p \left( \Lambda^{0,1}T^*s \otimes u^*TM \right), \]
from suitable Sobolev completions of the space of sections of $u^*TM$. (cf \cite[(8i)]{Sei08}). We say that $u$ is \emph{regular} if this operator is surjective. In particular, this implies that the set of solutions is a manifold near $u$.

In particular notice that $D_{s,u}$ takes into account the \emph{variations} of the domain.

\begin{lemma}[Automatic regularity, \cite{Sei08}, Lemma 13.2]
In the above setting, assume that $u$ is either a non-constant Floer strip or a non-constant holomorphic polygon. Then it is automatically regular, meaning that the extended Cauchy-Riemann operator
\[ D_{s,u} : T_s \mathcal{R} \times W^{1,p} \left(u^*TM,u^*TL \right) \to L^p \left( \Lambda^{0,1}T^*s \otimes u^*TM \right), \]
is surjective.
\end{lemma}

From now on, we will fix a choice of Floer and perturbation datum which satisfy hypothesis $(H)$.

\subsubsection*{Indices}
\label{subsubsection:Indices}
Let $u$ be an element of one of the moduli spaces $\mathcal{M}\left(c_-,c_+,A \right)$ or $\mathcal{M} \left( c_0,\ldots,c_d,A \right)$. We call $\Ind(u)$ the virtual-dimension of its moduli space. Using (for instance) the index formula of \cite{AJ10}, it is easily seen to depend only on the homotopy class $A$ of $u$. Therefore, in what follows, we will denote this index by $\Ind(A)$.

\subsubsection{Spin structures and signs}
To take care of signs issues, we need some additional datum which we explain now. 

We need to equip each unobstructed curve $\gamma : S^1 \to S_g$ with a \emph{Spin structure}. Since the space of oriented orthonormal frames of a tangent space $T_x S^1$ has a unique point, a choice of Spin structure is a choice of a double covering of $S^1$. We choose the Spin structure given by the nontrivial double covering. This is called the \emph{bounding Spin structure} (see \cite[II.1]{LM89}).

Let $\gamma_1$ and $\gamma_2$ be two transverse unobstructed curves and $x$ an intersection point. We let
\[ \lambda_{\gamma_1,\gamma_2,x} : [0,1] \to \mathcal{G}(TS_g), \]
be the path represented in Figure \ref{figure:FigureO}.

\begin{figure}
\definecolor{cfc0d1b}{RGB}{252,13,27}
\definecolor{cff0000}{RGB}{255,0,0}

\begin{tikzpicture}[y=0.80pt, x=0.80pt, yscale=-1.000000, xscale=1.000000, inner sep=0pt, outer sep=0pt]
  \path[draw=cfc0d1b,line join=round,line cap=round,miter limit=4.00,fill
    opacity=0.357,line width=0.268pt] (203.0407,125.5601) -- (203.0407,261.9050);
  \path[draw=black,line join=round,line cap=round,miter limit=4.00,fill
    opacity=0.357,line width=0.268pt] (141.5464,193.7325) -- (264.5349,193.7325);
  \path[draw=cfc0d1b,line join=round,line cap=round,miter limit=4.00,fill
    opacity=0.357,line width=0.268pt] (356.5838,125.5601) -- (356.5838,261.9050);
  \path[draw=black,line join=round,line cap=round,miter limit=4.00,fill
    opacity=0.357,line width=0.268pt] (295.0896,193.7325) -- (418.0781,193.7325);
  \path[draw=black,line join=miter,line cap=butt,miter limit=4.00,even odd
    rule,line width=0.300pt] (238.2143,190.2193) .. controls (238.2143,190.2193)
    and (238.2143,190.2193) .. (229.6429,178.2550) .. controls (221.0714,166.2908)
    and (191.0742,171.4693) .. (181.2500,171.4693) .. controls (171.4258,171.4693)
    and (170.7143,166.2425) .. (170.7143,161.3801) .. controls (170.7143,156.5176)
    and (173.9724,151.2908) .. (181.2500,151.2908) -- (195.0000,151.2908);
  \path[draw=black,line join=round,line cap=round,miter limit=4.00,fill
    opacity=0.067,line width=0.300pt] (392.8571,198.6122)arc(0.000:90.000:32.321);
  \path[draw=black,fill=black,line join=round,line cap=round,miter limit=4.00,line
    width=0.268pt] (191.9871,149.4682) .. controls (192.7006,150.5830) and
    (194.3819,151.0233) .. (196.2273,151.4694) .. controls (194.3221,151.9298) and
    (192.6760,152.3941) .. (191.9871,153.4706) -- cycle;
  \path[draw=black,fill=black,line join=round,line cap=round,miter limit=4.00,line
    width=0.268pt] (363.7795,228.8260) .. controls (363.0660,229.9408) and
    (361.3847,230.3812) .. (359.5393,230.8272) .. controls (361.4444,231.2877) and
    (363.0905,231.7519) .. (363.7795,232.8284) -- cycle;
  \path[draw=cff0000,fill=cff0000,line join=round,line cap=round,miter
    limit=4.00,line width=0.268pt] (201.0711,153.5895) .. controls
    (202.1858,152.8760) and (202.6262,151.1946) .. (203.0723,149.3493) .. controls
    (203.5327,151.2544) and (203.9970,152.9005) .. (205.0735,153.5895) -- cycle;
  \path[draw=black,fill=black,line join=round,line cap=round,miter limit=4.00,line
    width=0.268pt] (251.5861,191.7683) .. controls (252.2996,192.8831) and
    (253.9809,193.3235) .. (255.8263,193.7695) .. controls (253.9211,194.2300) and
    (252.2750,194.6942) .. (251.5861,195.7707) -- cycle;
  \path[draw=cff0000,fill=cff0000,line join=round,line cap=round,miter
    limit=4.00,line width=0.268pt] (354.6425,256.4255) .. controls
    (355.7573,257.1390) and (356.1977,258.8203) .. (356.6437,260.6657) .. controls
    (357.1042,258.7606) and (357.5684,257.1145) .. (358.6449,256.4255) -- cycle;
  \path[draw=black,fill=black,line join=round,line cap=round,miter limit=4.00,line
    width=0.268pt] (413.5503,191.7683) .. controls (414.2639,192.8831) and
    (415.9452,193.3235) .. (417.7905,193.7695) .. controls (415.8854,194.2300) and
    (414.2393,194.6942) .. (413.5503,195.7707) -- cycle;
    
  \draw (207,197) node {\tiny $x$};  
  \draw (360,197) node {\tiny $x$};
  \draw (252,199) node {\tiny $\gamma_1$};  
  \draw (405,199) node {\tiny $\gamma_1$};  
  \draw (363,129) node {\tiny $\gamma_2$};
  \draw (210,129) node {\tiny $\gamma_2$};

\end{tikzpicture}
\caption{The path $\lambda_{\gamma_1,\gamma_2,x}$ when $x$ is of degree $1$ (left) and of degree $0$ (right)}
\label{figure:FigureO}
\end{figure}
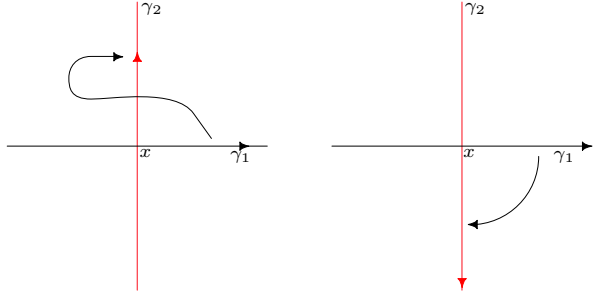 

If $\gamma_1$ and $\gamma_2$ are not transverse, we choose a hamiltonian chord $c \in \mathcal{P}_{H_{\gamma_1,\gamma_2}}$. This corresponds to a unique intersection point $x \in \gamma_1 \cap \phi^{-1}_{H_{\gamma_1,\gamma_2}}\left( \gamma_2 \right)$. We define, for $t \in [0,1]$, the vector space $\lambda_{\gamma_1,\gamma_2,c} (t) \in \mathcal{G} (T_c(t) S_g)$ by
\[ \lambda_{\gamma_1,\gamma_2,c} (t)  = d\phi^t_{H_{\gamma_1,\gamma_2}} (\lambda_{x, \gamma_1, \phi^{-1}_{H_{\gamma_1,\gamma_2}}\left( \gamma_2 \right)}(t)). \]
For each Hamiltonian chord $c \in \mathcal{P}_{H_{\gamma_1,\gamma_2}}$, define a one dimensional real vector space $o(c)$ as follows. Consider the Poincaré half plane $\HP = \ens{z = x+ iy \in \C}{y \leqslant 0}$. Equip this with the incoming strip-like end 
\[ \varepsilon :
\begin{array}{ccc} 
\R  \times [0,1] & \to & \HP \\
(s,t) & \mapsto  &-e^{-\pi(s+it)}	
\end{array}.
\] 
As explained, for instance, in \cite{Sei08}, there is a complex bundle pair $(E,F)$ associated to $\lambda_{\gamma_1,\gamma_2,c}$. In particular, there is a Cauchy-Riemann operator $D$ associated to it.

The \emph{orientation line} of $c$ is the real one-dimensional vector space
\[ o(c) := \det(D) . \]
We choose once and for all an orientation of each vector space $o(c)$.

Further, assume that $\gamma_1$ and $\gamma_2$ are unobstructed curves. Let $u \in \mathcal{M}\left(c_-,c_+,A \right)$ be a holomorphic strip which satisfies $\Ind(A) = 0$. Since $\gamma_1$ and $\gamma_2$ are equipped with spin structure, gluing induces an isomorphism
\[ \Lambda T_u \widetilde{\mathcal{M}} \left(c_-,c_+,A \right) \tilde{\rightarrow} o(c_-)^{\vee} \otimes o(c_+). \]
 We orient the left side by the vector field generated by the $\R$ action. On the other hand, the right hand sign inherits an orientation from $o(c_)$ and $o(c_+)$. The difference between these two orientations yields a sign
\[ \Sign(u) \in \{ -1,1 \}. \]

Similarly let $v \in \mathcal{M} \left ( c_0, \ldots, c_d, A \right ) $ be a holomorphic polygon which satisfies $\Ind(A) = 0$. Now there is an isomorphism
\[ \Lambda T_u \mathcal{M} \left ( c_0, \ldots, c_d, A \right ) \tilde{\rightarrow} o(c_0)^{\vee} \otimes o(c_1) \otimes \ldots \otimes o(c_d). \]
But the left-hand side is naturally oriented as the determinant of a $0$-dimensional vector space. Comparing orientations yields a sign
\[ \Sign (v) \in \{ -1, 1 \}. \]
\subsubsection*{Gromov compactness}
In order to define the $A_\infty$ operations for our category, we need to describe a Gromov-type compactification of the moduli spaces introduced above. Gromov compactness for holomorphic curves with immersed boundaries has already been considered in Ivashkovich and Shevchishin's paper \cite{IS02}.

Here, our situation is slightly different since we considered the solutions of a perturbed Cauchy-Riemann equation. However, the relevant analysis is worked-out in Alston and Bao's article \cite[Proposition 4.4]{AB14}. We can summarize their results in our setting as follows.

As is usual by now, we let $\gamma_0, \gamma_1$ be unobstructed curves in $S_g$ and $c_-, c_+$ be Hamiltonian chords from $\gamma_0$ to $\gamma_1$. 

\begin{prop}[Gromov compactness for strips]
\label{prop:Gromovcompactnessforstrips}
\begin{enumerate}
\item Let $A$ be a homotopy class of strips satisfying $\Ind(A) = 1$. Then the topological space
 \[ \mathcal{M}\left(c_-,c_+,A \right) \]
 is a compact, $0$-dimensional manifold.
\item Assume that $A$ is a homotopy class of strips satisfying $\Ind(A) = 2$. Then the topological space
\[ \overline{\mathcal{M}}(c_-,c_+,A) \]
admits a natural compactification $\overline{\mathcal{M}} \left(c_-,c_+,A \right)$ given by
\[ \overline{\mathcal{M}} \left(c_-,c_+,A \right) := 
\coprod_{c \in \mathcal{P}_{H_{\gamma_-,\gamma_+}} } \mathcal{M} \left (c_-,c,A \right) \times \mathcal{M} \left (c,c_+,A \right)  
\bigcup \mathcal{M} \left (c_-,c_+,A \right). \]
The space $\overline{\mathcal{M}} \left(c_-,c_+,A \right)$ has a natural structure of $1$-dimensional manifold with boundary
\[ 
 \partial \overline{\mathcal{M}} \left(c_-,c_+,A \right) = 
 \coprod_{c \in \mathcal{P}_{H_{\gamma_-,\gamma_+}} } \mathcal{M} \left (c_-,c,A \right) \times \mathcal{M} \left (c,c_+,A \right).
 \]
\end{enumerate} 
\end{prop}

\begin{proof}
Let $(u_n)$ be a sequence of holomorphic strips in the homotopy class $A$. It is easy to check that their energy is finite so that we can apply \cite[Proposition 4.4]{AB14}. Hence, there is a subsequence $(u_{n_k})$ which converges in Gromov's sense to a set of broken holomorphic strips, polygons, spheres, and disks. Moreover, if there is one polygon in this decomposition, there must be a polygon with one corner.

Notice that there cannot be holomorphic disks or holomorphic polygon with one corner with boundary condition on $\gamma_+$ or $\gamma_-$ since both of these are unobstructed. Moreover, there are no holomorphic spheres. Hence the above set can only consist of strips which are all regular. A dimension counting argument finishes the proof.
\end{proof}

Similarly, let $\gamma_0, \ldots, \gamma_d$ (with $d \geqslant 2$) be unobstructed curves, $c_0$ be in $\mathcal{P}_{H_{\gamma_0,\gamma_d}}$ and $c_i$ be in $\mathcal{P}_{H_{\gamma_i,\gamma_{i+1}}}$ for $1 \leqslant i \leqslant d$. We fix a homotopy class $B$ of holomorphic polygons with corners at the $\gamma_i$.

\begin{prop}
\label{prop:Gromovcompactnessforstrips}
\begin{enumerate}
	\item Assume that $\Ind(B) = 0$, then the space $\mathcal{M} \left( c_0,\ldots,c_d,B \right)$ is compact.
	\item Assume that $\Ind(B) = 1$, then the topological space
	\[\mathcal{M} \left( c_0,\ldots,c_d,B \right) \] 
	admits a natural compactification $\overline{\mathcal{M}} \left( c_0,\ldots,c_d,B \right)$
	\begin{equation}
	\begin{split} 
	& \overline{\mathcal{M}} \left( c_0,\ldots,c_d,B \right) := \\
	& \mathcal{M} \left( c_0,\ldots,c_d,B \right) \coprod \\
	& \coprod_{ \substack {1 \leqslant i \leqslant d, \tilde{c}_{i+1} \in \mathcal{P}_{H_{\gamma_{i+1}, \gamma_{i+2}}}, \\ \Ind(B_k) =0    } }
	\mathcal{M} \left( c_0,\ldots,\tilde{c}_{i+1}, \ldots c_d,B_1 \right) \times
	\mathcal{M} \left( \tilde{c}_{i+1}, c_{i+1}, \ldots,c_d,B_2 \right).
	\end{split}
	\end{equation}
The space $\overline{\mathcal{M}} \left( c_0,\ldots,c_d,B \right)$ has a natural structure of $1$-dimensional manifold with boundary
\[ \begin{split}
& \partial \overline{\mathcal{M}} \left( c_0,\ldots,c_d,B \right) := \\
& \coprod_{ \substack {1 \leqslant i \leqslant d, \tilde{c}_{i+1} \in \mathcal{P}_{H_{\gamma_{i+1}, \gamma_{i+2}}}, \\ \Ind(B_k) =0    } }
	\mathcal{M} \left( c_0,\ldots,\tilde{c}_{i+1}, \ldots c_d,B_1 \right) \times
	\mathcal{M} \left( \tilde{c}_{i+1}, c_{i+1}, \ldots,c_d,B_2 \right).
	\end{split} \]
\end{enumerate}	
\end{prop}

\subsection{Definition}
\label{subsection:definition}
We now have all the ingredients to define a $A_\infty$ category $\Fuk(S_g)$. The coefficients are taken over the \emph{Novikov field}
\[ \Lambda : = \ens{\sum_{i=0}^{+\infty} a_i T^{\lambda_i}}{a_i \in \R, \lambda_i \in \R, \lambda_i \to +\infty} . \]

The objects of $\Fuk(S_g)$ are unobstructed curves. Given two unobstructed curves $\gamma_1, \gamma_2$, their morphism space is the $\Z / 2$-graded $\Lambda$-vector space generated by Hamiltonian chords between these
\[ \Hom^i_{\Fuk(S_g)}(\gamma_1,\gamma_2) = \bigoplus_{\substack{c \in \mathcal{P}_H \\ \norm{c} = i}} \Lambda \cdot c. \]
The $A_\infty$ operations are defined as follows
\[ \mu^d (c_1, \ldots, c_d) = (-1)^{ \bullet } \sum_{c_0, A} \sum_{u \in \mathcal{M}(c_0, \ldots, c_d,A) } \Sign(u) \cdot T^{\omega(A)} \cdot c_0.\] 
Here the sum is over the Hamiltonian chords $c_0 \in \mathcal{P}_{H_{\gamma_1,\gamma_2}}$ and over the homotopy classes $A$ of polygons such that $\Ind(A) = 0$. The sign $\bullet$ is given by
\[ \bullet = \sum_{k=1}^d k \norm{c_k} . \]
Proposition \ref{prop:Gromovcompactnessforstrips} implies that the operations $(\mu^d)_{d \geqslant 2}$ satisfy the $A_\infty$ relation modulo $2$. To see that these are satisfied over $\R$, one has to use that gluing is compatible with the isomorphisms. This is done in \cite[Sections (12b) and (12g)]{Sei08}.
 
The $A_\infty$-category $\Fuk(S_g)$ admits a triangulated envelope : this is the smallest triangulated $A_\infty$-category generated by $\Fuk(S_g)$. We call its $0$-th degree cohomology the \emph{derived category} of $\Fuk(S_g)$ and denote it by $\DFuk(S_g)$. We refer to Seidel's book \cite[(3j)]{Sei08} for more details.

\subsection{Properties of the Fukaya category}
\label{subsection:Propertiesfukcat}
The following theorem also immediately follows from the recipe presented in Seidel's book.

\begin{prop}
The $A_\infty$-category $\Fuk(S_g)$ is homologically unital and independent of the choice of perturbation datum and almost compatible complex structure.

Moreover, two Hamiltonian isotopic curves are quasi-isomorphic. 	
\end{prop}

Since we chose our perturbation datum $(H)$, there is a combinatorial description of the operations of the category $\Fuk(S_g)$ with boundaries in a tuple $(\gamma_0, \ldots, \gamma_d)$ in general position. 

We let $c_0 \in \mathcal{P}_{H_{\gamma_0,\gamma_d}}$, $c_i \in \mathcal{P}_{H_{\gamma_i,\gamma_{i+1}}}$ for $1 \leqslant i \leqslant d$. We choose $s \in \mathcal{S}^{d+1}$ and label it by $(\gamma_0, \ldots, \gamma_d)$ and call $r$ its pre-image by $\pi : \mathcal{R}^{d+1} \to \mathcal{S}^{d+1}$. We fix a homotopy class $A$ of polygons with corners at $c_0, \ldots, c_d$ such that $\Ind(A) = 0$.

We let $\tilde{\Delta} (c_0, \ldots, c_d,A)$ be the set of orientation preserving immersions (up to the boundary) $f : s \to S_g$ with $s \in \pi^{-1}(r)$ such that
\begin{itemize}
	\item the map $f$ is a polygon with corners at $c_0, \ldots,c_d$,
	\item each corner of $f$ is convex (see Figure \ref{figure:rigidpolygon}).	
\end{itemize}

Here, a corner $x \in s$ of $f$ is convex if a neighborhood of $x$ has a convex image (see Figure \ref{figure:rigidpolygon}). 
\begin{figure}
\captionsetup{justification=centering,margin=2cm}

\begin{tikzpicture}[y=0.80pt, x=0.80pt, yscale=-1.000000, xscale=1.000000, inner sep=0pt, outer sep=0pt]
  \path[draw=black,fill=black,line join=round,line cap=round,miter limit=4.00,fill
    opacity=0.156,line width=0.400pt] (272.7622,192.8957) .. controls
    (272.7622,192.8957) and (256.5009,234.8718) .. (246.2328,244.9562) .. controls
    (235.9647,255.0407) and (201.6269,252.7053) .. (201.6269,252.7053) --
    (137.3439,254.7959) -- (129.2323,214.1995) .. controls (129.2323,214.1995) and
    (133.3431,209.4731) .. (136.7610,196.1347) .. controls (140.1788,182.7964) and
    (143.4823,133.4491) .. (143.4823,133.4491) .. controls (143.4823,133.4491) and
    (169.6403,136.5829) .. (183.3733,134.7477) .. controls (197.1063,132.9125) and
    (224.6840,122.0485) .. (224.6840,122.0485) .. controls (224.6840,122.0485) and
    (217.2451,151.6593) .. (225.5769,165.7348) .. controls (233.9086,179.8104) and
    (272.7622,192.8957) .. (272.7622,192.8957) -- cycle;
  \path[draw=black,line join=miter,line cap=butt,miter limit=4.00,even odd
    rule,line width=0.400pt] (272.6149,193.1012) -- (284.4842,162.2915);
  \path[draw=black,line join=miter,line cap=butt,miter limit=4.00,even odd
    rule,line width=0.400pt] (137.2545,254.7205) -- (107.2024,255.8569);
  \path[draw=black,line join=miter,line cap=butt,miter limit=4.00,even odd
    rule,line width=0.400pt] (137.3808,254.4680) -- (140.1587,267.7262) --
    (140.1587,267.7262);
  \path[draw=black,line join=miter,line cap=butt,miter limit=4.00,even odd
    rule,line width=0.400pt] (129.2364,214.3775) -- (127.6580,191.7754);
  \path[draw=black,line join=miter,line cap=butt,miter limit=4.00,even odd
    rule,line width=0.400pt] (143.5048,133.4706) -- (126.3322,131.4188);
  \path[draw=black,line join=miter,line cap=butt,miter limit=4.00,even odd
    rule,line width=0.400pt] (143.5048,133.5338) -- (144.1993,114.7197);
  \path[draw=black,line join=miter,line cap=butt,miter limit=4.00,even odd
    rule,line width=0.400pt] (224.4433,122.0748) -- (233.0927,117.7817);
  \path[draw=black,line join=miter,line cap=butt,miter limit=4.00,even odd
    rule,line width=0.400pt] (224.5695,122.2958) -- (226.2742,113.9305);
  \path[draw=black,line join=miter,line cap=butt,miter limit=4.00,even odd
    rule,line width=0.400pt] (129.4643,213.8801) .. controls (122.9167,221.0638)
    and (118.1548,224.5348) .. (112.6786,226.7372);
  \path[draw=black,line join=miter,line cap=butt,miter limit=4.00,even odd
    rule,line width=0.400pt] (272.7412,192.8487) -- (290.4189,198.4045);

\end{tikzpicture}

\caption{An immersed polygon with convex corners}
\label{figure:rigidpolygon}
\end{figure}
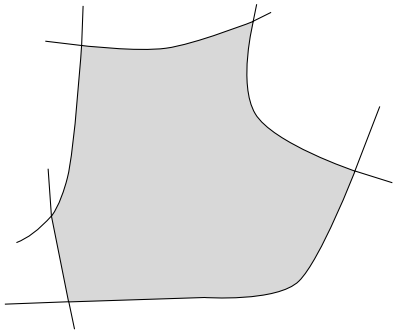

We let $\Delta(c_0, \ldots, c_d, A)$ be the quotient of $\tilde{\Delta} (c_0, \ldots, c_d,A)$ by the group of diffeomorphisms of $r$ which preserve the marked points.

\begin{prop}
\label{prop:Fuk(S_g)iscombinatorial}
Each holomorphic polygon $u \in \overline{\mathcal{M}} \left( c_0, \ldots, c_d , A\right)$ is (up to reparameterization) an element of $\tilde{\Delta}(c_0,\ldots,c_d)$.

Moreover, the inclusion
\[ \mathcal{M} \left( c_0, \ldots, c_d, A \right ) \hookrightarrow \Delta \left(c_0, \ldots, c_d, A \right) \]
is bijective.
\end{prop}

\begin{proof}
This result is well-known (see \cite[(13b)]{Sei08}, \cite{ENS02}, \cite{RSdS14}). Let us quickly recall the idea of the proof.

Let $u \in \overline{\mathcal{M}} \left( c_0, \ldots, c_d , A\right)$. If $u$ has an interior branch point, there is a two-dimensional continuous family in $\overline{\mathcal{M}} \left( c_0, \ldots, c_d , A\right)$ (see the proof of \cite[Proposition 7.8]{ENS02}). But $\dim(\overline{\mathcal{M}} \left( c_0, \ldots, c_d , A\right) ) \leqslant 1$, a contradiction. Similarly, if there is a branch point on the boundary, there is a contribution of $1$ to the dimension. So $u$ is an immersion up to the boundary. Moreover, it is easy to see that $u$ cannot have non convex corners at the boundary.

Now if $u \in \tilde{\Delta}(c_0, \ldots, c_d, A)$, an easy application of the uniformization theorem shows that $u$ can be reparameterized to a holomorphic curve. Hence, the inclusion $\mathcal{M} \left( c_0, \ldots, c_d, A \right ) \hookrightarrow \Delta \left(c_0, \ldots, c_d, A \right)$ is surjective.

If $u_1$ and $u_2$ are such that $u_1 \circ \phi = u_2$ with $\phi$ a diffeomorphism $r \to r$, $\phi$ is holomorphic since $u_1$ and $u_2$ are immersions. Hence, the inclusion $\mathcal{M} \left( c_0, \ldots, c_d, A \right ) \hookrightarrow \Delta \left(c_0, \ldots, c_d, A \right)$ is injective.
\end{proof}

In particular, one can define a $A_\infty$ pre-category $\Fuk_{\text{comb}}(S_g)$ whose objects are unobstructed immersed curves, whose morphisms spaces are given by the Floer complexes and whose higher operations are given by a count of elements of $\Delta \left(c_0, \ldots, c_d, A \right)$. This is done, using combinatorial arguments, in Abouzaid's paper (\cite{ab07}). We conclude that there is a pre-$A_\infty$ quasi-isomorphism
\[  \Fuk_{\text{comb}} (S_g) \hookrightarrow \Fuk(S_g).\]

\section{Immersed Lagrangian cobordisms and iterated cones}
\label{section:immersedcobordismandcones}
In this section, we study immersed Lagrangians which are well-behaved for Floer theory. We have already seen that the main obstruction to this is the existence of teardrops, that is polygons with one corner points. Our object of interest are cobordisms which do not these for a particular choice of complex structure.

In what follows, we will consider only compatible almost complex structures on $\C \times S_g$ such that the projection on the first factor $\pi_\C : \C \times S_g $ is holomorphic. If this holds, we say that the almost complex structure is \emph{adapted}.

\begin{defi}
\label{defi:unobstructedLagcob}
Let $(\gamma_1, \ldots, \gamma_n)$ and $(\tilde{\gamma}_1, \ldots, \tilde{\gamma}_m)$ be \emph{embedded} curves in $S_g$. An \emph{unobstructed Lagrangian cobordism} from $(\gamma_1, \ldots, \gamma_N)$ to $(\tilde{\gamma}_1, \ldots, \tilde{\gamma}_m)$ is an  oriented immersed Lagrangian cobordism 
\[ V : \left( \gamma_1, \ldots, \gamma_N \right) \leadsto \left( \tilde{\gamma}_1, \ldots, \tilde{\gamma}_m \right)  \]
which satisfies the following conditions.
\begin{enumerate}[label=(\roman*)]
	\item The immersion $V$ has no triple points, and all its double points are transverse.
	\item There are no topological teardrops with boundary on $V$.
\end{enumerate}
\end{defi}

Biran and Cornea proved that, in the monotone setting, an embedded, monotone Lagrangian cobordism induces a cone relation between its end in the Fukaya category (\cite{BC14}). This result still holds for unobstructed Lagrangian cobordisms.

\begin{Theo}
\label{Theo:DecompositionforLagrangiancobordisms}
Let 
\[ V : (\gamma_1, \ldots, \gamma_n) \leadsto \gamma \] 
be an unobstructed Lagrangian cobordism. Then, there is an isomorphism in $\DFuk(S_g)$
\[ \gamma \simeq \Cone \left( \gamma_1[1] \to \Cone \left( \gamma_2[1] \ldots \to \Cone \left(\gamma_{n-1}[-1] \to \gamma_n \right) \right) \right) . \]
\end{Theo}

In our setting, the proof is the same as \cite[Theorem 4.1]{Hau15}.

There are two natural corollaries to this result, which we now give.

\begin{coro}
\label{coro:shiftequalstooppositeorientation}
Let $\gamma$ be an unobstructed curve. In $\DFuk(S_g)$ we have the isomorphism
\[ \gamma[1] \simeq \gamma^{-1}. \] 
\end{coro}

\begin{proof}
Consider a properly embedded path $\alpha : \R \to \C$ such that
\begin{itemize}
	\item for $s < 0$, we have $\alpha(s) = (s,0)$,
	\item for $s > 1$, we have $\alpha(s) = (1-s, 1)$,
	\item for $s \in [0,1]$, the derivative $y'$ satisfies $y'(s) > 0$.
\end{itemize}
The immersed manifold 
\[ \begin{array}{ccc}
 \R \times S^1 & \to & \C \times S_g \\
 (s,t) & \mapsto & (\alpha(s), \gamma(t)) 	
 \end{array}. \]
 is a Lagrangian cobordism $(\gamma,\gamma^{-1}) \leadsto \emptyset$. Therefore, Theorem \ref{Theo:DecompositionforLagrangiancobordisms} gives the desired isomorphism.
\end{proof}

\begin{coro}
\label{coro:theBiranCorneamapexists}
There is a natural group morphism
\[ \Theta_{BC} : \Gimmunob \to K_0(\DFuk(S_g)), \]
which maps the class of an embedded curve $\gamma$ to its representative in $K_0(\DFuk(S_g))$.
Further, this morphism is surjective.	
\end{coro}

\begin{proof}
The existence of the morphism $\Theta_{BC} : \Gimmunob \to K_0(\DFuk(S_g))$ is immediate from Theorem \ref{Theo:DecompositionforLagrangiancobordisms}.

Moreover, recall from \cite{ab07} that the group $K_0(\DFuk(S_g))$ is generated by embedded curves. Hence, the image of $\Theta_{BC}$ is the whole group $K_0(\DFuk(S_g))$.
\end{proof}

\section{Computation of the unobstructed Lagrangian Cobordism Group}
In this section, we compute the unobstructed Lagrangian cobordism group $\Gimmunob$.

First, notice that by the results of subsection \ref{subsubsection:definitionofpiandmu}, the homology class and the Maslov class yield maps
\[ \pi : \Gimmunob \to H_1(S_g,\Z), \ \mu : \Gimmunob \to \Z / \chi(S_g) \Z . \]

Our main tool is the following
\begin{Theo}
\label{Theo:Computationofunobgroup}
There is a long exact sequence
\[ 0 \to \R \xrightarrow[]{i} \Gimmunob \xrightarrow{\pi \oplus \mu} H_1(S_g,\Z) \oplus \Z / \chi(S_g) \Z \to 0. \]
Furthermore, this exact sequence is split.	
\end{Theo}

\subsection{Holonomy and the map $i$}
In this subsection, we define an injection 
\[ i : \R \to \Gimmunob.\] 
This map has a simple geometric interpretation: to a real number $x$, it associates the oriented boundary of a cylinder of area $x$. We have to check that the map is indeed well-defined.

Let $p: S(TS_g) \to S_g$ be the unit tangent bundle with respect to the metric $g_j = \omega(\cdot, j \cdot)$. We choose a one-form $A \in \Omega^1(S(TS_g))$ such that 
\[ p^* \omega = dA. \]

An immersed curve $\gamma : S^1 \to S_g$ admits a canonical lift to $S(TS_g)$:
\[\tilde{\gamma} : \begin{array}{ccc}
 	S^1 & \to & S(TS_g) \\
 	t & \mapsto & \left(\gamma(t), \frac{\gamma'(t)}{\norm{\gamma'(t)}  } \right)
 \end{array} \]
We define the \emph{holonomy} of $\gamma$ with respect to $A$ by
\[ \Hol_A(\gamma) := \int_{S^1} \tilde{\gamma}^*A . \] 
We shall use the following properties of this number.

\begin{prop}
\label{prop:holonomyproperties}
The following assertions are true.
\begin{enumerate}[label=(\roman*)]
\item Let $F : [0,1] \times S^1 \to S_g$ be an isotopy between the two immersed curves $\gamma_0$ and $\gamma_1$. Then
	\[ \Hol_A(\gamma_1) - \Hol_A(\gamma_0) = \int_{[0,1] \times S^1} F^* \omega. \] 
\item There is a well-defined group morphism
 \[	\Hol_A : \Gimmunob \to \R \]
whose value on the class of an embedded curve $\gamma$ is $\Hol_A(\gamma)$.
\end{enumerate}
\end{prop}

\begin{proof}
There is a natural lift of $F$ to $S(TS_g)$:
\[ \tilde{F} : \begin{array}{ccc}
 [0,1] \times S^1 & \to & S(TS_g) \\
 (t,s) & \mapsto & \left (F(t,s), \frac{\partial_t F (t,s)}{\norm{\partial_t F(t,s)}} \right)	
 \end{array}.
 \]
We now apply Stokes Theorem:
\begin{align*}
\Hol_A(\gamma_1) - \Hol_A(\gamma_0) & = \int_{S^1 \times \{1\}} \tilde{F}^*A - \int_{S^1 \times \{0\}} \tilde{F}^*A	\\
& = \int_{S^1 \times [0,1]} \tilde{F}^* dA \\
& = \int_{S^1 \times [0,1]} \tilde{F}^* p^* \omega \\
& = \int_{S^1 \times [0,1]} F^* \omega .
\end{align*}

Recall from Corollary \ref{coro:theBiranCorneamapexists} that there is a natural group morphism
\[ \Theta_{BC} : \Gimmunob \to K_0(\DFuk(S_g)). \] 
On the other hand, it is a result of Abouzaid (\cite[Proposition 6.1]{ab07}) that the holonomy induces a group morphism 
\[ K_0(\DFuk(S_g)) \to \R. \] 
Therefore, the composition of these two is a group morphism. This proves $(ii)$. 
\end{proof}

The following lemma seems to be a well-known fact (\cite[Section 6]{Sei11}). I learned its proof from Jordan Payette.
\begin{lemma}
\label{lemma:holonomyequalshamiltonianinvariant}	
Let $\gamma_0$ and $\gamma_1$ be two isotopic embedded curves with 
\[ \Hol_A(\gamma_1) = \Hol_A(\gamma_0). \] 
Then the curves $\gamma_0$ and $\gamma_1$ are Hamiltonian isotopic to each other.
\end{lemma}

\begin{proof}
We fix an isotopy $(\gamma_t)_{t \in [0,1]}$ from $\gamma_0$ to $\gamma_1$. 

Let $\phi^t : S_g \to S_g$ be a global isotopy of diffeomorphisms such that $\phi^t \circ \gamma_0 = \gamma_t$. Choose a symplectic embedding $\psi : S^1 \times (- \varepsilon, \varepsilon) \to S_g$ such that $\psi(s,0) = \gamma_0(s)$. In the coordinates $(s,u) \in S^1 \times (- \varepsilon, \varepsilon)$ there is a smooth function $f > 0$ such that $(\phi^t)^* \omega = f(s,u,t) ds \wedge du$. We let $\beta : S^1 \times (-\varepsilon,\varepsilon) \times [0,1] \to \R $ be a smooth function such that
\begin{itemize}
	\item we have $\beta(s,u,t) = \frac{1}{f(s,u,t)}$ for $\norm{u} < \frac{\varepsilon}{3}$,
	\item we have $\beta(s,u,t) = 1$ for $\norm{u} > \frac{2\varepsilon}{3}$.
\end{itemize}
Define, for $t \in [0,1]$
\[\psi^t : 
\begin{array}{ccc}
 S^1 \times (-\varepsilon,\varepsilon) & \to & S^1 \times (-\varepsilon,\varepsilon) \\
 (s,u) & \mapsto & \left(s,u\beta(u,s,t) \right)
 \end{array}
 \]	
 Since this map coincides with the identity on the open set $\ens{(s,u)}{\norm{u} > \frac{2\varepsilon}{3}}$, it extends to a diffeomorphism $\psi^t : S_g \to S_g$.

Consider $\chi^t = \phi^t \circ \psi^t$ and let $\omega_t = (\chi^t)^* \omega$. This isotopy satisfies $\chi^t \circ \gamma_0 = \gamma_t$. Moreover, the family $(\omega_t)_{t \in [0,1]}$ is constant along $\gamma_0$. One can easily apply Moser's trick to find a family of transformation $\Phi^t$ such that $(\Phi^t)^* \omega_t = \omega$ and $\Phi^t(\gamma_0) = \gamma_0$.

The composition $\chi^t \circ \Phi^t$ is a symplectic isotopy $\Psi^{t} : S_g \to S_g$ such that
\[ \Psi^t \circ \gamma_0 = \gamma_t . \]
Moreover, this has zero flux along the path $\gamma_0$. We call $X_t$ the symplectic vector field generated by $\Psi^t$.

 We adapt the construction of \cite[Theorem 10.2.5]{McDS17} to obtain a Hamiltonian isotopy between $\gamma_0$ and $\gamma_1$.

The difference of holonomy between $\gamma_1$ and $\gamma_0$ is
\begingroup
\allowdisplaybreaks
\begin{align*}
\Hol_A(\gamma_1) - \Hol_A(\gamma_0) & = \int_0^1 \int_0^1 \omega \left(\frac{d}{dt} (\Psi^t \circ \gamma_0) (s), \frac{d}{ds} (\Psi^t \circ \gamma_0)(s) \right) dt ds \\
& = \int_0^1 \int_0^1 \omega \left( X_t \circ \Psi^t \circ \gamma_0(s), d \Psi^t ( \gamma_0' (s) ) \right ) dt ds \\
& = \int_0^1 \int_0^1 \omega \left( \left(\Psi^t \right)^* X_t (\gamma_0(s)), \gamma_0'(s) \right) dt ds \\
& = \int_{S^1} \gamma_0^* \left[ \int_0^1 \omega \left( \left(\Psi^t \right)^* X_t, \cdot \right) dt \right ]
\end{align*}
\endgroup
So the one-form 
\[ \gamma_0^* \left[ \int_0^1 \omega \left( \left(\Psi^t \right)^* X_t, \cdot \right) dt \right ] \] 
is exact on $S^1$. Hence, there is a smooth function
\[ F : \Im(\gamma_1) \to \R \]
such that
\[ \forall v \in T (\Im(\gamma_0)), \ \int_0^1 \omega \left( \left(\Psi^t \right)^* X_t, v \right) dt = -dF \cdot d\Psi^1 (v). \]
We extend $F$ to a smooth function 
\[ F: S_g \to \R. \] 
We define a new isotopy $\left( \Phi^t \right)_{t \in [0,1]}$ by
\[ \Phi^t = \left\{
\begin{array}{cc}
	\psi^{2t} & \text{ for } t \in \left[0, \frac{1}{2}\right] \\
	\phi_F^{1-2t}  \circ \psi^1 &\text{ for } t \in \left[\frac{1}{2},1\right]
\end{array}
 \right. .\] 
 Call $Y_t$ its associated vector field. We compute for $v \in \Im(\gamma_0)$,
 \begin{align*}
 \int_0^1 \omega \left(\left(\Phi^t \right)^* Y_t, v \right) &= \int_0^{\frac{1}{2}} \omega \left( \left( \Psi^{2t} \right)^* X_{2t}, v \right) 2dt - \int_{\frac{1}{2}}^1 \omega \left( \left( \phi_F^{1-2t} \circ \Psi^1 \right)^* X_F, v \right) 2dt  \\
 & = \int_0^1 \omega \left( (\Psi^t)^* X_t, v \right) dt - \int_{\frac{1}{2}}^1 \omega \left( X_{F \circ \Psi^1}, v \right) 2dt \\
 & = -d \left(F \circ \Psi^1 \right) (v) + d \left(F \circ \Psi^1 \right) (v). 
 \end{align*}
So 
\begin{equation}
\int_0^1 \omega \left(\left(\Phi^t \right)^* Y_t, v dt \right) = 0 
\label{eqn:equation flux}	
\end{equation}
We let
\[ Z_t = - \int_0^t \left( \Phi^\lambda \right)^*Y_\lambda d \lambda \]
and $\theta_t^s$ be the flow associated to $Z_t$. Then the isotopy $\mu^t = \Phi^t \circ \theta^1_t$ is Hamiltonian (see the proof of \cite[Theorem 10.2.5]{McDS17}).

Moreover, from \ref{eqn:equation flux}, we have for all $v \in T(\Im(\gamma_0))$
\[ \omega(Z_1,v) = 0 .\]
Hence, $Z_1$ is tangent to $\Im(\gamma_0)$. Therefore, $\theta_1^1(\gamma_0) \subset \Im(\gamma_0)$. We deduce
\begin{align*} 
\mu^1 (\Im(\gamma_0)) & = \Phi^1(\theta_1^1(\Im(\gamma_0))) \\
 & = \phi_F^{-1} (\Im(\gamma_1)) .
\end{align*}
So the Hamiltonian isotopy $\left( \phi_F^t \circ \mu^t \right)_{t \in [0,1]}$ maps the image of the curve $\gamma_0$ to the image of $\gamma_1$.
\end{proof}

Here is the main result of this section.

\begin{prop}
\label{prop:twocylinderofsameareahavethesameclass}
Let $\gamma_1, \gamma_2$ (resp. $\tilde{\gamma}_1, \tilde{\gamma}_2$) be two isotopic non-separating embedded curves such that
\[ \Hol_A(\gamma_2) - \Hol_A(\gamma_1) = \Hol_A(\tilde{\gamma}_2) - \Hol_A(\tilde{\gamma}_1) .\]
Then, in $\Gimmunob$, we have
\[ [\gamma_2] - [\gamma_1] = [\tilde{\gamma}_1] - [\tilde{\gamma}_2]. \] 	
\end{prop}
	 
\begin{proof}
First, we assume that the curves $\gamma_1$, $\gamma_2$, $\tilde{\gamma}_1$, $\tilde{\gamma}_2$ are pairwise disjoint. Furthermore, we assume that the pair $(\gamma_1, \gamma_2)$ does not bound any oriented surface. Then, by the change of coordinates principle (\cite[1.3]{FM012}), we are in the situation of Figure \ref{figure:FigureM}. We do the successive surgeries indicated in this figure. The third red curve is a Hamiltonian perturbation of $\gamma_1$ while the fourth red curve is a Hamiltonian perturbation of $\gamma_3$. This process produces a curve $\gamma_7$ isotopic to $\tilde{\gamma}_1$.

\input{Figures/FigureM.tex}

Moreover, we compute
\begin{align*}
\Hol_A(\gamma_7) & = \Hol_A(\gamma_6) - \Hol_A(\gamma_3) \\
& = \Hol_A(\gamma_5) - \Hol_A(\gamma_1) - \Hol_A(\gamma_3) \\
& = \Hol_A(\gamma_4) + \Hol_A(\gamma_2) - \Hol_A(\gamma_1) - \Hol_A(\gamma_3) \\
& = \Hol_A(\tilde{\gamma}_1) + \Hol_A(\gamma_3) + \Hol_A(\gamma_2) - \Hol_A(\gamma_1) - \Hol_A(\gamma_3) \\
& = \Hol_A(\tilde{\gamma}_1) - \Hol_A(\gamma_1) + \Hol_A(\gamma_2) \\
& = \Hol_A(\tilde{\gamma}_2)
\end{align*}
since $\Hol_A(\gamma_2) - \Hol_A(\gamma_1) = \Hol_A(\tilde{\gamma}_2) - \Hol_A(\tilde{\gamma}_1 )$. Hence, $\gamma_7$ is Hamiltonian isotopic to $\tilde{\gamma}_2$ by Proposition \ref{prop:holonomyproperties}. Since two Hamiltonian isotopic curves are embbedded Lagrangian cobordant (see Remark \ref{rk:Lagrangiansuspension}), we have
\[ \tilde{\gamma_2} = \gamma_7 . \]
in $\Gimmunob$.

The surgery procedure produces several Lagrangian cobordisms that are all embedded and oriented. We can glue these together to obtain an embedded oriented Lagrangian cobordism $(\gamma_3^{-1},\tilde{\gamma}_1,\gamma_3,\gamma_2,\gamma_1^{-1}) \to \gamma_7$. Hence, in $\Gimmunob$
\begin{align*}
	\tilde{\gamma}_2 & =  \gamma_7 \\
		& =  - \gamma_3 + \tilde{\gamma_1} + \gamma_3 + \gamma_2 - \gamma_1 \\
		& =  \tilde{\gamma}_1. 
\end{align*}
In the general case, by \cite[Theorem 4.3]{FM012} there exists a sequence of non-separating embedded curves $ \gamma_1 = \alpha_1, \ldots, \alpha_k = \gamma_2$ such that for each $i \in \{1, \ldots, k\}$, $\alpha_i$ and $\alpha_{i+1}$ have no intersection points. Now, apply the first case iteratively to conclude.
\end{proof}

There is a direct definition for the application $i$. Let $\gamma$ be a non-separating embedded curve. There is $\varepsilon > 0$ such that if $\norm{x} < \varepsilon$, there exists an embedded curve $\tilde{\gamma}$ isotopic to $\gamma$ such that 
\[ \Hol_A(\tilde{\gamma}) - \Hol_A(\gamma) = x . \]
Now, let $x \in \R$ and $x_1, \ldots, x_m \in \R$ with $\left \{ \begin{array}{c} x_1 + \ldots x_m = x \\ \forall i \in \{ 1 \ldots m \}, \ \norm{x_i} < \varepsilon \end{array} \right.$.

 For $i = 1 \ldots m$, we choose an embedded curve $\gamma_i$ isotopic to $\gamma$ such that 
 \[ \Hol_A(\gamma_i) - \Hol_A(\gamma) = x_i. \]
We put
\[ i(x) = \sum_{i=1}^m \left( [\gamma_i] - [\gamma] \right). \]

\begin{coro}
\label{coro:iisinjective}
In the setting above, this defines an injective group morphism
\[ i : \R \to \Gimmunob. \]	
\end{coro}

\begin{proof}
It is easy to see, using Proposition \ref{prop:twocylinderofsameareahavethesameclass}, that $i(x)$ does not depend on the choice of $x_1, \ldots, x_m$, nor on the choice of $\gamma$. Therefore, it defines a group morphism $i : \R \to \Gimmunob$.

Moreover, notice that by definition 
\[ \Hol_A(i(x)) = x, \]  
so $i$ is injective.
\end{proof}

\subsection{Surgery of immersed curves and obstruction}
\label{subsection:surgeryofimmersedcurves}
Throughout this subsection we let $\alpha : S^1 \to S_g$ and $\gamma : S^1 \to S_g$ be immersed curves such that
\begin{enumerate}[label = (\roman*)]	
	\item $\alpha$ is embedded,
	\item $\gamma$ is unobstructed,
	\item $\alpha$ and $\gamma$ are transverse. 	
\end{enumerate}

A \emph{bigon} is an immersed polygon with one boundary arc on $\alpha$ and one boundary arc on $\gamma$. Notice that a bigon is \emph{not necessarily injective}. We say that $\alpha$ and $\gamma$ are in \emph{minimal position} if there are no bigons between $\alpha$ and $\gamma$. There is a useful criterion to detect minimal position for transverse curves.

\begin{lemma}
\label{lemma:criterionforminimalposition}
In the above setting, if there are $s_0 \neq \overline{s_0}$ and $t_0 \neq \overline{t_0}$ such that the loop $\gamma_{\vert [t_0,\overline{t_0}]} \cdot \alpha_{\vert [s_0,\overline{s_0}]}^{-1}$ is homotopic to a constant, then $\gamma$ and $\alpha$ are not in minimal position.
\end{lemma}

\begin{proof}
The hypothesis implies that the loop $\gamma_{\vert [t_0,\overline{t_0}]} \cdot \alpha_{\vert [s_0,\overline{s_0}]}^{-1}$ lifts to the universal cover $\tilde{S}_g$ of $S_g$. Denote this loop by $f : S^1 \to \tilde{S}_g$. 

So there are two lifts $\tilde{\alpha} : \R \to \tilde{S}_g$ and $\tilde{\gamma} : \R \to \tilde{S}_g$ of $\alpha$ and $\gamma$ respectively such that
\[ \tilde{\gamma}_{\vert [t_0,\overline{t_0}]} \cdot \tilde{\alpha}_{\vert [s_0,\overline{s_0}]}^{-1} = f .\]

Assume there are
\begin{itemize}
	\item two increasing sequences $(s_n)_{n \in \N}$, $(t_n)_{n \in \N}$,
	\item two decreasing sequences $(\overline{s}_n)_{n \in \N}$, $(\overline{t}_n)_{n \in \N}$,
\end{itemize}
such that
\begin{align*}
\forall n \in \N, \ s_0 < s_n < \overline{s}_n < \overline{s}_0, \\
\forall n \in \N, \ t_0 < t_n < \overline{t}_n < \overline{t}_0, 	
\end{align*}
and such that the path $\tilde{\gamma}_{\vert [t_n,\overline{t_n}]} \cdot \tilde{\alpha}_{\vert [s_n,\overline{s_n}]}^{-1}$ is a loop.

If $\overline{s}_n - s_n \to 0$, then the adjacent sequences $(s_n)$ and $(\overline{s}_n)$ converge to a common limit, say $l$. So there is a sequence of intersection points between $\tilde{\alpha}$ and $\tilde{\gamma}$ which accumulates at $\tilde{\alpha}(l)$. This is absurd since $\alpha$ and $\gamma$ are transverse. We conclude that the sequence $(\overline{s}_n - s_n)$ has a positive limit. 
The same argument shows that $(\overline{t}_n - t_n)$ has a positive limit.

From this we infer that there are $s < \overline{s}$ and $t < \overline{t}$ such that the path
\[ \tilde{\gamma}_{\vert [t,\overline{t}]} \cdot \tilde{\alpha}_{\vert [s,\overline{s}]}^{-1} \]
is an embedded loop. Hence, this bounds an embedded bigon, say $u$. Now since the projection $p : \tilde{S}_g \to S_g$ is an immersion, the map $p \circ u$ is an immersed bigon with boundary arcs on $\gamma$ and $\alpha$. 
\end{proof}

Moreover, we fix $x$ an intersection point of degree $1$ between $\gamma$ and $\alpha$. 

\begin{prop}
\label{prop:asurgeryisnonobstructed}
We let $\alpha$ and $\gamma$ be as in the beginning of Subsection \ref{subsection:surgeryofimmersedcurves}. Further, we assume
\begin{enumerate}[label=(\roman*)]
	\item the curves $\gamma$ and $\alpha$ are in minimal position,
	\item the curve $\gamma \#_x \alpha$ is not homotopic to a constant.	
\end{enumerate}
 Then, the surgery $\gamma \#_x \alpha $ is unobstructed.
\end{prop}

\begin{proof}[Proof of Proposition \ref{prop:asurgeryisnonobstructed}]	
The proof will proceed by contradiction. We assume that the surgery is obstructed.

Let us start with a quick outline of the proof of the Proposition.
\begin{enumerate}
	\item In Lemma \ref{lemma:Existenceofanimmersedteardrop}, we show that there is an \emph{immersed} teardrop $u$ on $\gamma \#_x \alpha$.
	\item In Lemma \ref{lemma:structureoftheteardropnearthesurgeredpoint} and Figure \ref{figure:behaviouroftheteardroparoundsurgeredpoint}, we describe precisely the behavior of the teardrop around the surgered point. From this we construct an immersed polygon $v$ with boundary on $\alpha$ and $\gamma$ in Lemma \ref{lemma:constructionofapolygonvwithcorners}.
	\item Using algebraic properties of the fundamental group of $S_g$, we bound the number of corners of the polygon $v$. Then by considering the connected components of $\D \backslash v^{-1}(\alpha)$, we conclude that there is a bigon between $\gamma$ and $\alpha$. This is a contradiction since $\gamma$ and $\alpha$ are in minimal position.
\end{enumerate}
We start with the following lemma.
\begin{lemma}
\label{lemma:Existenceofanimmersedteardrop}
Let $\gamma : S^1 \to S_g$ be a generic curve which is obstructed and non-homotopic to zero. Then there exists an immersed holomorphic teardrop with boundary on $\gamma$ and corner at a double point of $\gamma$.

Moreover, the corner of this teardrop covers one or three quadrants.
\end{lemma}

\begin{proof}
Let $\tilde{\gamma} : \R \to \tilde{S}_g$ be a lift of $\gamma$ to the universal cover $p : \tilde{S}_g \to S_g$ of $S_g$. Since $\gamma$ is obstructed, there are two reals $s_0 < t_0$ such that $\tilde{\gamma}(s_0) = \tilde{\gamma}(t_0)$.

We claim that there are $s < t$ such that $\tilde{\gamma}(s) = \tilde{\gamma} (t) $ and such that $\tilde{\gamma}_{\vert (s,t) }$ is injective. Assume by contradiction that there are sequences $(s_n)_{n \in \N}$ and $(t_n)_{n \in \N}$ with
\begin{itemize}
	\item $\forall n \in \N, \ s_0 < s_n < t_n < t_0$,
	\item the sequence $(s_n)_{n \in \N}$ is increasing, the sequence $(t_n)_{n \in \N}$ is decreasing and the sequence $(t_n - s_n)_{n \in \N}$ converges to $0$,
	\item $\forall n \in \N, \ \tilde{\gamma}(s_n) = \tilde{\gamma}(t_n)$.
\end{itemize}
Call $s_\infty$ the limit of the adjacent sequences $(s_n)_{n \in \N}$ and $(t_n)_{n \in \N}$. Since $\tilde{\gamma}$ is an immersion, it is in particular a local embedding around $s_\infty$. But there is a sequence of distinct double points converging to $S_\infty$.

Now, since $S^1$ is compact and $\tilde{\gamma}_{\vert (s,t)}$ is injective, $\tilde{\gamma}_{\vert (s,t)}$ is an embedded curve. Call $\Omega$ the bounded connected component of $\tilde{S}_g \backslash \Im \left( \tilde{\gamma}_{\vert (s,t)} \right)$. By the Riemann mapping Theorem, there is a biholomorphism $u : \D \to \Omega$ which extends to a homeomorphism $\tilde{u} : \overline{\D} \to \Omega$. Moreover, we assume that $\tilde{u}(-1) = \tilde{\gamma}(s) = \tilde{\gamma}(t)$. Since its image is of finite area, $\tilde{u}$ is of finite energy.
 
The map $\tilde{u}$ is a holomorphic teardrop. Let us call $\tilde{y} = \tilde{\gamma}(s) = \tilde{\gamma}(t)$ its corner. Choose $\varepsilon \ll 1$ and denote by $\alpha \in (0,\pi)$ the angle between $\tilde{\gamma}_{\vert (s-\varepsilon, s+ \varepsilon)}$ and $\tilde{\gamma}_{\vert (t-\varepsilon, t+\varepsilon)}$. Choose a local chart $\phi$ around $y$ such that $\phi(\tilde{\gamma}_{\vert (s-\varepsilon, s+ \varepsilon)}) = \R$ and $\phi(\tilde{\gamma}_{\vert (t-\varepsilon, t+\varepsilon)}) = e^{i\alpha} \R$. Then, by \cite[Proposition 5]{P18} , there is a local chart $\psi : \Omega \to \D$ around $-1$ with domain a neighborhood of $0$ in the Poincaré half-plane such that
\[ \phi \circ \tilde{u} \circ \psi (z) = z^{\alpha + m - 1}. \]
Here, $m$ is the number of quadrants covered by $\tilde{u}$ at its corner. Hence, if $m \geqslant 4$, $\tilde{u}$ is not injective.

Now, since the projection $p$ is an immersion, the map $p \circ \tilde{u}$ satisfies the conclusion of the lemma.
\end{proof}

We now return to the setting of Proposition \ref{prop:asurgeryisnonobstructed}. By Lemma \ref{lemma:Existenceofanimmersedteardrop} above, there is an immersed holomorphic teardrop $u$ with boundary on $\gamma \#_x \alpha$. We denote its corner by $y$. 

Recall from Subsubsection \ref{subsubsection:surgeryofintersectionpoints} that we denoted by $U_x$ the Darboux neighborhood in which we perform the surgery. We call $\gamma_+$ (resp. $\gamma_-$) the upper (resp. lower) connected component of $\Im(\gamma \#_x \alpha) \cap U_x$. See Figure \ref{figure:thecurvesgamma+andgamma-}.

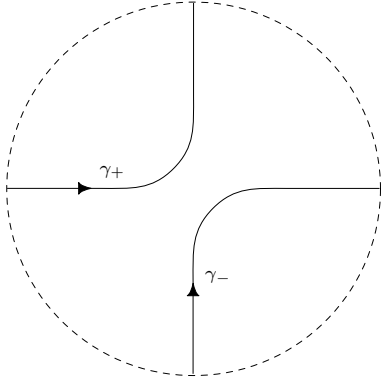
\begin{figure}
\captionsetup{justification=centering,margin=2cm}

\begin{tikzpicture}[y=0.80pt, x=0.80pt, yscale=-1.000000, xscale=1.000000, inner sep=0pt, outer sep=0pt]
  \path[xscale=1.000,yscale=-1.000,draw=black,dash pattern=on 2.40pt off
    2.40pt,line join=round,line cap=round,miter limit=4.00,fill opacity=0.156,line
    width=0.300pt] (338.3877,-207.8881) circle (2.4843cm);
  \path[draw=black,line join=miter,line cap=butt,miter limit=4.00,even odd
    rule,line width=0.300pt] (338.1384,295.4175) .. controls (338.1384,295.4175)
    and (338.1384,266.1787) .. (338.1384,251.5593) .. controls (338.1384,236.9399)
    and (337.2513,227.4156) .. (347.6145,217.0524) .. controls (357.9777,206.6893)
    and (367.5226,207.7010) .. (382.2148,207.7010) .. controls (396.9069,207.7010)
    and (426.2912,207.7010) .. (426.2912,207.7010);
  \path[draw=black,line join=miter,line cap=butt,miter limit=4.00,even odd
    rule,line width=0.300pt] (338.3598,119.9658) .. controls (338.3598,119.9658)
    and (338.3598,149.2046) .. (338.3598,163.8240) .. controls (338.3598,178.4434)
    and (339.2469,187.9676) .. (328.8837,198.3308) .. controls (318.5205,208.6940)
    and (308.9755,207.6822) .. (294.2834,207.6822) .. controls (279.5912,207.6822)
    and (250.2070,207.6822) .. (250.2070,207.6822);
  \path[draw=black,fill=black,line join=round,line cap=round,miter limit=4.00,line
    width=0.300pt] (284.0439,204.8869) .. controls (285.0327,206.4317) and
    (287.3626,207.0420) .. (289.9198,207.6600) .. controls (287.2798,208.2981) and
    (284.9987,208.9415) .. (284.0439,210.4332) -- cycle;
  \path[draw=black,fill=black,line join=round,line cap=round,miter limit=4.00,line
    width=0.300pt] (335.3818,258.4695) .. controls (336.9266,257.4808) and
    (337.5368,255.1509) .. (338.1549,252.5937) .. controls (338.7930,255.2337) and
    (339.4363,257.5148) .. (340.9281,258.4695) -- cycle;

	\draw (300,200) node {\tiny $\gamma_+$};
	\draw (350,250) node {\tiny $\gamma_-$};
\end{tikzpicture}

\caption{The curves $\gamma_+$ and $\gamma_-$ near the intersection point $x$.}
\label{figure:thecurvesgamma+andgamma-}
\end{figure}

The set $u_{\vert \partial \D}^{-1}(U_x)$ is a finite union of arcs. We label them clockwise by $A_1, \ldots, A_N$. Moreover, for $i= 1 \ldots N$, we call $C_i$ the connected component of $u^{-1}(U_x)$ which contains $A_i$.

\begin{lemma}
\label{lemma:structureoftheteardropnearthesurgeredpoint}
Let $i \in \{1, \ldots, N \}$. With the notations above, the open set $C_i$ is an embedded half-disk. Moreover, the map $u$ restricts to a biholomorphism from $C_i$ to one of the quadrants represented in Figure \ref{figure:behaviouroftheteardroparoundsurgeredpoint}.
\end{lemma}
\begin{figure}
\captionsetup{justification=centering,margin=2cm}

\begin{tikzpicture}[y=0.80pt, x=0.80pt, yscale=-1.000000, xscale=1.000000, inner sep=0pt, outer sep=0pt]
  \path[xscale=1.000,yscale=-1.000,draw=black,dash pattern=on 3.60pt off
    3.60pt,line join=round,line cap=round,miter limit=4.00,fill opacity=0.156,line
    width=0.300pt] (224.4844,-136.8466) circle (2.0474cm);
  \path[draw=black,line join=miter,line cap=butt,miter limit=4.00,even odd
    rule,line width=0.300pt] (224.2789,208.9811) .. controls (224.2789,208.9811)
    and (224.2789,184.8849) .. (224.2789,172.8368) .. controls (224.2789,160.7887)
    and (223.5478,152.9396) .. (232.0883,144.3991) .. controls (240.6288,135.8587)
    and (248.4950,136.6925) .. (260.6031,136.6925) .. controls (272.7111,136.6925)
    and (296.9272,136.6925) .. (296.9272,136.6925);
  \path[draw=black,line join=miter,line cap=butt,miter limit=4.00,even odd
    rule,line width=0.300pt] (224.4614,64.3883) .. controls (224.4614,64.3883) and
    (224.4614,88.4845) .. (224.4614,100.5326) .. controls (224.4614,112.5807) and
    (225.1925,120.4298) .. (216.6520,128.9703) .. controls (208.1115,137.5108) and
    (200.2453,136.6770) .. (188.1372,136.6770) .. controls (176.0292,136.6770) and
    (151.8131,136.6770) .. (151.8131,136.6770);
  \path[draw=black,fill=black,line join=round,line cap=round,miter limit=4.00,line
    width=0.300pt] (179.6987,134.3733) .. controls (180.5136,135.6464) and
    (182.4337,136.1493) .. (184.5411,136.6587) .. controls (182.3654,137.1845) and
    (180.4856,137.7147) .. (179.6987,138.9441) -- cycle;
  \path[xscale=1.000,yscale=-1.000,draw=black,dash pattern=on 3.60pt off
    3.60pt,line join=round,line cap=round,miter limit=4.00,fill opacity=0.156,line
    width=0.300pt] (428.5352,-136.8466) circle (2.0474cm);
  \path[draw=black,line join=miter,line cap=butt,miter limit=4.00,even odd
    rule,line width=0.300pt] (428.5122,64.3883) .. controls (428.5122,64.3883) and
    (428.5122,88.4845) .. (428.5122,100.5326) .. controls (428.5122,112.5807) and
    (429.2433,120.4298) .. (420.7028,128.9703) .. controls (412.1623,137.5108) and
    (404.2961,136.6770) .. (392.1881,136.6770) .. controls (380.0800,136.6770) and
    (355.8639,136.6770) .. (355.8639,136.6770);
  \path[draw=black,fill=black,line join=round,line cap=round,miter limit=4.00,line
    width=0.300pt] (388.5919,138.9441) .. controls (387.7770,137.6710) and
    (385.8569,137.1681) .. (383.7495,136.6587) .. controls (385.9252,136.1328) and
    (387.8050,135.6026) .. (388.5919,134.3733) -- cycle;
  \path[xscale=1.000,yscale=-1.000,draw=black,dash pattern=on 3.60pt off
    3.60pt,line join=round,line cap=round,miter limit=4.00,fill opacity=0.156,line
    width=0.300pt] (224.4844,-328.7756) circle (2.0474cm);
  \path[draw=black,line join=miter,line cap=butt,miter limit=4.00,even odd
    rule,line width=0.300pt] (224.2789,400.9101) .. controls (224.2789,400.9101)
    and (224.2789,376.8139) .. (224.2789,364.7658) .. controls (224.2789,352.7177)
    and (223.5478,344.8686) .. (232.0883,336.3281) .. controls (240.6288,327.7876)
    and (248.4950,328.6214) .. (260.6031,328.6214) .. controls (272.7111,328.6214)
    and (296.9272,328.6214) .. (296.9272,328.6214);
  \path[draw=black,fill=black,line join=round,line cap=round,miter limit=4.00,line
    width=0.300pt] (222.0071,370.4606) .. controls (223.2802,369.6458) and
    (223.7832,367.7257) .. (224.2925,365.6183) .. controls (224.8184,367.7940) and
    (225.3486,369.6738) .. (226.5779,370.4606) -- cycle;
  \path[xscale=1.000,yscale=-1.000,draw=black,dash pattern=on 3.60pt off
    3.60pt,line join=round,line cap=round,miter limit=4.00,fill opacity=0.156,line
    width=0.300pt] (428.5352,-328.7756) circle (2.0474cm);
  \path[draw=black,line join=miter,line cap=butt,miter limit=4.00,even odd
    rule,line width=0.300pt] (428.3297,400.9101) .. controls (428.3297,400.9101)
    and (428.3297,376.8139) .. (428.3297,364.7658) .. controls (428.3297,352.7177)
    and (427.5987,344.8686) .. (436.1391,336.3281) .. controls (444.6796,327.7876)
    and (452.5458,328.6214) .. (464.6539,328.6214) .. controls (476.7619,328.6214)
    and (500.9780,328.6214) .. (500.9780,328.6214);
  \path[draw=black,line join=miter,line cap=butt,miter limit=4.00,even odd
    rule,line width=0.300pt] (428.5122,256.3173) .. controls (428.5122,256.3173)
    and (428.5122,280.4135) .. (428.5122,292.4616) .. controls (428.5122,304.5097)
    and (429.2433,312.3588) .. (420.7028,320.8993) .. controls (412.1623,329.4398)
    and (404.2961,328.6060) .. (392.1881,328.6060) .. controls (380.0800,328.6060)
    and (355.8639,328.6060) .. (355.8639,328.6060);
  \path[draw=black,fill=black,line join=round,line cap=round,miter limit=4.00,line
    width=0.300pt] (430.6288,365.6183) .. controls (429.3557,366.4332) and
    (428.8527,368.3533) .. (428.3434,370.4607) .. controls (427.8175,368.2850) and
    (427.2873,366.4052) .. (426.0580,365.6183) -- cycle;
  \path[draw=black,fill=black,line join=round,line cap=round,miter limit=4.00,fill
    opacity=0.067,line width=0.000pt] (152.2232,129.1494) .. controls
    (152.7497,116.8315) and (161.1433,97.6453) .. (173.1337,85.6549) .. controls
    (185.1241,73.6646) and (200.8265,66.0648) .. (216.7175,64.8337) --
    (224.3202,64.2447) -- (224.4095,89.9825) .. controls (224.5061,117.8150) and
    (223.6050,120.7047) .. (218.1380,127.0916) .. controls (212.3714,133.8286) and
    (205.6740,136.6413) .. (192.9830,136.6530) .. controls (185.8324,136.6600) and
    (182.4900,136.6500) .. (181.2553,136.6500) -- (178.3825,136.6500) .. controls
    (177.4004,136.6500) and (175.2545,136.6627) .. (165.6250,136.6627) --
    (151.9020,136.6627) -- cycle;
  \path[draw=black,fill=black,dash pattern=on 0.00pt off 0.00pt,line
    join=round,line cap=round,miter limit=4.00,fill opacity=0.067,line
    width=0.000pt] (411.2851,207.4732) .. controls (381.6145,199.8676) and
    (358.9910,174.5662) .. (356.5418,144.8574) -- (355.8635,136.6296) --
    (371.5783,136.5403) .. controls (381.9677,136.4813) and (385.3492,136.2929) ..
    (386.1838,137.1275) .. controls (386.9346,137.8783) and (389.2710,137.5787) ..
    (389.2710,136.9216) .. controls (389.2710,136.3058) and (392.8714,137.0147) ..
    (398.2237,136.6348) .. controls (403.2950,136.2748) and (409.1394,135.7838) ..
    (411.4296,135.1071) .. controls (416.7186,133.5444) and (424.8035,125.3368) ..
    (426.6914,119.8448) .. controls (427.6719,116.9924) and (428.7371,108.5354) ..
    (428.6122,90.4355) -- (428.4297,63.9954) -- (436.5669,64.7537) .. controls
    (455.6649,66.5334) and (475.0343,77.9139) .. (486.5426,93.3971) .. controls
    (496.9345,107.3782) and (501.7323,123.3065) .. (500.8161,140.9012) .. controls
    (499.8312,159.8152) and (493.5149,174.1655) .. (479.9175,187.7629) .. controls
    (470.7599,196.9205) and (462.4119,202.1160) .. (450.0587,206.2806) .. controls
    (440.7779,209.4095) and (421.1410,209.9995) .. (411.2851,207.4732) -- cycle;
  \path[draw=black,fill=black,line join=round,line cap=round,miter limit=4.00,fill
    opacity=0.067,line width=0.000pt] (207.5022,399.3372) .. controls
    (183.1251,393.3124) and (163.2061,374.6731) .. (155.2273,350.6273) .. controls
    (152.8488,343.4595) and (151.7351,340.8624) .. (151.7351,328.7985) .. controls
    (151.7351,316.7226) and (152.9361,314.0577) .. (155.3198,306.9016) .. controls
    (159.1329,295.4541) and (164.2044,286.4239) .. (172.9551,277.6732) .. controls
    (181.7059,268.9225) and (191.1824,263.4046) .. (202.6300,259.5915) .. controls
    (209.7860,257.2078) and (212.2724,256.1854) .. (224.3482,256.1854) .. controls
    (236.4122,256.1854) and (240.2592,257.6562) .. (247.4271,260.0346) .. controls
    (266.8319,266.4735) and (282.1119,280.6459) .. (290.6083,298.9990) .. controls
    (294.0759,306.4892) and (296.7565,318.1889) .. (296.8726,324.0349) --
    (296.9619,328.5303) -- (270.2362,328.4118) .. controls (240.8531,328.2816) and
    (238.9373,328.9752) .. (231.5733,336.6814) .. controls (226.2406,342.2618) and
    (224.7754,347.2502) .. (224.2036,358.3607) .. controls (223.9320,363.6388) and
    (224.8057,368.5993) .. (224.1809,369.3842) .. controls (223.3288,370.4545) and
    (223.0926,370.7388) .. (224.0397,371.0575) .. controls (224.9718,371.3712) and
    (224.1682,375.1168) .. (224.2310,385.8670) -- (224.3203,401.1443) --
    (219.1343,401.1593) .. controls (216.7730,401.1184) and (211.3913,400.2985) ..
    (207.5022,399.3373) -- (207.5022,399.5159) -- cycle;
  \path[draw=black,fill=black,line join=round,line cap=round,miter limit=4.00,fill
    opacity=0.067,line width=0.000pt] (428.3337,384.5764) .. controls
    (428.5848,372.7754) and (427.8525,368.3393) .. (428.8646,367.1216) .. controls
    (430.0564,365.6876) and (430.7172,366.2742) .. (429.3194,366.2742) .. controls
    (427.9145,366.2742) and (427.7727,364.1364) .. (428.2768,357.3723) .. controls
    (429.0685,346.7493) and (430.0001,343.0399) .. (434.8723,337.7331) .. controls
    (442.1789,329.7746) and (444.4281,327.3365) .. (474.2797,328.6051) --
    (500.9412,328.5829) -- (500.5554,336.5876) .. controls (499.7888,352.4930) and
    (491.9781,368.0598) .. (479.7191,380.3188) .. controls (467.2617,392.7762) and
    (452.0303,401.5552) .. (435.7204,400.9961) -- (428.3053,401.2470) --
    (428.3337,386.2810) -- cycle;
   
   \draw (230,224) node {Type $A+$};
   \draw (430,224) node {Type $B+$}; 
   \draw (230,414) node {Type $A-$};
   \draw (430,414) node {Type $B-$};

\end{tikzpicture}

\caption{The four possibilities for the image of $u$	 around the surgered point. \\
The arrows correspond to the orientation of the boundary of $u$.}
\label{figure:behaviouroftheteardroparoundsurgeredpoint}
\end{figure}  	
 
\begin{proof}
Call $Q$ the shaded region in Figure \ref{figure:behaviouroftheteardroparoundsurgeredpoint}. We first show that $\Im(u) \subset Q$.
Assume the opposite, so there is $z \in C_i$ with $u(z) \notin Q$. Since $u$ is $j$-holomorphic and imnersed, there is $\tilde{z}$ such that $u(\tilde{z}) \in Q$. Since $C_i$ is path-connected, there is a point in the interior of $C_i$ which maps to $\Im(\gamma_+)$.

Since the map $u$ is immersed, the set $\Int(C_i) \cap u^{-1}(\gamma_+)$ is a finite union of embedded arcs. None of these arcs has ends on $\partial C_i$ since $\gamma_+$ is embedded. For the same reason, these arcs are disjoint. Call $B$ the closed region enclosed by the innermost arc. Then $u_{\vert B}$ is a local homeomorphism. Its image is a subset of $Q$. Since $B$ and $Q$ are compact, the map $u_{\vert B}$ is proper. Therefore, it is a connected covering of $Q$ which is simply connected: hence $u_{\vert B}$ is a homeomorphism. However, each point of $\gamma_+$ has two pre-images by $u_{\vert B}$, so that $u_{\vert B}$ can not be one to one.

Thus $u_{\vert C_i}: C_i \to Q$ is a proper local homeomorphism, hence a covering. Since $Q$ is simply connected, it is a homeomorphism.  
\end{proof}

The next step of the proof is the construction of an immersed holomorphic polygon $v$ with boundary on $\gamma$ and $\alpha$ by "filling the corners".

\begin{lemma}
\label{lemma:constructionofapolygonvwithcorners}
We use the notations above. There is an immersed holomorphic polygon 
\[ v : (\D,\partial \D) \to (S_g, \Im(\alpha) \cup \Im(\gamma)) \]
such that 
\begin{itemize}
	\item the map $v_{\vert v^{-1}(U_x)}$ is a reparameterization of $u_{\vert u^{-1}(U_x)}$,
	\item each connected component of $v^{-1}(U_x)$ is a disk,
	\item $v$ has a unique corner which maps to $x$ within each connected component of $v^{-1}(U_x)$,
	\item $v$ has one corner at $y$ and an odd number of corners at $x$.
\end{itemize}
\end{lemma}

\begin{proof}
The construction of $v$ is represented in Figure \ref{figure:FigureS}. Let $i \in \{1 \ldots N \}$. 

\input{Figures/FigureS.tex}
 
\paragraph{\textbf{First case:}} The map $u$ is of Type $A+$ or of Type $B-$ near  $A_i$ (see Figure \ref{figure:behaviouroftheteardroparoundsurgeredpoint} for the definition of Type). Consider the closed region $R_i$ at $x$ indicated in Figure \ref{figure:FigureS}. Choose a biholomorphism 
\[ v_i : \D \to R_i. \] 
This extends, by Caratheodory Theorem, to a homeomorphism
\[ v_i : \overline{\D} \to \overline{R}_i . \] 
We call $B_i$ the arc $v_i^{-1} (\gamma_{\pm})$. There is a unique map $\phi_i : B_i \to A_i$ such that
\[ \forall x \in B_i, \ u(\phi_i(x)) = v_i(x) . \]
We let $S = \D \cup_{\phi_i} \D$ be the surface obtained by gluing two copies of the disk along $\phi_i$. Then, $S$ can be endowed with a Riemann surface structure. For $x$ in the interior of one of the copies of $\D$, the chart is the natural one. If $z = \phi_i(w)$ belongs to $A_i$, choose a disk $D_x$ contained in $U_x$ and such that $u(z) \in D_x$. We put $U = u^{-1} \left (U_x \right) \cup_{\phi_i} v_i^{-1} \left(U_x \right)$. A chart around $x$ is given by the map
	\[ y \in U \mapsto \left\{ \begin{array}{c} u(y) \text{ if } y \in u^{-1} \left (U_x \right)  \\
 v(y) 	\text{ if } y \in v_i^{-1} \left (U_x \right) 
 \end{array} \right. .\]
Then, the applications $u$ and $v_i$ induce a map $v : S \to S_g$. By definition, it is indeed holomorphic. Since $v$ is holomorphic, its energy is the area of the image. Hence, it is finite. The surface $S$ is simply connected, hence biholomorphic to a disk. We conclude that $v$ is a holomorphic polygon with one more corner at $x$.

\paragraph{\textbf{Second case:}} The corner is of type $A-$ or $B+$. Denote by $A$ the red piecewise differentiable arc represented in Figure \ref{figure:FigureS}. Recall that $C_i$ is the connected component of $u^{-1}(U_x)$ which contains $A_i$. Then $Ar : = u^{-1}(A) \cap C_i$ is an embedded piecewise differentiable arc. Therefore, $\D \backslash Ar$ has two distinct simply-connected components $\Omega_1$ and $\Omega_2$. We can assume that $\Omega_1$ is the only one which contains $x$. The map $v$ is the restriction of $u$ to $\Omega_1$. 

We repeat this process for $i = 1 \ldots N$ to obtain an immersed holomorphic polygon with boundary on $\alpha$ and $\gamma$. It has one corner at $y$ and $N$ corners at $x$. Since the boundary of the polygon necessarily switches from $\alpha$ to $\gamma$ or $\gamma$ to $\alpha$ at each corner, the number $N$ is odd.
\end{proof}

Notice that the concatenation $\beta = \gamma \cdot \alpha$ is a continuous loop and as such can be regarded as an continuous map $\beta : S^1 \to S_g$. We assume that it is parameterized so that 
\begin{itemize}
	\item $\beta(i) = \beta(-i) =x $,
	\item $\beta(-1) = \beta(1) = y$, 
	\item $\beta$ restricted to the counterclockwise arc from $i$ to $-i$ is a reparameterization of $\gamma$,
	\item $\beta$ restricted to the counterclockwise arc from $-i$ to $i$ is a reparamaterization of $\alpha$.
\end{itemize}
Moreover, we introduce the following notations.
\begin{itemize}
	\item We let $x_-: [0,1] \to S^1$ be the counterclockwise arc of $S^1$ from $-1$ to $1$ and $\beta_-$ be the map $\beta \circ x_-$.
	\item We let $x_+: [0,1] \to S^1$ be the counterclockwise arc of $S^1$ from $1$ to $-1$ and $\beta_+$ be the map $\beta \circ x_+$.
\end{itemize}
From the construction of the polygon $v$, we see that we can lift its boundary to a continuous path $f : [0,1] \to S^1$ such that $\beta \circ f = v_{\vert \partial \D}$. Since $v$ has a corner at $x$, we have 
\begin{itemize}
	\item either $f(0) = -1, \ f(1) = 1 $,
	\item or $f(0) = 1, \ f(1) = -1$. 	
\end{itemize}

\paragraph{\bf{Case $f(0) = -1, f(1) = 1$}.} Since the concatenation $f \cdot x_-^{-1}$ is a loop based at $-1$, there is $k \in \Z$ such that $f \cdot x_-^{-1}$ is homotopic to $(x_- \cdot x_+)^k$ relative to $-1$. 

By a Theorem of Jaco (\cite{Jac70}), the subgroup of $\pi_1(S_g,y)$ generated by the classes of $\beta_-$ and $\beta_+$ is free. So there are three alternatives to consider.
\begin{enumerate}
	\item There are no relations between $\beta_-$ and $\beta_+$ in $\pi_1 \left(S_g,y \right)$ (so they generate a free group of rank $2$).
	\item There is $m \in \Z$ such that $\beta_+ = \beta_-^m$ in $\pi_1(S_g,y)$.
	\item There is $m \in \Z$ such that $\beta_- = \beta_+^m$ in $\pi_1(S_g,y)$.
\end{enumerate}
Since $\beta \circ f$ is the boundary of $v$, we have in $\pi_1(S_g,y)$ 
\begin{align*}
e  & = (\beta \circ f) \\
  & =(\beta_- \cdot \beta_+)^k \beta_-.	
\end{align*}
So the case $(1)$ can not hold. Therefore, we are in one of the cases $(2)$ or $(3)$.
\paragraph{\bf{Case $\beta_- = \beta_+^m$.}} Then, we have in $\pi_1(S_g,y)$
\begin{align}
e = & (\beta_- \cdot \beta_+)^k \beta_- \\
 = & \beta_+^{k(m+1) + m}.
\label{eqn:relation1} 	
\end{align}
If $\beta_+ = e$ in $\pi_1(S_g,y)$, then $\beta_- = e$ so that $\beta$ bounds a disk. Therefore, the surgery $\gamma \#_x \alpha$ bounds a disk and is contractible. This contradicts the hypothesis on $\gamma \#_x \alpha$. Therefore, 
\[ k(m+1) + m = 0. \] 
There are two solutions to this equation, $(m,k) = (0,0)$ or $(m,k) = (-2,-2)$.

If $m = k = 0$, then equation \ref{eqn:relation1} implies $\beta_- = e$. We conclude that $\gamma$ and $\alpha$ are not in minimal position by Lemma \ref{lemma:criterionforminimalposition}.

If $m = k = -2$, the conclusion follows from a combinatorial argument. Indeed, the boundary $f$ is homotopic to $x_+^{-1} \cdot x_-^{-1} \cdot x_+^{-1}$. We deduce that the polygon $v$ has three corners at $x$ of successive types $B-$, $B+$ and $B-$ (see Figure \ref{figure:behaviouroftheteardroparoundsurgeredpoint}). We assume that these corners are counterclockwise the image of $x_1$, $x_2$ and $x_3$.

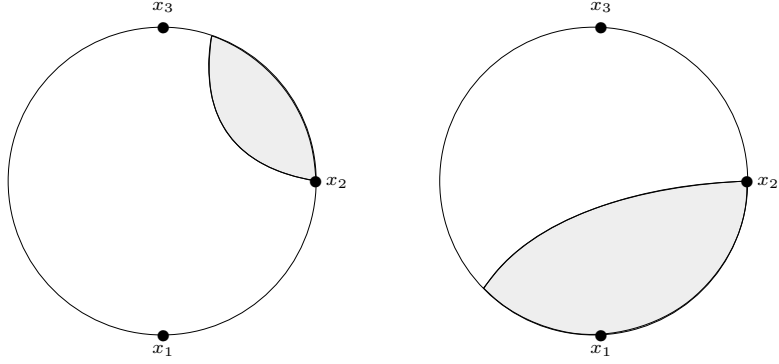
\begin{figure}
\captionsetup{justification=centering,margin=2cm}

\begin{tikzpicture}[y=0.80pt, x=0.80pt, yscale=-1.000000, xscale=1.000000, inner sep=0pt, outer sep=0pt]
  \path[xscale=1.000,yscale=-1.000,draw=black,line join=round,line cap=round,miter
    limit=4.00,fill opacity=0.156,line width=0.300pt] (259.4032,-712.3622) circle
    (2.0474cm);
  \path[xscale=1.000,yscale=-1.000,draw=black,line join=round,line cap=round,miter
    limit=4.00,fill opacity=0.156,line width=0.300pt] (463.4540,-712.3622) circle
    (2.0474cm);
  \path[draw=black,line join=miter,line cap=butt,miter limit=4.00,even odd
    rule,line width=0.300pt] (331.9614,712.3196) .. controls (294.4197,705.8017)
    and (276.5849,684.0752) .. (282.8427,643.7555) -- (282.8427,643.7555);
  \path[draw=black,line join=miter,line cap=butt,miter limit=4.00,even odd
    rule,line width=0.300pt] (535.9821,712.3622) .. controls (486.3541,714.3830)
    and (433.2244,728.6976) .. (411.5179,763.1658);
  \path[draw=black,fill=black,line join=round,line cap=round,miter limit=4.00,fill
    opacity=0.067,line width=0.000pt] (452.3600,784.1623) .. controls
    (438.5568,781.8975) and (424.6751,775.4381) .. (415.2538,766.6644) --
    (411.3230,763.0038) -- (416.3857,756.6121) .. controls (435.5564,732.4086) and
    (477.6484,715.5883) .. (527.4803,712.8197) -- (535.8594,712.3542) --
    (535.8338,716.7415) .. controls (535.7202,736.1919) and (524.7106,756.2176) ..
    (508.7683,768.9822) .. controls (493.5690,781.1519) and (471.0254,787.2249) ..
    (452.3600,784.1623) -- cycle;
  \path[draw=black,fill=black,line join=round,line cap=round,miter limit=4.00,fill
    opacity=0.067,line width=0.000pt] (329.2075,711.7048) .. controls
    (328.8147,711.6172) and (326.5949,711.2436) .. (324.2362,710.6817) .. controls
    (300.5681,705.0441) and (286.6869,691.1314) .. (282.6697,671.3514) .. controls
    (281.7531,666.8379) and (281.2075,659.0105) .. (281.4648,655.1474) .. controls
    (281.9452,647.9335) and (282.4545,643.9877) .. (282.9287,643.8058) .. controls
    (283.6567,643.5264) and (291.1391,646.9019) .. (295.3437,649.4065) .. controls
    (315.0279,661.1320) and (328.5070,680.6420) .. (331.3146,703.3545) .. controls
    (331.7494,706.8721) and (332.2822,712.2181) .. (331.8943,712.1581) .. controls
    (331.7470,712.1354) and (329.6004,711.7925) .. (329.2075,711.7048) -- cycle;
	\draw (260,786) node {$\bullet$};
	\draw (260,790) node[below] {\tiny $x_1$};
	\draw (260,640) node {$\bullet$};
	\draw (260,630) node {\tiny $x_3$};
	\draw (332,713) node {$\bullet$};
	\draw (342,713) node {\tiny $x_2$};
	
	\draw (467,786) node {$\bullet$};
	\draw (467,790) node[below] {\tiny $x_1$};
	\draw (467,640) node {$\bullet$};
	\draw (467,630) node {\tiny $x_3$};
	\draw (536,713) node {$\bullet$};
	\draw (546,713) node {\tiny $x_2$};
\end{tikzpicture}

\caption{The arc $A$ and the bigons delimited by $A$ (shaded)}
\label{figure:thebigonsandthearcA}
\end{figure}

Since $v$ is immersed and $\alpha$ is embedded and not contractible, the set $v^{-1}(\alpha)$ is a union of embedded arcs with endpoints on the boundary of $\D$. The corner at $x_2$ is of type $B+$, so there is one arc $A$ with an endpoint at $x_2$. Since $\alpha$ is embedded and $x_1$ and $x_2$ are of type $B-$, the other endpoint of $A$ belongs to one of the open boundary arcs $(x_2, x_3)$ or $[-1,x_1)$ (see Figure \ref{figure:thebigonsandthearcA}). In the left case, the map $v$ restricted to the bigon delimited by $A$ and the boundary (represented in Figure \ref{figure:thebigonsandthearcA}) yields a strip with boundary on $\alpha$ and $\gamma$. In the right case, since $x_2$ is of type $B-$, $v$ restricted to the shaded area is also a bigon. So $\alpha$ and $\gamma$ are not in minimal position.

\paragraph{\bf{Case $\beta_+ = \beta_-^m$.}} Then, in $\pi_1(S_g,y)$,
\begin{align}
	e = & (\beta_- \cdot \beta_+ )^k \beta_- \\
	  = & (\beta_-^{m+1})^k \beta_- \\
	  = & \beta_-^{k(m+1) + 1}.
\label{eqn:relation2}
\end{align}
If $\beta_- = e$, then $\beta_+$ also represents the neutral element. We deduce that the surgery $\gamma \#_x \alpha$ bounds a disk, which contradicts the hypothesis of the proposition. Hence, $k(m+1) +1 = 0 $ so $(k,m) = (1,-2)$ or $(k,m) = (-1,0)$. But if $m = 0$, then $\beta_+ = e$ so that $\alpha$ and $\gamma$ are not in minimal position by Lemma \ref{lemma:criterionforminimalposition}. Therefore, $(k,m) = (1,-2)$. The boundary $f$ of the polygon $v$ is homotopic to $x_- \cdot x_+ \cdot x_-$. So the polygon $v$ has three corners at $x$ of successive types $A_+$, $A_-$ and $A_+$. We call their respective pre-images $x_1$, $x_2$ and $x_3$

As before, $v$ is an immersion and $\alpha$ is embedded and not contractible, so $v^{-1}(\alpha)$ is a union of embedded arcs with endpoints on the boundary $\partial \D$. Since $x_2$ has type $A_-$, there is one arc with an endpoint at $x_2$ which we call $A$. Its other end lies either in the boundary arc $[-1,x_1)$ or in the boundary arc $(x_2,x_3)$ (see Figure \ref{figure:thebigonsandthearcA}). In each case $v$ restricts to an immersed bigon, so that $\gamma$ and $\alpha$ are not in minimal position, a contradiction.

\paragraph{\bf{Case $f(0) = 1, f(1) = -1$.}} Here, the concatenation $f \cdot x_-$ is a loop based at $1$. So there is $k \in \Z$ such that $f$ is homotopic relative endpoints to $(x_+ \cdot x_-)^k \cdot x_-^{-1}$.

Since we have, in $\pi_1(S_g,y)$ 
\[ (\beta_+ \cdot \beta_-)^k \beta_-^{-1} = e \]
We have either $\beta_- = \beta_+^m$ for some $m \in \Z$ or $\beta_+ = \beta_-^m$ for some $m \in \Z$.
 
Assume there is $m \in \Z$ such that $\beta_- = \beta_+^m$. The integer $m$ cannot be zero, otherwise $\gamma$ and $\alpha$ are not in minimal position. Then from $\beta \circ f = e$, we get $ (\beta_+)^{k(m+1) - m}  = e$. So $m = -2$ and $k = 2$. Therefore, $f$ is homotopic to $x_+ \cdot x_- \cdot x_+$. So the polygon $v$ has three successive corners at $x$ of successive types $A_-$, $A_+$ and $A_-$. We call their respective pre-images $x_1$, $x_2$ and $x_3$.

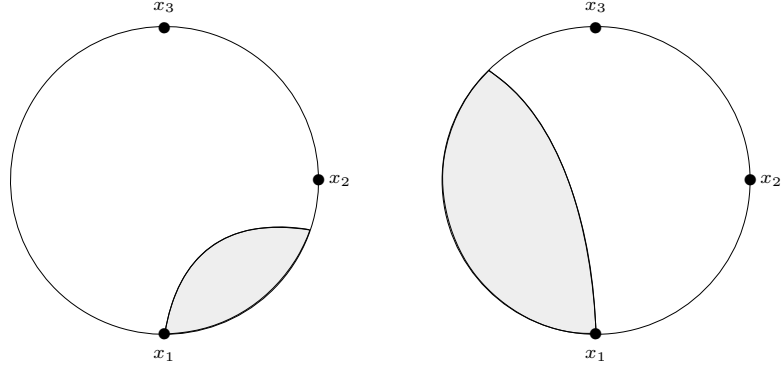
\begin{figure}
\captionsetup{justification=centering,margin=2cm}

\begin{tikzpicture}[y=0.80pt, x=0.80pt, yscale=-1.000000, xscale=1.000000, inner sep=0pt, outer sep=0pt]
  \path[cm={{0.5,-0.86603,-0.86603,-0.5,(0.0,0.0)}},draw=black,line
    join=round,line cap=round,miter limit=4.00,fill opacity=0.156,line
    width=0.300pt] (-57.5114,-154.5123) circle (2.0474cm);
  \path[cm={{0.0,1.0,1.0,0.0,(0.0,0.0)}},draw=black,line join=round,line
    cap=round,miter limit=4.00,fill opacity=0.156,line width=0.300pt]
    (127.0625,309.1067) circle (2.0474cm);
  \path[draw=black,line join=miter,line cap=butt,miter limit=4.00,even odd
    rule,line width=0.300pt] (105.0984,199.6207) .. controls (111.6164,162.0790)
    and (133.3429,144.2443) .. (173.6626,150.5020) -- (173.6626,150.5020);
  \path[draw=black,line join=miter,line cap=butt,miter limit=4.00,even odd
    rule,line width=0.300pt] (309.1067,199.5906) .. controls (307.0859,149.9626)
    and (292.7713,96.8329) .. (258.3031,75.1263);
  \path[draw=black,fill=black,line join=round,line cap=round,miter limit=4.00,fill
    opacity=0.067,line width=0.000pt] (237.3066,115.9685) .. controls
    (239.5714,102.1653) and (246.0308,88.2836) .. (254.8045,78.8623) --
    (258.4650,74.9315) -- (264.8568,79.9942) .. controls (289.0603,99.1649) and
    (305.8805,141.2569) .. (308.6491,191.0888) -- (309.1147,199.4679) --
    (304.7274,199.4423) .. controls (285.2770,199.3287) and (265.2513,188.3191) ..
    (252.4866,172.3768) .. controls (240.3170,157.1775) and (234.2439,134.6339) ..
    (237.3066,115.9685) -- cycle;
  \path[draw=black,fill=black,line join=round,line cap=round,miter limit=4.00,fill
    opacity=0.067,line width=0.000pt] (105.7132,196.8669) .. controls
    (105.8008,196.4740) and (106.1745,194.2543) .. (106.7363,191.8955) .. controls
    (112.3740,168.2274) and (126.2866,154.3463) .. (146.0666,150.3290) .. controls
    (150.5801,149.4124) and (158.4075,148.8668) .. (162.2707,149.1241) .. controls
    (169.4846,149.6045) and (173.4303,150.1138) .. (173.6123,150.5880) .. controls
    (173.8916,151.3160) and (170.5162,158.7984) .. (168.0116,163.0030) .. controls
    (156.2860,182.6873) and (136.7760,196.1664) .. (114.0636,198.9739) .. controls
    (110.5460,199.4087) and (105.1999,199.9415) .. (105.2599,199.5536) .. controls
    (105.2827,199.4063) and (105.6256,197.2597) .. (105.7132,196.8669) -- cycle;
	\draw (105,200) node {$\bullet$};
	\draw (178,127) node {$\bullet$};
	\draw (105,55) node {$\bullet$};
	
	\draw (105,210) node {\tiny $x_1$};
	\draw (188,127) node {\tiny $x_2$};
	\draw (105,45) node {\tiny $x_3$};
	
	\draw (309,200) node {$\bullet$};
	\draw (309,55) node {$\bullet$};
	\draw (382,127) node {$\bullet$};
	
	\draw (309,210) node {\tiny $x_1$};
	\draw (392,127) node {\tiny $x_2$};
	\draw (309,45) node {\tiny $x_3$};
\end{tikzpicture}

\caption{The arc $A$ and the bigons delimited by $A$ (shaded)}
\label{figure:thebigonsandthearcA2}
\end{figure}

The polygon $v$ is immersed and $\alpha$ is embedded an not contractible, so the set $v^{-1}(\alpha)$ is a union of embedded arcs with endpoints on $\partial \D$. Since $x_1$ is of type $A_-$, there is an arc $A$ with an endpoint at $x_1$. Since $\alpha$ is embedded, the other endpoint of $A$ is either on the boundary arc $[x_3,x_1]$ or on the arc $(x_1,x_2)$ (see Figure \ref{figure:thebigonsandthearcA2}). In both cases, $v$ restricted to the shaded area in Figure \ref{figure:thebigonsandthearcA2} is a bigon. So $\alpha$ and $\gamma$ are not in minimal position,  a contradiction.

Assume that there is $m \in \Z$ such that $\beta_+ = \beta_-^m$. This time, we have $(m,k) = (-2,-1)$. So the polygon $v$ has three successive corners at $x$ of types $B_-$, $B_+$, $B_-$. We call $x_1$, $x_2$ and $x_3$ their pre-images.

The arc $A \subset v^{-1}(\alpha)$ with an endpoint on $x_1$ has other endpoint either on $[x_3,x_1]$ or on the arc $(x_1,x_2)$ (see Figure \ref{figure:thebigonsandthearcA2}). In either case, $v$ restricted to the shaded area in Figure \ref{figure:thebigonsandthearcA2} is a bigon. So $\alpha$ and $\gamma$ are not in minimal position, a contradiction.
\end{proof}

\subsection{Obstruction of the surgery cobordisms}
We suppose that we are in the setting of Subsection \ref{subsection:surgeryofimmersedcurves}. There are 
\begin{itemize}
	\item an embedded curve $\alpha$,
	\item a generic curve $\gamma$ in minimal position with $\alpha$,
	\item an intersection point $x$ of degree $1$ between $\alpha$ and $\gamma$. 	
\end{itemize}

\begin{prop}
\label{prop:degenerationofthehandle}
Under the hypotheses of Subsection \ref{subsection:surgeryofimmersedcurves}, the immersed surgery cobordism
\[ V : (\gamma, \alpha) \leadsto \gamma \#_x \alpha ,\]
constructed in \ref{subsubsection:surgeryofintersectionpoints} does not bound a continuous polygon with a unique corner. 	
\end{prop}

\begin{proof}
Recall from Subsection \ref{subsection:surgeryofimmersedcurves}	that the cobordism $V: (\gamma,\alpha) \leadsto \gamma \#_x \alpha$ is an immersion $i : P \looparrowright \C \times S_g$ of a pair of pants $P$. The immersion $i$ is the smoothing of a piecewise smooth immersion $i^+ \sqcup j : P \looparrowright \C \times S_g$. Moreover, by Lemma \ref{lemma:degenerationofilambda} there is a homotopy $(i_\lambda)_{\lambda \in [0,1]}$ which interpolates between $i^+ \sqcup j$ and $i$ and is constant near the double points of $i$. See Figure \ref{figure:immersionofthesurgery} for the projections of these objects to the complex plane.

Assume there is a topological teardrop with boundary on $i$. So there are
\begin{itemize}
	\item a continuous map $u : (\D,\partial \D) \to (\C \times S_g,i(V))$,
	\item a continuous map $\tilde{u} :  (-\pi,\pi) \to S$ such that 
	\[ \forall \theta \in (-\pi,\pi), \  i \circ \tilde{u}(\theta) = u \left(e^{i\theta} \right), \]
	and
	\[ \lim_{t \to -\pi^+} \gamma(t) \ \neq \ \lim_{t \to \pi^-} \gamma(t). \]	
\end{itemize}
In particular $u(-1)$ is a double point of the immersion $i$ :
\[ u(-1) \in DP \cup DP_1 \cup DP_2 \cup DP_0,\] 
(see the end of \ref{subsubsection:surgeryofintersectionpoints} for the notations).

\paragraph{\bf{First step}} From Lemma \ref{lemma:degenerationofilambda}, we know that there is a continuous homotopy $(i_\lambda)_{\lambda \in [0,1]}$ from  $i_0 = i^+ \sqcup j$ to $i_1$ which is constant near the double points of $i$. For $\lambda \in [0,1]$, the path
\[ i_\lambda \circ \tilde{u} : (-\pi,\pi) \to S, \]
satisfies 
\[i_\lambda \circ \tilde{u}(-\pi) = i_\lambda \circ \tilde{u}(\pi).\]
We let $u_\lambda : \partial \D \to \C \times S_g$ be the path defined by
\[ \forall \theta \in (-\pi,\pi), \ u_\lambda \left( e^{i\theta} \right) = i_\lambda \circ \tilde{u}(\theta). \]
Let $\overline{A}(1,2) \subset \C$ be the closed annulus $\ens{z \in \C}{1 \leqslant \norm{z} \leqslant 2}$. We now glue the map
\[ \begin{array}{ccc}
 \overline{A}(1,2) & \to & \C \times S_g \\
 re^{i\theta} & \mapsto & u_{2-r} \left( e^{i\theta} \right)	
 \end{array},
 \]
along the boundary of $u$ to obtain a teardrop $v$ with boundary on $i_0$. Its lift to the domain of $i_0$ is the map $\tilde{u}$.

\paragraph{\bf{Second step}}
We let $C$ be the union of $\R^- \times \{0\}$ and $\{ 0 \} \times \R^+$. We let 
\[ p : \begin{array}{ccc}
 \R^2 & \to & C \\
 (x,y) & \mapsto &
 \left \{ \begin{array}{c}
 (0,x+y) \text{ if } y \leqslant -x \\
 (x+y,0) \text{ if } x \leqslant - y
 \end{array} \right. 	
 \end{array}. \]
The map $p$ is the continuous projection along the ray $\{ x = -y\}$ onto $C$.

We let 
\[p^- : \begin{array}{ccc}
 \R^2 & \to & \R^2 \\
 (x,y) & \mapsto & \left \{ \begin{array}{c}
 (x,y) \text{ if } x \leqslant 0, \ y \geqslant 0 \\
 p(x,y) \text{ else }
 \end{array} \right.	
 \end{array}
\] 
The map $w = p^- \circ v$ is a continuous teardrop with boundary on $i^+ \sqcup i^-$.

\paragraph{\bf{Third step}}%% For $t \in [0,1]$, we let
%%\[ p_t : \begin{array}{ccc} \ens{(x,y)}{x \leqslant 0, y \geqslant 0} & \to & \R^2 \\
%% (x,y) & \mapsto & \left\{ 
%% \begin{array}{c} \left( x + ty,(1-t)y \right) \text{ if } x + y \leqslant 0\\
%% ((1-t)x,y+tx) \text{ if } x+y \geqslant 0 \end{array} \right. 	
%% \end{array}. \]
%%The family $(p_t)_{t \in [0,1]}$ is a homotophy between the identity and $p_{\vert \ens{(x,y)}{x \leqslant 0, y \geqslant 0}}$. Recall from subsection \ref{subsubsection:surgeryofdoublepoints} that $c$ is the path which gives a local model for the handle of the surgery. For $t \in [0,1)$, we let 
%%\[ c_t : = p_t \circ c.  \]
%%This is a smooth path $c_t : \R \to \R$.

Recall from Subsubsection \ref{subsubsection:surgeryofdoublepoints} that, in the chart $\C \times U_x$, the immersion $i^+ \sqcup i^-$ coincides with the handle
\[ H_\varepsilon = \ens{\varepsilon c(t) z}{t \in \R, \ z = (x,y) \in S^1}, \] 
where $c$ is a smooth path interpolating between $\R$ and $i\R$.

We let
\[ A = \ens{\varepsilon c(0) z}{ \ z = (x,y) \in S^1} \]
be the core of the handle $H_\varepsilon$.

Denote by $\tilde{w} : (-\pi,\pi) \to S$ the lift of $w$ to the domain of $i^+ \sqcup i^-$ defined by
\[ \forall \theta \in (0,\pi), \ w \left(e^{i\theta} \right) =  i^+ \sqcup i^- (\tilde{w}(\theta)). \]
Both of the endpoints of $\tilde{w}$ do not belong to $A$. Therefore, we can homotope $\tilde{w}$ relative to its endpoints to a smooth path
\[ \tilde{a} : (-\pi,\pi) \to S  \]
which is transverse to the set $A$. Since it is homotopic to $\tilde{w}$, the map $i^+ \sqcup i^- \circ \tilde{a}$ bounds a smooth topological teardrop $a$.

\paragraph{\bf{Fourth step}} For $t \in [0,1]$, we let
\[ p_t :
\begin{array}{ccc}
\ens{(x,y)}{x \leqslant 0, \ y \geqslant 0} & \to & \R^2 \\
(x,y) & \mapsto &
\left\{ 
\begin{array}{c}
(x+ty, (1-t)y) \text{ if } x+y \leqslant 0 \\
((1-t)x, y + tx) \text{ if } x+y \geqslant 0 
\end{array}
\right.	
\end{array}
\]
This is a continuous family which interpolates between $p$ and the identity on the set 
\[ \ens{(x,y)}{x \leqslant 0, \ y \geqslant 0}. \] 
Notice that for $t \in [0,1)$ the paths
\[ c_t := p_t \circ c \]
are smooth.

Let $U$ be the set $\partial \D \cap a^{-1} (\C \times U_x)$. There are smooth functions 
\[ s : U_x \to \R, \ (x,y) : U \to S^1 \] 
such that
\[ \forall \theta \text{ such that } e^{i\theta} \in U_x, \ a(e^{i\theta}) = c \circ s(\theta) (x(\theta),y(\theta)). \]
Now, we attach the map
\[ \begin{array}{ccc}
 A(1,2) & \to & \C \times S_g \\
 re^{i\theta} & \mapsto & p_{r-1}\circ c \circ s (\theta) (x(\theta),y(\theta)) 
 \end{array}
 \]
along the boundary of $a$ to obtain a continuous polygon $\overline{b}$.

We let $b$ be the projection of $\overline{b}$ on the surface $S_g$
\[ b := p_{S_g} \circ \overline{b}. \]

\paragraph{\bf{Fifth step}} We build a non-constant teardrop on $\gamma_1 \#_x \gamma_2$ from $b$. Notice that for $t \in \partial \D$, we have $a(t) \in A$ if and only if $b(t) = x$. Let $t_0$ be such that $a(t_0) \in A$.

There is a connected open neighborhood $V$ of $t_0$ in $\partial \D$ such that $\forall t \in V, w(t) \in \C \times U_x$ and $b(t) \in U_x$. We write (in the Darboux chart for $U_x$ that we fixed earlier) $a(t) = (w_1(t),w_2(t))$ for $t \in V$. Then, there are smooth functions $s : V \to \R$, $x : V \to \R$ and $y: V \to \R$ with $s(t_0) = 0$ such that
\[ \forall t \in V, \ a(t) = c \circ s(t) (x(t),y(t)) . \]
Up to homotopy, we can always assume that $y'(t_0) \neq 0$. So 
\[ \forall t \in V, \ a'(t) = s'(t) c'\circ s(t) (x(t),y(t))  + c \circ s(t)(x'(t),y'(t)). \]  
At $a(t_o) = c(0) (x(0),y(0))$, the tangent space of $A$ is generated by $c(0) (-y(0),x(0))$. Hence, since $a'(t_0)$ is transverse to $A$, we have $s'(0) \neq 0$.

Assume $s'(t_0) > 0$. Then, $p \circ c(t) = (c_1(t) + c_2(t),0)$ for $ t_0 - \alpha < t < t_0$ and $p \circ c(t) = (0,c_1(t)+c_2(t)) $ for $t_0 + \alpha > t > t_0$ close enough to $t_0$. So we have
\[ b(t) = \left\{ \begin{array}{c} y (c_1 \circ s + c_2 \circ s) \text{ if } t_0 - \alpha <t < t_0 \\ i y (c_1 \circ s + c_2 \circ s) \text{ if } t_0 < t < t_0 + \alpha \end{array} \right. .\]  
In particular, the left and right derivative at $t_0$ are given by
\begin{align*}
b'(t_0^-) & =  y(t_0) (c_1' + c_2')(0) s'(t_0) \\
b'(t_0^+) & =  iy(t_0) (c_1' + c_2')(0) s'(t_0)	
\end{align*}
So if $y(t_0) < 0$, the path $b$ parameterizes the real line $\R_+$ in the opposite orientation followed by $i\R_-$ in the opposite orientation. If $y(t_0) > 0$, $b$ parameterizes $\R_-$ according to its orientation followed by $i \R_+$.

If $y(t_0) = 0$, we compute the second derivatives to get
\begin{align*}
b^{(2)}(t_0^-) & = 2y'(t_0) (c_1' + c_2')(0) s'(t_0) \\
b^{(2)}(t_0^+) & = 2 i y'(t_0) (c_1' + c_2')(0) s'(t_0)	
\end{align*}
So we easily see that the same conclusion holds.

Therefore, we can lift $b$ to a continuous path $d : \partial \D \to \gamma \#_x \alpha$ such that $p \circ d = b$. This path is homotopic (through the applications $p_t$) to $b$. Hence, it extends to a map $d : \D \to S_g$ with boundary on $\gamma \#_x \alpha$. The map $d$ is easily seen to be a teardrop. Hence, $\gamma \#_x \alpha$ is obstructed, a contradiction with the hypothesis.  
\end{proof}

Proposition \ref{prop:degenerationofthehandle} generalizes to the following.

\begin{prop}
\label{prop:degenerationofthehandle2}
Assume that $\gamma, \alpha_1, \ldots, \alpha_N$ are unobstructed curves. We, moreover, assume the following by induction.

We assume that $\gamma$ and $\alpha_1$ are transverse. We let $x_1$ be an intersection point between $\gamma$ and $\alpha_1$ of degree $1$. 

For $k \in \{ 1 \ldots n-1 \}$,	we assume that $(\gamma \#_{x_1} \alpha_1) \ldots \#_{x_k} \alpha_k$ is transverse to $\alpha_{k+1}$. We assume that these two curves are in minimal position. We let $x_k$ be an intersection point of degree $1$ between these curves.

Moreover, we assume that each of the curves $(\gamma \#_{x_1} \alpha_1) \ldots  \#_{x_k} \alpha_k$ is unobstructed for $k \in \{ 1 \ldots n \}$.

Then, the composition of the successive cobordisms 
\[\left ( (\gamma \#_{x_1} \alpha_1) \ldots  \#_{x_k} \alpha_k, \alpha_{k+1} \right ) \leadsto (\gamma \#_{x_1} \alpha_1) \ldots  \#_{x_{k+1}} \alpha_{k+1} \] 
does not admit any topological teardrop.
\end{prop}

\begin{proof}
The proof is a repeated application of the proof of the preceding proposition.	
\end{proof}

%%%%%%%%%%%%% PEUT ÊTRE EST-CE TOUJOURS VRAI ???? %%%%%%%%%%%%%%%%%%
%%\begin{prop}
%%\label{prop:gluingofunobstructedcobordismswithoneoutgoing}
%%We let $V_1, \ldots, V_N$ be immmersed Lagrangian cobordisms with unobstructed ends satisfying the following assumptions.
%%\begin{enumerate}
%%	\item Each cobordism $V_i$ has a \emph{unique} right end $\tilde{\gamma_i}$. We denote its left ends by $(\gamma_i^1, \ldots \gamma_i^{n_i})$.
%%	\item The left ends of $V_1$ and the right end of $V_N$ are \emph{embedded}.
%%	\item Each $V_i$ does not admit any topological teardrop.
%%	\item For $i \in \{1,\ldots,N-1\}$, the end $\tilde{\gamma}_i$ belongs to the set $(\gamma_{i+1}^1, \ldots \gamma_{i+1}^{n_{i+1}})$. 	
%%\end{enumerate}
%%Then, there exists an unobstructed Lagrangian cobordism
%%\[ V : (\gamma_1^1, \ldots, \gamma_1^{n_1}) \leadsto \tilde{\gamma}_N.\]
	
%%\end{prop}
We can deduce the following proposition.

\begin{prop}
\label{prop:cobordismassociatedtosuccessivesurgeries}
Assume that $\gamma$, $\alpha_1, \ldots, \alpha_N$ are as in Proposition \ref{prop:degenerationofthehandle2}. Then there is an unobstructed Lagrangian cobordism
\[ (\gamma, \alpha_1, \ldots, \alpha_N) \leadsto (\gamma \#_{x_1} \alpha_1) \ldots \#_{x_N} \alpha_N . \] 	
\end{prop}

\begin{proof}
First, we prove the following lemma which is a refinement of the construction in the proof of Lemma \ref{lemma:immersedcurvescobordanttogenericcurves}.

\begin{lemma}
\label{lemma:PerturbationofthecobordismV}
Assume that $i : V \looparrowright \C \times S_g$ is an immersed oriented Lagrangian cobordism with embedded ends such that
\begin{enumerate}
\item the set of double points 
 \[ \ens{(x,y) \in V \times V}{i(x) = i(y), \ x \neq y} \] 
 is a finite disjoint union of embedded intervals $I_k$	 (with $k = 1 \ldots N$) and of points ${(x_p,y_p)}$ (with $p = 1 \ldots M$),
\item the immersion $i$ restricts to an embedding of the intervals $I_k$ for $k = 1, \ldots, N$,
\item if $(x,y) \in I_k$ for some $k$, we have \[ \dim(di_x(T_x V) \cap di_y(T_y V) ) =1, \]
\item the space \[ di_{x_p}(T_{x_p} V) \cap di_{y_p}(T_{y_p} V )\] is null for $p \in \{1, \ldots M \}$.
\end{enumerate}
Then there is a smooth family of immersions $i_t : V \looparrowright \C \times S_g$ for $t \in [0,1]$ such that the following properties hold.
\begin{enumerate}[label = (\roman*)]
	\item We have $i_0 = i$.
	\item For all $t \in [0,1]$, the immersion $i_t$ coincides with $V$ outside of a compact subset.
	\item For almost all $t \in [0,1]$, the double points of $i_t$ are transverse.
	\item If $(x,y) \in V$ are such that $i_1(x) = i_1(y)$, then there are
	\begin{itemize} 
		\item a smooth path $\gamma = (x,y) : [0,1] \to V \times V$ with $\gamma(1) = (x,y)$,
		\item  a function $f: [0,1] \to [0,1]$ with $f(0) = 0$ and $f(1) = 1$, 	
	\end{itemize}
such that $i_{f(t)}(x) = i_{f(t)}(y)$ for all $t \in [0,1]$.
\end{enumerate}
\end{lemma}

\begin{proof}[Proof of Lemma \ref{lemma:PerturbationofthecobordismV}]
We extend the immersion $i : V \to \C \times S_g$ to a Weinstein embedding $ \Phi : T_\varepsilon^*V \to \C \times S_g $ . We let $K \subset V$ be a compact subset such that the image of $i_{\vert V \backslash K}$ is the disjoint union 
\[ \bigcup_{i=1 \ldots n} (\R^- \times \gamma_i) \cup \bigcup_{j=1 \ldots m} (\R^+ \times \gamma_j). \]

There is a smooth function $f : V \to \R$ such that the following holds.
\begin{enumerate}[label=(\Alph*)]
	\item The function $f$ is null outside of $K$.
	\item Let $\pi : T_\varepsilon^* V \to V$ be the standard projection. We let $X_{f \circ \pi}$ be the hamiltonian vector field of $f \circ \pi$. For $(x,y) \in I_k$ for some $k \in \{ 1, \ldots N \}$, the vector
	\[ d\Phi_x(X_{f \circ \pi}(x)) - d\Phi_y(X_{f \circ \pi}(y)) \] 
	is transverse to the vector space $di(T_x V) + di(T_y V)$.
\end{enumerate}
To see this, choose a vector field $X$ on each $\sqcup_{k} i(I_k)$ such that 
\[ di(T_x V)+ di(T_y V) + X(i(x)) = T_{i(x)} (\C \times S_g) \] 
for each $(x,y) \in I_k$. Choose disjoint neighborhoods $D_k$ of $I_k$ diffeomorphic to disks. We extend $X$ to a vector field on $\sqcup D_k$. There is a smooth function $f$ such that $X_{f \circ \pi} = X$ on $\sqcup D_k$. Now extend it to $V$ using a smooth cut-off function.

We choose an increasing, smooth, cut-off function $\beta : [0,\varepsilon] \to [0,1]$ such that $\beta(t) = 1$ for $t \in \left[0,\frac{\varepsilon}{3} \right]$ and $\beta(t) = 0$ for $t \in \left[ \frac{2 \varepsilon}{3},1 \right]$. We let $g$ be the smooth function given by 
\[ g: \begin{array}{ccc}
T^*_\varepsilon V & \to & \R \\
 	(x,v) & \mapsto & \beta(\norm{v}) f \circ \pi(x,v)
 \end{array}
 .\] 
We denote by $\phi^t_g : T^*_\varepsilon V \to T^*_\varepsilon V$ the flow of $X_g$ at the time $t \in [0,1]$. 

We claim that for $\eta > 0$ small enough, the map
\[ \Psi : \begin{array}{ccc} L \times L \times [- \eta,\eta] & \to & (\C \times S_g) \times (\C \times S_g) \\ 
(x,y,t) & \mapsto & \left(\Phi \left(\phi^t_g (x) \right),\Phi \left(\phi^t_g (y) \right) \right) 
\end{array}. \]
 is transverse to the diagonal $\Delta = \ens{(z,z)}{z \in \C \times S_g}$.

Indeed, there is $\eta >0$ such that
\begin{enumerate}[resume*]
	\item if $\Psi(x) = \Psi(y)$ with $x \neq y$, then $\dim \left[ d\Phi \left(T_x V \right) \cap  d\Phi \left(T_y V \right) \right] \leqslant 1$,
	\item if $\Psi(x) = \Psi(y)$ with $x \neq y$, then $d\Phi(X_g(\phi^t(x))) - d\Phi(X_g(\phi^t(y))) \notin d\Phi(T_xV) +  d\Phi(T_y V) $. 	
\end{enumerate}
This is seen by using assertions $(A)$ and $(B)$, the compactness of $K$ and lower semi-continuity of the rank. 

Let $(x,y,t) \in \Psi^{-1}(\delta)$ and $v \in T_{\Psi(x,y,t)}(\C \times S_g)$. Due to the preceding assumption, there are $v_1 \in T_x V, v_2 \in T_y V $ and $\lambda \in \R$ such that
\[ d\Phi (d\phi^t_g(v_1)) - d\Phi (d \phi^t_g(v_2)) + \lambda \left [d\Phi \left( X_g \circ \phi^t_g(x) \right) - d\Phi\left(X_g \circ \phi^t_g(y) \right) \right] \]
is equal to $v$.

Now we conclude by the following claim (which we learned from \cite{McDS12}) whose proof is an easy exercise.
\paragraph{\bf{Claim}} Let $h = (h_1,h_2) : M \to N \times N $ be smooth map. Let $x$ be such that $h_1(x) =h_2(x)$. $h$ is transverse to the diagonal if and only if for all $v \in T_{h_1(x)} N$, there is $w \in t_xM$ such that $d(h_1)_x(w) - d(h_2)(x) (w) = v$.

So the the set $\Psi^{-1}(\Delta)$ has a natural structure of compact 1-dimensional manifold. In particular, it has a finite number of connected components which are compact as well. Hence, for $\alpha > 0$ small enough, all the connected components of $V \times V \times [-\alpha,\alpha] \cap \Psi^{-1}(\Delta)$ have non empty intersection with $V\times V \times {0}$.

The projection 
\[\begin{array}{ccc} \Psi^{-1}( \Delta ) & \to  & \R \\
	(x,y,t) & \mapsto & t \\
\end{array}
 \]
is smooth. Hence, there is a regular value $ 0 < t_0 < \alpha$. The family $\Phi \circ \phi^t_g$ for $0 \leqslant t \leqslant t_0$ satisfies the conclusion of the lemma.
\end{proof}

We let $i : V \looparrowright \C \times S_g$ be the immersed Lagrangian cobordism
\[ (\gamma, \alpha_1, \ldots, \alpha_N) \leadsto (\gamma \#_{x_1} \alpha_1) \ldots \#_{x_N} \alpha_N, \]
given by Proposition \ref{prop:degenerationofthehandle2}. From Subsection \ref{subsection:surgeryofimmersedcurves}, the immersion $V$ satisfies the hypotheses of Lemma \ref{lemma:PerturbationofthecobordismV}. Hence, there is a family $(i_t)_{t \in [0,1]}$ which satisfies the above properties $(i)$, $(ii)$, $(iii)$ and $(iv)$.

Assume there is a continuous teardrop
\[ u : (\D, \partial \D) \to (M,i_1(V)) \]
with boundary on the immersion $i_1$. In particular, there is a path
\[ \tilde{u} : (-\pi,\pi) \to V \]
such that
\begin{align*} \forall \theta \in (-\pi,\pi), \ u \left(e^{i\theta} \right) = i_1 \circ \tilde{u} (\theta), \\ 
x = \lim_{\theta \to \pi^-} \tilde{u} (\theta) \neq \lim_{\theta \to -\pi^+} \tilde{u} (\theta) = y .
\end{align*}
We call $f$, $\gamma = (x,y)$ the smooth functions provided by the point $(iv)$ of Lemma \ref{lemma:PerturbationofthecobordismV}. There is a continuous family of paths $(\gamma_t)_{t \in [0,1]} : [-\pi,\pi] \to V$ such that
\[ \forall t \in [0,1], \ \gamma_t(-\pi) = x(t), \ \gamma_t(\pi) = y(t), \]
and
\[	\gamma_1 = \tilde{u}.
\]
We glue the map
\[ \begin{array}{ccc}
 A(1,2) & \to & \C \times S_g \\
 re^{i\theta} & \mapsto & i_{f(2-r)} \left(\gamma_{2-r}(\theta) \right)	
 \end{array}, \]
 along the teardrop $u$ to obtain a topological teardrop with boundary on the immersion $i$. This does not exist by hypothesis. Therefore, there are no topological teardrops on $i_1$.
\end{proof}

\subsection{Action of the Mapping Class Group and proof of Theorem \ref{Theo:Computationofunobgroup}}

First, we need the following lemma.
\begin{lemma}
\label{lemma:Cylinderoverseparatingcurves}
Let $\gamma_1$ and $\gamma_2$ be two isotopic separating curves. There is $x \in \R$ such that
\[ [\gamma_2] - [\gamma_1] = i(x), \]
in $\Gimmunob$. 	
\end{lemma}

\begin{figure}

\begin{tikzpicture}[y=0.80pt, x=0.80pt, yscale=-1.000000, xscale=1.000000, inner sep=0pt, outer sep=0pt]
  \path[draw=black,line join=miter,line cap=butt,miter limit=4.00,even odd
    rule,line width=0.600pt] (161.0146,167.2858)arc(90.000:30.000:23.676151 and
    27.146)arc(30.000:-30.000:23.676151 and 27.146)arc(330.000:270.000:23.676151
    and 27.146);
  \path[draw=black,line join=miter,line cap=butt,miter limit=4.00,even odd
    rule,line width=0.600pt] (335.3685,169.1659)arc(90.000:150.000:23.676151 and
    27.146)arc(150.000:210.000:23.676151 and 27.146)arc(210.000:270.000:23.676151
    and 27.146);
  \path[draw=black,fill=black,line join=round,line cap=round,miter limit=4.00,fill
    opacity=0.065,line width=0.600pt] (161.5039,193.6260) ellipse (0.2423cm and
    0.7412cm);
  \path[draw=black,line join=round,line cap=round,miter limit=4.00,line
    width=0.600pt] (335.2111,61.9805)arc(267.151:328.982:8.586297 and
    26.264)arc(-31.018:30.813:8.586297 and 26.264)arc(30.813:92.644:8.586297 and
    26.264);
  \path[draw=black,line join=round,line cap=round,miter limit=4.00,line
    width=0.598pt] (334.9422,169.1908)arc(270.000:360.000:8.588736 and
    26.087)arc(-0.000:90.000:8.588736 and 26.087);
  \path[draw=black,dash pattern=on 1.80pt off 1.80pt,line join=round,line
    cap=round,miter limit=4.00,line width=0.600pt] (252.0064,143.1759) ellipse
    (0.2630cm and 1.5531cm);
  \path[draw=black,fill=black,line join=round,line cap=round,miter limit=4.00,fill
    opacity=0.065,line width=0.600pt] (161.7763,86.7861) ellipse (0.2423cm and
    0.7412cm);
  \path[draw=black,line join=miter,line cap=butt,miter limit=4.00,even odd
    rule,line width=0.600pt] (161.7348,60.5044) .. controls (194.6981,60.5044) and
    (220.4304,88.3055) .. (252.1479,88.3055) .. controls (285.8040,88.3055) and
    (302.4833,61.9457) .. (335.4855,61.9457);
  \path[draw=black,line join=miter,line cap=butt,miter limit=4.00,even odd
    rule,line width=0.600pt] (161.3866,219.8407) .. controls (199.6313,219.8407)
    and (218.0031,198.2030) .. (251.7168,198.2030) .. controls (284.1451,198.2030)
    and (307.3566,221.2973) .. (335.0417,221.2973);
  \path[draw=black,line join=round,line cap=round,miter limit=4.00,line
    width=0.600pt] (252.0064,88.1445)arc(270.000:360.000:9.318095 and
    55.031)arc(-0.000:90.000:9.318095 and 55.031);
  \path[draw=black,line join=round,line cap=round,miter limit=4.00,fill
    opacity=0.000,line width=0.620pt]
    (322.4255,165.0661)arc(0.000:60.000:75.750778 and
    5.322)arc(60.000:120.000:75.750778 and 5.322)arc(120.000:180.000:75.750778 and
    5.322);
  \path[draw=black,line join=round,line cap=round,miter limit=4.00,fill
    opacity=0.000,line width=0.602pt]
    (312.4094,150.9326)arc(0.000:60.000:64.641510 and
    5.412)arc(60.000:120.000:64.641510 and 5.412)arc(120.000:180.000:64.641510 and
    5.412);
  \path[draw=black,line join=round,line cap=round,miter limit=4.00,fill
    opacity=0.000,line width=0.603pt]
    (312.8598,132.7284)arc(4.758:88.992:64.510658 and
    5.357)arc(88.992:173.226:64.510658 and 5.357);
  \path[draw=black,line join=round,line cap=round,miter limit=4.00,fill
    opacity=0.000,line width=0.538pt]
    (324.4296,117.6542)arc(6.879:89.263:75.553932 and
    5.320)arc(89.263:171.648:75.553932 and 5.320);
  \path[draw=black,dash pattern=on 2.50pt off 5.00pt,line join=round,line
    cap=round,miter limit=4.00,fill opacity=0.000,line width=0.625pt]
    (173.8657,117.0170)arc(180.000:270.000:75.553932 and
    5.320)arc(-90.000:0.000:75.553932 and 5.320);
  \path[draw=black,dash pattern=on 2.41pt off 4.83pt,line join=round,line
    cap=round,miter limit=4.00,fill opacity=0.000,line width=0.603pt]
    (184.0608,132.2840)arc(180.000:270.000:64.510658 and
    5.357)arc(270.000:360.000:64.510658 and 5.357);
  \path[draw=black,dash pattern=on 2.41pt off 4.81pt,line join=round,line
    cap=round,miter limit=4.00,fill opacity=0.000,line width=0.602pt]
    (183.1264,150.9326)arc(180.000:270.000:64.641510 and
    5.412)arc(-90.000:0.000:64.641510 and 5.412);
  \path[draw=black,dash pattern=on 2.48pt off 4.96pt,line join=round,line
    cap=round,miter limit=4.00,fill opacity=0.000,line width=0.620pt]
    (170.9240,165.0661)arc(180.000:270.000:75.750778 and
    5.322)arc(-90.000:0.000:75.750778 and 5.322);
  \path[draw=black,fill=black,line join=round,line cap=round,miter limit=4.00,line
    width=0.268pt] (255.4423,98.9005) .. controls (256.3314,97.9200) and
    (256.3161,96.1820) .. (256.2634,94.2843) .. controls (257.2066,96.0023) and
    (258.0857,97.4694) .. (259.3051,97.8525) -- cycle;
  \path[draw=black,fill=black,line join=round,line cap=round,miter limit=4.00,line
    width=0.268pt] (288.7096,167.7695) .. controls (287.9502,168.8535) and
    (286.2520,169.2234) .. (284.3896,169.5920) .. controls (286.2739,170.1316) and
    (287.8992,170.6641) .. (288.5426,171.7684) -- cycle;
  \path[draw=black,fill=black,line join=round,line cap=round,miter limit=4.00,line
    width=0.268pt] (276.0121,153.6803) .. controls (276.7731,154.7633) and
    (278.4719,155.1307) .. (280.3347,155.4967) .. controls (278.4513,156.0389) and
    (276.8267,156.5738) .. (276.1849,157.6790) -- cycle;
  \path[draw=black,fill=black,line join=round,line cap=round,miter limit=4.00,line
    width=0.268pt] (270.6701,135.2300) .. controls (271.3871,136.3426) and
    (273.0698,136.7777) .. (274.9166,137.2179) .. controls (273.0129,137.6844) and
    (271.3683,138.1538) .. (270.6827,139.2324) -- cycle;
  \path[draw=black,fill=black,line join=round,line cap=round,miter limit=4.00,line
    width=0.268pt] (275.3585,120.0146) .. controls (274.6414,121.1272) and
    (272.9588,121.5623) .. (271.1120,122.0025) .. controls (273.0157,122.4689) and
    (274.6603,122.9384) .. (275.3459,124.0170) -- cycle;
    
  \draw (327,162) node {\tiny $\alpha_1$};  
  \draw (320,133) node {\tiny $\alpha_3$};
  \draw (175,151) node {\tiny $\alpha_2$};
  \draw (168,120) node {\tiny $\alpha_4$};

\end{tikzpicture}
\caption{The curves $\alpha_1$, $\alpha_2$, $\alpha_3$, $\alpha_4$ }
\label{figure:FigureT}
\end{figure}
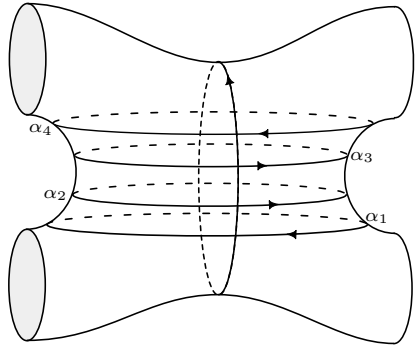

\begin{proof}
We are in the situation of Figure \ref{figure:FigureT}. 
We perform four successive surgeries along four isotopic curves $\alpha_1$, $\alpha_2$, $\alpha_3$ and $\alpha_4$  represented in Figure \ref{figure:FigureT}. The surgeries with $\alpha_1$ and $\alpha_3$ resolve the intersection point in the front of the surface. Meanwhile, the surgeries with $\alpha_2$ and $\alpha_4$ resolve the intersection point in the back of the surface.

The end-product is a curve $\beta$ isotopic to $\gamma_1$ whose holonomy is
\[ \Hol_A(\gamma_1) + \Hol_A(\alpha_1) + \Hol_A(\alpha_2) + \Hol_A(\alpha_3) + \Hol_A(\alpha_4). \]
Furthermore, by Proposition \ref{prop:cobordismassociatedtosuccessivesurgeries}, there is an unobstructed Lagrangian cobordism
\[ V : (\alpha_3, \alpha_4, \gamma_1, \alpha_1, \alpha_2 ) \leadsto \beta . \]
Since $\alpha_k$ for $k \in \{1, \ldots, 4 \}$ are non-seperating, we have by Lemma \ref{lemma:classofanonseparatingcurve},
\begin{align*}
\alpha_1 + \alpha_2 = i \left[ \Hol_A(\alpha_1) + \Hol_A(\alpha_2) \right], \\
\alpha_3 + \alpha_4 = i \left[ \Hol_A(\alpha_3) + \Hol_A (\alpha_4) \right]. 	
\end{align*}
So
\[ \beta = \gamma_1 + i \left[\Hol_A(\alpha_3 + \alpha_4) + \Hol_A(\alpha_2 + \alpha_1) \right].\]
From this and Lemma \ref{lemma:holonomyequalshamiltonianinvariant}, we deduce that there is $\varepsilon > 0$ such that all curve $\gamma$ isotopic to $\gamma_1$ with 
\[ \norm{\Hol_A(\gamma) - \Hol_A(\gamma_1)} < \varepsilon \]
satisifies in $\Gimmunob$
\[ \gamma - \gamma_1 = i(\Hol_A(\gamma) - \Hol_A(\gamma_1)) .\]
Let $S$ be the set of $\varepsilon > 0$ with this property. We let $\varepsilon_{\gamma_1}$ be the supremum of $S$.

Choose a smooth isotopy $ t \mapsto \gamma_t$ between $\gamma_0$ and $\gamma_1$. Since $[0,1]$ is compact, one can build a finite sequence 
\[ 0 = t_0 < t_1 < \ldots < t_N = 1\]
such that
\[ \forall i \in \{1, \ldots, N-1 \} , \ \norm{\Hol_A(\gamma_{t_i}) -\Hol_A(\gamma_{t_{i+1}})} < \varepsilon_{\gamma_{t_i}} .\]
So the conclusion follows. 
\end{proof}

The following Proposition is the analog of Proposition \ref{prop:ActionofaDehntwistonGimm} for $\Gimmunob$.

\begin{prop}
\label{prop:ActionofaDehntwistonGimmunob}
Let $\alpha: S^1 \to S_g$ and $\beta: S^1 \to S_g$ be two embedded curves. Then, there is $x \in \R$ such that
\[ [T_\alpha(\beta)] = [\beta] + (\beta \cdot \alpha) [\alpha] + i(x) \]
in $\Gimmunob$. 
\end{prop}

\begin{proof}
Assume that $\alpha$ and $\beta$ are in minimal position. We construct a representative $\gamma$, up to isotopy, of $T_\alpha(\beta)$ using the procedure of Proposition \ref{prop:ActionofaDehntwistonGimm}.

We show that the successive surgeries of the proof of Proposition \ref{prop:ActionofaDehntwistonGimm} are unobstructed. 

Recall that for $k \in \{ 1, \ldots, N \}$, the curve $c_k$ is obtained from $k$ surgeries along the curves $\tilde{\alpha}_1, \ldots, \tilde{\alpha}_k$. The curve $\tilde{\alpha}_{k+1}$ is a perturbation of $\alpha$.  We also fixed Darboux charts $\phi_m$ near each intersection point $x_m$.

\begin{lemma}
\label{lemma:curvesinsuccessivesurgeriesareinminimalposition}
In the above setting, the curves $\tilde{\alpha}_{k+1}$ and $c_k$ are in minimal position.
\end{lemma}

\begin{proof}
The proof proceeds by contradiction. Let us start with an outline of the proof.
\begin{enumerate}
	\item We first show that there is a bigon $v$ between $\tilde{\alpha}_{k+1}$ and $\beta$ whose boundary arc on $\tilde{\alpha}_{k+1}$ is embedded.
	\item This bigon has a precise behavior near the surgered points described in Figure \ref{figure:FigureP}.  
	\item Then, we build from $v$ a bigon between $\alpha$ and $\beta$ following the procedure represented in \ref{figure:FigureR}.
\end{enumerate}
So $\alpha$ and $\beta$ are not in minimal position, this contradicts the hypothesis.

Assume there is a bigon $u$ between $\tilde{\alpha}_{k+1}$ and $c_l$. Since $\tilde{\alpha}_{k+1}$ is embedded, the set $u^{-1}(\tilde{\alpha}_{k+1})$ is a union of embedded arcs. We take an innermost such arc $A$. We let $v : \D \to S_g$ be the immersed bigon delimited by this arc. We parameterize $v$ so that the preimages of its corners are $-1$ and $1$. 

Notice that the bigon $v$ is immersed. If one of its corners is non convex, the set $v^{-1}(\tilde{\alpha}_{k+1})$ has non-empty intersection with $\Int(\D)$, so $A$ is not innermost. Hence, $v$ has convex corners. 

We fix a (regular) parameterization $\gamma : S^1 \to S_g$ of the curve $c_k$. There are two by two disjoint arcs of $A_1, \ldots, A_k, B_1, \ldots, B_k \subset S^1$ satisfying the following properties
\begin{itemize}
	\item for $m \in \{ 1, \ldots, k \}$, $\gamma_{\vert A_m}$ parameterizes an arc of $\tilde{\alpha}_m$ ,
	\item for $m \in \{ 1, \ldots, l \}$, $\gamma_{\vert B_m}$ parameterizes an arc of $\beta$ 
\end{itemize}
Notice that the arcs $A_1, \ldots, A_k$ and $B_1, \ldots, B_k$ are intertwined. Moreover, the complements of these in $S^1$ is a union of arcs which parameterize the handles of the successive surgeries. We call these arcs $C_1, \ldots, C_{2k}$.

Let $A$ be the arc of $\partial \D$ which maps into $c_k$ through $v$. Choose an immersed lift $\lambda : A \to S^1$ such that $\gamma \circ \lambda = v_{\vert A}$. Whenever the map $\lambda$ parameterizes one of the arcs $C_i$, we say that $v$ has a switch.

There is at least one switch. Assume otherwise. Then the image of $\lambda$ is contained in one of the $A_i$ or one of the $B_i$. If it is a subset of one of the $A_i$, then $v$ yields an immersed strip with convex corners between $\tilde{\alpha}_i$ and $\tilde{\alpha}_{k+1}$. There are only two such strips, and both of these are not strips on the surgeries. If the image of $\lambda$ is a subset of one of the $B_i$, then $v$ is an immersed bigon between $\tilde{\alpha}_{k+1}$ and $\beta$. Hence these curves are not in minimal position, a contradiction.

There are at most two branch switches. If not, there is $m$ such that $A_m \subset \Im(\lambda)$. So $\Im(c_l \circ\lambda )$ contains one of the vertical lines in the charts $\phi_k$. Hence, it must intersect $\tilde{\alpha}_{k+1}$ in this chart, a contradiction.

\input{Figures/FigureP.tex}

The possible behaviors of the curve $v$ at a switch are described in Figure \ref{figure:FigureP}. Using the techniques of the proof of Proposition \ref{prop:asurgeryisnonobstructed}, it is an easy exercise to show that these are indeed the only possible cases.

\paragraph{\textbf{First case}} One of the corners maps to a point $x_m^k$ for some $m$. Moreover, if we parameterize $v$ so that $v(-1) = x_k^m$, the lower arc is mapped to $c_k$ and the upper boundary arc to $\tilde{\alpha}_{k+1}$.

\begin{figure}
\definecolor{cff0000}{RGB}{255,0,0}
\captionsetup{justification=centering,margin=2cm}

\begin{tikzpicture}[y=0.80pt, x=0.80pt, yscale=-1.000000, xscale=1.000000, inner sep=0pt, outer sep=0pt]
    \path[xscale=1.000,yscale=-1.000,draw=black,line join=round,line cap=round,miter
      limit=4.00,fill opacity=0.156,line width=0.500pt] (163.6447,-140.1945) circle
      (2.0474cm);
    \path[draw=black,line join=miter,line cap=butt,miter limit=4.00,even odd
      rule,line width=0.500pt] (163.8973,67.4635) -- (163.8973,212.6729);
    \path[draw=cff0000,line join=miter,line cap=butt,miter limit=4.00,even odd
      rule,line width=0.500pt] (182.4549,210.1497) .. controls (182.4549,140.8309)
      and (145.0668,144.2529) .. (145.0668,69.9535);
    \path[fill=black,line join=round,line cap=round,miter limit=4.00,fill
      opacity=0.067,line width=0.500pt] (157.6106,130.0569) .. controls
      (155.2133,125.0480) and (152.7925,118.8426) .. (152.0129,116.4855) .. controls
      (148.0259,104.4315) and (145.6048,89.1299) .. (145.3032,77.2138) --
      (145.1246,70.1582) -- (149.6782,69.0680) .. controls (151.7220,68.5787) and
      (155.6758,67.7784) .. (158.3389,67.8627) -- (163.7853,67.7476) --
      (163.7853,103.3665) .. controls (163.7853,122.8587) and (163.8566,141.1283) ..
      (163.7255,141.1283) .. controls (163.5944,141.1283) and (160.0079,135.0658) ..
      (157.6106,130.0569) -- cycle;
    \path[fill=black,line join=round,line cap=round,miter limit=4.00,fill
      opacity=0.067,line width=0.500pt] (164.0009,177.1105) -- (163.9486,142.6462)
      -- (169.5845,152.8195) .. controls (174.5757,161.8291) and (177.8590,172.9309)
      .. (179.6291,181.6316) .. controls (181.1883,189.2962) and (181.9785,200.5838)
      .. (182.1325,205.6819) -- (182.2674,210.1462) -- (177.6246,211.6845) ..
      controls (175.5289,212.3788) and (171.5076,212.3225) .. (169.0531,212.3988) --
      (163.9638,212.5569) -- (164.0008,178.8069) -- cycle;
    \path[xscale=1.000,yscale=-1.000,draw=black,line join=round,line cap=round,miter
      limit=4.00,fill opacity=0.156,line width=0.500pt] (361.5019,-140.1945) circle
      (2.0474cm);
    \path[draw=black,line join=miter,line cap=butt,miter limit=4.00,even odd
      rule,line width=0.500pt] (361.7544,67.4635) -- (361.7544,212.6729);
    \path[draw=cff0000,line join=miter,line cap=butt,miter limit=4.00,even odd
      rule,line width=0.500pt] (380.3121,210.1497) .. controls (380.3121,140.8309)
      and (342.9239,144.2529) .. (342.9239,69.9535);
    \path[fill=black,line join=round,line cap=round,miter limit=4.00,fill
      opacity=0.067,line width=0.500pt] (347.1304,211.1298) .. controls
      (328.9557,207.4844) and (312.3394,196.5373) .. (301.7860,181.1483) .. controls
      (297.5330,174.9465) and (292.9087,164.5049) .. (291.0951,157.2570) .. controls
      (288.8254,148.1860) and (288.5663,132.6145) .. (290.8284,123.5786) .. controls
      (294.2724,109.8211) and (299.6586,99.3143) .. (309.8427,89.1302) .. controls
      (313.7942,85.1788) and (319.9156,80.6282) .. (323.0819,78.6673) .. controls
      (328.2539,75.4644) and (340.9419,70.0703) .. (342.8790,70.0420) .. controls
      (343.5125,70.0327) and (343.1405,72.2292) .. (343.1405,75.9063) .. controls
      (343.1405,83.5595) and (344.5578,98.7018) .. (346.5243,106.4736) .. controls
      (348.9893,116.2154) and (351.6795,122.3754) .. (357.0342,132.6558) --
      (361.7751,141.8159) -- (361.2069,176.8477) -- (361.5857,212.3215) --
      (356.4900,212.5070) .. controls (354.0360,212.5963) and (349.8733,211.8945) ..
      (347.3198,211.3824) -- (347.2567,211.3193) -- cycle;
    \path[fill=black,line join=round,line cap=round,miter limit=4.00,fill
      opacity=0.067,line width=0.500pt] (380.1455,201.1355) .. controls
      (379.4490,182.0367) and (375.8196,168.6767) .. (367.0542,151.9165) --
      (361.8657,141.9956) -- (361.7764,104.9047) -- (361.8210,67.9030) --
      (367.5799,68.0106) .. controls (377.6209,68.1982) and (386.6086,71.3550) ..
      (395.0921,75.8836) .. controls (413.7403,85.8384) and (428.0707,104.3041) ..
      (432.5011,125.7311) .. controls (434.3870,134.8518) and (433.8719,150.1977) ..
      (431.5336,158.7623) .. controls (427.2160,174.5766) and (417.3344,189.4178) ..
      (404.2034,199.0682) .. controls (398.5661,203.2111) and (392.7670,206.7615) ..
      (385.8329,208.4222) -- (380.4697,210.0266) -- cycle;
      
  \draw (162,226) node {Type 1};   
  \draw (361,226) node {Type 2};

\end{tikzpicture}
\caption{The different possibilities at a corner $y_j$, \\
The image of the disk $v$ is one of the four shaded areas }
\label{figure:FigureQ}
\end{figure}
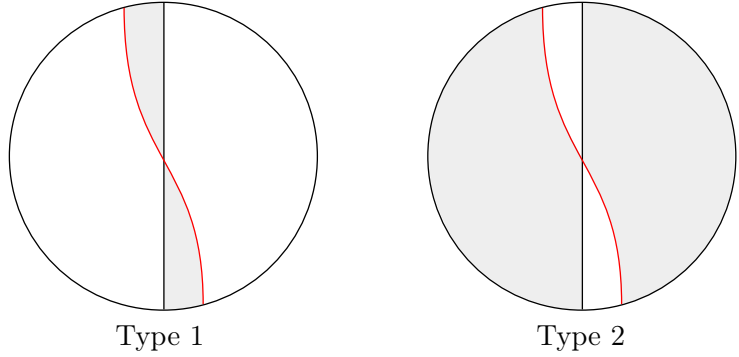
Then there is a unique switch of type $A\pm$ (Figure \ref{figure:FigureP}). The other corner is one of the points $y_i$ and must be of type $1 \pm$ (Figure \ref{figure:FigureQ}). It is then easy, following the procedure of the proof of Proposition \ref{prop:asurgeryisnonobstructed}, to produce a non-constant bigon with arcs on $\alpha$ and $\beta$ (see Figure \ref{figure:FigureR}).
  
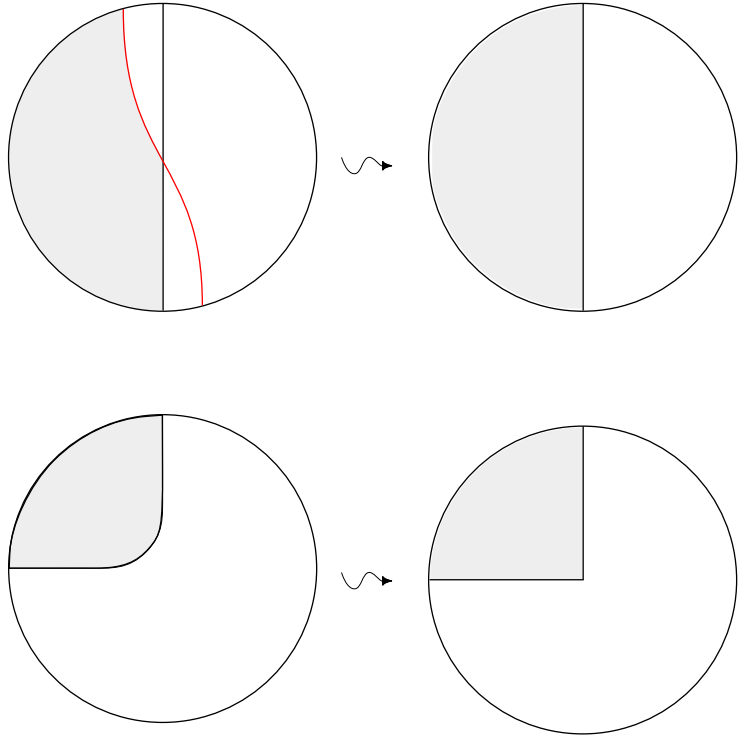
\begin{figure}
\definecolor{cff0000}{RGB}{255,0,0}
\captionsetup{justification=centering,margin=2cm}

\begin{tikzpicture}[y=0.80pt, x=0.80pt, yscale=-1.000000, xscale=1.000000, inner sep=0pt, outer sep=0pt]
  \begin{scope}[shift={(9.23985,0)}]
    \path[xscale=1.000,yscale=-1.000,draw=black,line join=round,line cap=round,miter
      limit=4.00,fill opacity=0.156,line width=0.500pt] (107.1429,-107.0717) circle
      (2.0474cm);
    \path[draw=black,line join=miter,line cap=butt,miter limit=4.00,even odd
      rule,line width=0.500pt] (107.3954,34.3407) -- (107.3954,179.5501);
    \path[draw=cff0000,line join=miter,line cap=butt,miter limit=4.00,even odd
      rule,line width=0.500pt] (125.9531,177.0270) .. controls (125.9531,107.7081)
      and (88.5649,111.1301) .. (88.5649,36.8307);
    \path[fill=black,line join=round,line cap=round,miter limit=4.00,fill
      opacity=0.067,line width=0.500pt] (92.7714,178.0070) .. controls
      (74.5967,174.3616) and (57.9804,163.4145) .. (47.4270,148.0255) .. controls
      (43.1740,141.8238) and (38.5497,131.3822) .. (36.7361,124.1342) .. controls
      (34.4664,115.0632) and (34.2073,99.4917) .. (36.4694,90.4558) .. controls
      (39.9134,76.6983) and (45.2996,66.1915) .. (55.4837,56.0074) .. controls
      (59.4352,52.0560) and (65.5566,47.5054) .. (68.7229,45.5445) .. controls
      (73.8949,42.3416) and (86.5829,36.9475) .. (88.5200,36.9192) .. controls
      (89.1535,36.9099) and (88.7815,39.1064) .. (88.7815,42.7835) .. controls
      (88.7815,50.4367) and (90.1988,65.5791) .. (92.1653,73.3508) .. controls
      (94.6303,83.0927) and (97.3205,89.2527) .. (102.6752,99.5330) --
      (107.4161,108.6931) -- (106.8479,143.7250) -- (107.2267,179.1988) --
      (102.1310,179.3842) .. controls (99.6770,179.4735) and (95.5143,178.7718) ..
      (92.9608,178.2596) -- (92.8977,178.1965) -- cycle;
  \end{scope}
  \begin{scope}[shift={(25.0,0)}]
    \path[xscale=1.000,yscale=-1.000,draw=black,line join=round,line cap=round,miter
      limit=4.00,fill opacity=0.156,line width=0.500pt] (290.0000,-107.0717) circle
      (2.0474cm);
    \path[draw=black,line join=miter,line cap=butt,miter limit=4.00,even odd
      rule,line width=0.500pt] (290.2525,34.3407) -- (290.2525,179.5501);
    \path[fill=black,line join=round,line cap=round,miter limit=4.00,fill
      opacity=0.067,line width=0.500pt] (275.5714,177.5387) .. controls
      (258.4429,173.8658) and (242.5410,163.5812) .. (231.9952,149.3557) .. controls
      (227.6653,143.5150) and (222.6762,133.1898) .. (220.4361,125.4336) .. controls
      (218.9752,120.3755) and (218.7903,118.2430) .. (218.8252,106.8622) .. controls
      (218.8617,94.9628) and (219.0101,93.5262) .. (220.8167,87.5765) .. controls
      (225.2405,73.0081) and (233.6113,60.3711) .. (244.8086,51.3571) .. controls
      (256.4723,41.9676) and (269.1049,36.9103) .. (284.6786,35.3956) --
      (289.8571,34.8919) -- (289.8571,106.7699) -- (289.8571,178.6479) --
      (285.0357,178.5911) .. controls (282.3839,178.5599) and (278.1250,178.0863) ..
      (275.5714,177.5387) -- cycle;
  \end{scope}
    \path[xscale=1.000,yscale=-1.000,draw=black,line join=round,line cap=round,miter
      limit=4.00,fill opacity=0.156,line width=0.500pt] (116.4286,-301.6479) circle
      (2.0474cm);
    \path[draw=black,line join=miter,line cap=butt,miter limit=4.00,even odd
      rule,line width=0.300pt] (116.4056,229.1896) .. controls (116.4056,229.1896)
      and (116.4056,253.2858) .. (116.4056,265.3340) .. controls (116.4056,277.3821)
      and (117.1366,285.2311) .. (108.5961,293.7716) .. controls (100.0556,302.3121)
      and (92.1895,301.4783) .. (80.0814,301.4783) .. controls (67.9734,301.4783)
      and (43.7572,301.4783) .. (43.7572,301.4783);
    \path[draw=black,fill=black,line join=round,line cap=round,miter limit=4.00,fill
      opacity=0.067,line width=0.500pt] (44.1673,293.9507) .. controls
      (44.6939,281.6329) and (53.0875,262.4466) .. (65.0779,250.4563) .. controls
      (77.0682,238.4659) and (92.7707,230.8661) .. (108.6616,229.6350) --
      (116.2644,229.0460) -- (116.3537,254.7838) .. controls (116.4503,282.6163) and
      (115.5492,285.5060) .. (110.0822,291.8930) .. controls (104.3155,298.6300) and
      (97.6181,301.4426) .. (84.9272,301.4544) .. controls (77.7766,301.4614) and
      (74.4341,301.4514) .. (73.1995,301.4514) -- (70.3266,301.4514) .. controls
      (69.3446,301.4514) and (67.1987,301.4641) .. (57.5692,301.4641) --
      (43.8462,301.4641) -- cycle;
  \path[xscale=1.000,yscale=-1.000,draw=black,line join=round,line cap=round,miter
    limit=4.00,fill opacity=0.156,line width=0.500pt] (315.0000,-307.0717) circle
    (2.0474cm);
  \path[draw=black,line join=miter,line cap=butt,miter limit=4.00,even odd
    rule,line width=0.500pt] (315.2525,234.3407) -- (315.2525,306.9162);
  \path[draw=black,line join=miter,line cap=butt,miter limit=4.00,even odd
    rule,line width=0.500pt] (315.5358,306.9454) -- (242.6478,306.9454);
  \path[fill=black,line join=round,line cap=round,miter limit=4.00,fill
    opacity=0.067,line width=0.500pt] (242.9999,304.3108) .. controls
    (242.9999,299.5694) and (244.2355,291.9814) .. (245.9146,286.4118) .. controls
    (252.9233,263.1633) and (271.0291,245.0397) .. (294.2818,237.9972) .. controls
    (299.9189,236.2898) and (307.5207,235.0537) .. (312.3832,235.0537) --
    (314.7208,235.0537) -- (314.7208,270.7879) -- (314.7208,306.5220) --
    (278.8603,306.5220) -- (242.9999,306.5220) -- (242.9999,304.3108) -- cycle;
  \begin{scope}[cm={{0.0,-1.0,1.0,0.0,(-43.15504,436.63133)}}]
    \path[draw=black,line join=miter,line cap=butt,miter limit=4.00,even odd
      rule,line width=0.268pt] (329.5622,244.1495) .. controls (329.5622,244.1495)
      and (321.3086,246.4926) .. (321.8351,250.5742) .. controls (322.2293,253.6295)
      and (328.9884,253.9786) .. (329.5622,256.9988) .. controls (329.9870,259.2352)
      and (325.5377,262.0745) .. (325.5377,263.2360) -- (325.5377,265.2080) --
      (325.4571,265.2080);
    \path[draw=black,fill=black,line join=round,line cap=round,miter limit=4.00,line
      width=0.268pt] (327.5796,263.6342) .. controls (326.4648,264.3477) and
      (326.0245,266.0290) .. (325.5784,267.8744) .. controls (325.1180,265.9693) and
      (324.6537,264.3232) .. (323.5772,263.6342) -- cycle;
    \path[draw=black,fill=black,line join=round,line cap=round,miter limit=4.00,line
      width=0.268pt] (327.5796,263.6342) .. controls (326.4648,264.3477) and
      (326.0245,266.0290) .. (325.5784,267.8744) .. controls (325.1180,265.9693) and
      (324.6537,264.3232) .. (323.5772,263.6342) -- cycle;
  \end{scope}
  \begin{scope}[cm={{0.0,-1.0,1.0,0.0,(-43.15504,633.0599)}}]
    \path[draw=black,line join=miter,line cap=butt,miter limit=4.00,even odd
      rule,line width=0.268pt] (329.5622,244.1495) .. controls (329.5622,244.1495)
      and (321.3086,246.4926) .. (321.8351,250.5742) .. controls (322.2293,253.6295)
      and (328.9884,253.9786) .. (329.5622,256.9988) .. controls (329.9870,259.2352)
      and (325.5377,262.0745) .. (325.5377,263.2360) -- (325.5377,265.2080) --
      (325.4571,265.2080);
    \path[draw=black,fill=black,line join=round,line cap=round,miter limit=4.00,line
      width=0.268pt] (327.5796,263.6342) .. controls (326.4648,264.3477) and
      (326.0245,266.0290) .. (325.5784,267.8744) .. controls (325.1180,265.9693) and
      (324.6537,264.3232) .. (323.5772,263.6342) -- cycle;
    \path[draw=black,fill=black,line join=round,line cap=round,miter limit=4.00,line
      width=0.268pt] (327.5796,263.6342) .. controls (326.4648,264.3477) and
      (326.0245,266.0290) .. (325.5784,267.8744) .. controls (325.1180,265.9693) and
      (324.6537,264.3232) .. (323.5772,263.6342) -- cycle;
  \end{scope}

\end{tikzpicture}
\caption{The procedure to obtain a bigon between $\alpha$ and $\beta$}
\label{figure:FigureR}
\end{figure}

\paragraph{\textbf{Second case}} One of the corners maps to a point $x_m^k$ for some $m$. Moreover, if we parameterize $v$ so that $v(-1) = x_k^m$, the upper arc is mapped to $c_k$ and the lower boundary arc to $\tilde{\alpha}_{k+1}$.

Following the upper boundary arc from $-1$ to $1$, there is a first switch of type $B \pm$ or $C \pm$. There cannot be another switch. Otherwise, the boundary condition $\lambda$ parameterizes one of the vertical arcs in the chart $\phi_k$. Hence, the other corner maps to one of the $y_i$ and must be of type $2$ (Figure \ref{figure:FigureQ}). From this, we deduce that the switch was of type $B \pm$.

In the chart near the switch, $\beta$ yields an arc in the image of $v$ from the handle to $\tilde{\alpha}_{k+1}$. Cutting the disk along this arc, we obtain a bigon between $\tilde{\alpha}_{k+1}$ and $\beta$.

\paragraph{\textbf{Third case}} Both of the corners map to points $y_i$ and $y_j$ for some $i$ and some $j$. We assume that the lower boundary arc of the bigon maps to $\tilde{\alpha}_{k+1}$ and the upper boundary arc maps to $c_k$.

Following the upper boundary arc from $-1$ to $1$, there is a first switch point of type $A$ and one second of type $B$ or $C$. However, the corner at $-1$ is of type $2$, and the corner at $1$ is of type $1$. So the second switch is of type $B$.

In the chart near the first switch point, there is an arc along $\beta$ from $c_k$ to $\tilde{\alpha}_{k+1}$ which cuts the image of $v$ in half. We cut $v$ along this arc and solve the corners as in Figure \ref{figure:FigureR} to obtain a bigon with boundary on $\alpha$ and $\beta$.

In all three cases, we obtain that $\alpha$ and $\beta$ are not in minimal position. Therefore, the lemma must hold. 
\end{proof}

Now, an induction and Proposition \ref{prop:asurgeryisnonobstructed} show that each of the curve $c_k$ is unobstructed. Recall that we denoted by $\gamma$ the curve obtained by the successive surgeries of the proof of Proposition \ref{prop:ActionofaDehntwistonGimm} and that it is isotopic to $T_\alpha(\beta)$.

Hence, by Proposition \ref{prop:cobordismassociatedtosuccessivesurgeries}, there is an unobstructed Lagrangian cobordism 
\[ \left( \alpha, \beta, \ldots, \beta, \beta^{-1}, \ldots, \beta^{-1} \right) \leadsto \gamma , \]
with as many copies of $\alpha$ as there are intersection points of degree $0$ and as many copies of $\alpha^{-1}$ as there are intersection points of degree 1. On the other hand, $\gamma$ and $T_\alpha (\beta)$ are isotopic curves. By Corollary \ref{coro:iisinjective} when $\beta$ is non-separating and Lemma \ref{lemma:Cylinderoverseparatingcurves} when $\beta$ is, there is $x \in \R$ such that
\[  \gamma = i(x) + T_\alpha (\beta).\]
Hence, in $\Gimmunob$,
\begin{align*} 
T_\alpha(\beta) & =  \gamma + i(x)\\
	& = \beta + (\alpha \cdot \beta) \alpha + i(x). 
\end{align*}
This concludes the proof.

In general, isotope $\alpha$ to a curve $\tilde{\alpha}$ in minimal position with $\alpha$. So there is $x \in \R$ such that
\[ T_{\tilde{\alpha}}(\beta) = \beta + (\alpha \cdot \beta) \tilde{\alpha} + i(x). \]
The conclusion follows by Lemma \ref{lemma:Cylinderoverseparatingcurves} since $T_{\tilde{\alpha}}(\beta)$ and $\tilde{\alpha}$ are isotopic to $\alpha$ and $T_{\alpha}(\beta)$ respectively.
\end{proof}

As a first consequence, let $\gamma$ be the oriented boundary of an embedded torus. We let
\begin{align}
  T = [\gamma] - i(\Hol_A(\gamma)).
\label{eqn:definitionofT}  	
\end{align}

\begin{lemma}
\label{lemma:theclass$T$iswelldefined}
The class $T$ defined in equation \ref{eqn:definitionofT} does not depend on the choice of $\gamma$.	
\end{lemma}

\begin{proof}
Let $\gamma_1$ and $\gamma_2$ be two embedded curves which bound a torus. There is a sequence of Dehn Twists $T_{\delta_1}, \ldots, T_{\delta_n}$ about the curves $\delta_1, \ldots, \delta_n$ such that
\[  T_{\delta_1} \ldots T_{\delta_n} (\gamma_1) \]
is isotopic to $\gamma_2$. Since $\gamma_1$ is null-homologous, by Proposition \ref{prop:ActionofaDehntwistonGimmunob}, there is $x \in \R$ such that
\[ T_{\delta_1} \ldots T_{\delta_n} (\gamma_1) = \gamma_1 + i(x). \]
Since $\gamma_2$ is isotopic to $T_{\delta_1} \ldots T_{\delta_n} (\gamma_1)$, by Lemma \ref{lemma:Cylinderoverseparatingcurves}, there is $y \in \R$ such that
\[ \gamma_2 = T_{\delta_1} \ldots T_{\delta_n} (\gamma_1) + i(y). \]
Hence,
\[ \gamma_2 = \gamma_1 + i(y + x). \]
We apply the holonomy morphism to obtain
\[ \Hol_A(\gamma_2) - \Hol_A(\gamma_1) = y +x . \]
So
\[ \gamma_2 - \Hol_A(\gamma_2) = \gamma_1 - \Hol_A(\gamma_1). \] 
\end{proof}

The following is the analog of Lemma \ref{lemma:classofanonseparatingcurve}.
\begin{lemma}
\label{lemma:classofanonseparatingcurveunob}
Let $\gamma$ be the oriented boundary of an embedded surface $S_1$. Then there is $x \in \R$ in $\Gimm$ such that
\[ 	[\gamma] = \chi(S_1) \cdot T + i(x). \]	
\end{lemma}

\begin{proof}
As in the proof of \ref{lemma:classofanonseparatingcurve}, we choose $\gamma_1$ and $\gamma_2$ such as in Figure \ref{figure:FigureE}. Now, let us call $c$ (resp. $\overline{c}$) the curve given by the successive surgeries on the left (resp. right) of Figure \ref{figure:FigureE}. There is an homeomorphism $h : S_g \to S_g$, isotopic to the identity, such that $h(\overline{c}) = c$.

In particular, there are curves $\tilde{\gamma}_1^{-1}$ and $\tilde{\gamma}^{-1}$ respectively isotopic to $\gamma_1^{-1}$ and $\gamma^{-1}$ such that the successive surgeries on the left of Figure \ref{figure:FigureE} produce the curve $c$.

Composing these cobordisms, we obtain an immersed cobordism
\[ V : \left (\gamma_2^{-1}, \alpha, \beta, \gamma_1^{-1}, \alpha^{-1}, \gamma^{-1}  \right) \leadsto \emptyset. \]
By Proposition \ref{prop:cobordismassociatedtosuccessivesurgeries}, there is a immersed unobstructed cobordism between these curves. Hence,
\[ - \gamma_2 + \alpha + \beta - \tilde{\gamma_1} - \tilde{\alpha} + \tilde{\gamma} = 0 . \]

The curves $\tilde{\gamma_1}^{-1}$, $\tilde{\gamma}^{-1}$ and $\tilde{\alpha}$ are embedded and isotopic to $\gamma_1^{-1}, \gamma^{-1}$ and $\alpha$ respectively. So there is $x \in \R$ such that
\[ - \tilde{\gamma_1} - \tilde{\alpha} + \tilde{\gamma} = i(x) + - \alpha + \gamma . \]

So there is $y \in \R$ such that
\[ \gamma_1 + \gamma_2 + \gamma = T + i(y) . \]
Now the proof follows by induction on the genus of the surface bounded by $\gamma$.
\end{proof}

\begin{lemma}
The class $T \in \Gimmunob$ defined in equation \ref{eqn:definitionofT} is of order $\chi(S_g)$. 	
\end{lemma}

\begin{proof}
As in the proof of Lemma \ref{lemma:subgroupgeneratedbyseparatingcurves}, consider a curve $\gamma$ which is the oriented boundary of a torus. By the definition of $T$, there is $x \in \R$ such that $\gamma = T + i(x)$. On the other hand, by lemma \ref{lemma:classofanonseparatingcurveunob}, there is $y \in \R$ such that $\gamma^{-1} = (-3 + 2g) \cdot T + i(y)$.

We conclude that $\chi(S_g) \gamma = i(y-x)$. Now, apply the holonomy morphism to both sides of this equation to obtain $y-x = 0$. 	
\end{proof}

\begin{lemma}
\label{lemma:classofgammaiunob}
Recall that in our notations, $\alpha_1, \ldots, \alpha_g$, $\beta_1, \ldots, \beta_g$ and $\gamma_1, \ldots, \gamma_{g-1}$ are the Lickorish generators represented in Figure \ref{figure:Lickorishgenerators}.

Let $i \in \{ 1, \ldots , g-1 \}$. Then, there is $x \in \R$ such that 
\[	[\gamma_i] = [ \alpha_{i+1} ] - [\alpha_i] -T + i(x) . \]	
\end{lemma}

\begin{proof}
This follows from the sequence of surgeries in Figure \ref{figure:FigureN} and from Proposition \ref{prop:cobordismassociatedtosuccessivesurgeries}. 	 
\end{proof}

\begin{proof}[Proof of Theorem \ref{Theo:Computationofunobgroup}]
It only remains to see that
\[ \Ker(\pi \oplus \mu) \subset \Im(i). \]
To see this, let $\gamma$ be a non-separating curve, there is a product of Dehn Twists about $\alpha_1, \ldots, \alpha_g$, $\beta_1, \ldots, \beta_g$, $\gamma_1, \ldots, \gamma_{g-1}$ which maps $\gamma$ to a curve isotopic to $\alpha_1$. Therefore, $\gamma$ belongs to the subgroup generated by the image of $i$ and the Lickorish generators. By Lemma \ref{lemma:classofgammaiunob}, this is the subgroup generated by $\alpha_1, \ldots, \beta_g$ and the image of $i$.

Hence, by Lemma \ref{lemma:classofanonseparatingcurveunob}, the group $\Gimmunob$ is generated by $T$, $\alpha_1, \ldots, \beta_g$ and the image of $i$.

Let $g = \sum_i n_i \alpha_i + \sum_j m_j \beta_j + i(x) + kT$ be a element of $\Ker(\pi \oplus \mu)$. Composing this by $\mu$, we get $k = 0 \mod \chi(S_g) $. Moreover, taking homology classes, the $n_i$ and $m_j$ are zero. Hence, $g$ is in the image of $i$.

Moreover, the holonomy map
\[ \Hol_A : \Gimmunob \to \R , \]
is a section of the map $i : \R \to \Gimmunob$. So the exact sequence is split.
\end{proof}

\begin{proof}[Proof of Theorem \ref{theo:theBiranCorneamapisanisomorphism}]
In \cite{ab07}, Abouzaid shows that the Maslov index and homology class induce well-defined map
\[ \pi : K_0(\DFuk(S_g)) \to H_1(S_g,\Z), \ \mu : K_0(\DFuk(S_g)) \to \Z / \chi(S_g) \Z. \]
Therefore, there is a commutative diagram 
\[
\begin{CD}
    0     @>>> \R @>>>       \Gimmunob     @>>>          H_1(S_g, \Z) \oplus \Z /  \chi(S_g) \Z      @>>>       0 \\
@.                     @.     @VVV                                                          @VV \Id V                                                       @. \\
     	     @.  @.	K_0(\DFuk(S_g))  		   @>>> 		H_1(S_g, \Z) \oplus \Z /  \chi(S_g) \Z @. 
\end{CD}.
\]
Assume that $x \in \Gimmunob$ maps to $0$ in $K_0(\DFuk(S_g))$. Then, from the above diagram, we have $\pi \oplus \mu (x) = 0$. So $x \in \Im(i)$. Composing with the holomomy map, we obtain that $x = 0$.	We conclude that the map \ref{eqn:Birancorneamap} is injective.
\end{proof}

\bibliographystyle{alpha}

\end{document}